\documentclass[11pt]{article}
\usepackage[colorlinks=true,linkcolor=blue,urlcolor=black,citecolor=blue,breaklinks]{hyperref}
\usepackage{color}
\usepackage{enumerate}
\usepackage{graphicx}
\usepackage{varwidth}
\usepackage{mathrsfs}
\usepackage{mathtools}
\usepackage[font=small,labelfont=bf]{caption}
\usepackage{subcaption}
\usepackage{overpic}
\usepackage{multirow}

\usepackage{natbib}
\bibpunct{(}{)}{;}{a}{,}{,}

\usepackage{fullpage}
\usepackage{authblk}
\usepackage{amsmath,amsthm,amssymb,colonequals,etoolbox}
\usepackage{thmtools}
\usepackage{url}
\usepackage{cleveref}

\renewcommand{\geq}{\geqslant}

\renewcommand{\leq}{\leqslant}
\renewcommand{\preceq}{\preccurlyeq}

\newtheorem{thm}{Theorem}[section]
\newtheorem{mydef}[thm]{Definition}
\newtheorem{myprop}[thm]{Proposition}
\newtheorem{mylemma}[thm]{Lemma}

\newtheorem{assmp}[thm]{Assumption}
\newtheorem{rmk}[thm]{Remark}
\newtheorem{obs}[thm]{Observation}
\newtheorem{ex}[thm]{Example}

\DeclareMathOperator*{\esssup}{ess\,sup}
\DeclareMathOperator*{\argmin}{arg\!\min}

\newcommand{\kernel}{\mathsf{K}}

\newcommand{\calL}{\mathcal{L}}

\newcommand{\R}{\ensuremath{\mathbb{R}}}

\newcommand{\norm}[1]{\lVert #1 \rVert}
\newcommand{\bignorm}[1]{\left\lVert #1 \right\rVert}

\newcommand{\bigip}[2]{\ensuremath{\left\langle #1, #2 \right\rangle}}
\newcommand{\ip}[2]{\ensuremath{\langle #1, #2 \rangle}}

\newcommand{\E}{\mathbb{E}}
\newcommand{\abs}[1]{\ensuremath{| #1 |}}
\newcommand{\bigabs}[1]{\ensuremath{\left| #1 \right|}}

\newcommand{\ind}{\mathbf{1}}

\renewcommand{\Pr}{\mathbb{P}}
\newcommand{\T}{\mathsf{T}}

\newcommand{\calA}{\mathcal{A}}

\newcommand{\calE}{\mathcal{E}}
\newcommand{\calX}{\mathcal{X}}

\newcommand{\calF}{\mathcal{F}}

\newcommand{\calH}{\mathcal{H}}

\numberwithin{equation}{section}

\newcommand{\calK}{\mathcal{K}}

\newcommand{\ha}{\hat{\alpha}}

\newcommand{\Symm}{\mathsf{Sym}}

\newcommand{\opnorm}[1]{\norm{#1}_{\mathrm{op}}}

\newcommand{\bregd}[2]{d_\psi\left(#1 \| #2\right)}
\newcommand{\bregdp}[2]{d_{\psi_p}\left(#1 \| #2\right)}
\newcommand{\bregdm}[2]{d_{\psi_m}\left(#1 \| #2\right)}

\title{Nonparametric adaptive control and prediction: theory and randomized algorithms}

\author[1]{Nicholas M.~Boffi%
\thanks{Authors contributed equally.}}
\author[2]{Stephen Tu$^*$}
\author[2,3]{Jean-Jacques~Slotine}

\affil[1]{Courant Institute of Mathematical Sciences, New York University}
\affil[2]{Google Brain Robotics}
\affil[3]{Nonlinear Systems Laboratory, Massachusetts Institute of Technology}

\begin{document}
\maketitle

\begin{abstract}
A key assumption in the theory of nonlinear adaptive control 
is that the uncertainty of the system can be expressed in the linear span
of a set of known basis functions.
While this assumption leads to efficient algorithms, it limits applications to very specific classes of systems.
We introduce a novel \textit{nonparametric} adaptive algorithm that estimates an infinite-dimensional density over parameters online to learn an unknown dynamics in a reproducing kernel Hilbert space. Surprisingly, the resulting control input admits an analytical expression that enables its implementation despite its underlying infinite-dimensional structure.
While this adaptive input is rich and expressive -- subsuming, for example, traditional linear parameterizations -- its computational complexity grows linearly with time, making it comparatively more expensive than its parametric counterparts. 
Leveraging the theory of random Fourier features, we provide an efficient randomized implementation that recovers the complexity of classical parametric methods while provably retaining the expressivity of the nonparametric input.
In particular, our explicit bounds only depend \emph{polynomially} on the underlying parameters of the system, allowing our proposed algorithms to efficiently scale 
to high-dimensional systems.
As an illustration of the method, we demonstrate the ability of the randomized approximation algorithm to learn a predictive model of a 60-dimensional system consisting of ten point masses interacting through Newtonian gravitation.
By reinterpretation as a gradient flow on a specific loss, we conclude with a natural extension of our kernel-based adaptive algorithms to deep neural networks. We show empirically that the extra expressivity afforded by deep representations can lead to improved performance at the expense of closed-loop stability that is both rigorously guaranteed and consistently observed for kernel machines.

\end{abstract}

\section{Introduction}
\label{sec:intro}

One of the fundamental assumptions of nonlinear adaptive systems theory is
that the uncertainty of the system can be written as a linear expansion in a set of known basis functions that are nonlinear in the system state.
While such linear parameterizations enable the derivation of efficient algorithms with provable guarantees, results outside of this restrictive regime are scarce. 
Notable examples typically leverage notions of monotonicity~\citep{astolfi03immersion,tyukin07adaptation} or convexity~\citep{annaswamy98adaptive,fradkov99} to make the underlying learning problem tractable. 

Here we broaden the applicability of adaptive control by relaxing this classical assumption. In statistical learning, nonlinear function approximation is handled through the use of reproducing kernel Hilbert spaces (RKHSs)~\citep{cucker02learning}, which are infinite-dimensional function spaces that admit tractable algorithms reminiscent of finite-dimensional linear regression. Inspired by this approach, we develop an adaptive input that learns directly over an RKHS without reference to a finite-dimensional vector of parameters.

One significant drawback of RKHSs is their computational cost. While the representer theorem ensures that estimation in an RKHS can always be cast as a finite expansion over the dataset, the number of parameters grows with its size, which makes learning on large datasets computationally demanding. A key breakthrough in overcoming this difficulty was the theory of random Fourier features, which shows that 
elements in many RKHSs can be approximated in the linear span of a finite set of \textit{random} basis functions with high probability. Remarkably, the number of
random basis functions needed can be shown to scale polynomially~\citep{rahimi08uniform} in the
function norm and the ambient dimension, which enables efficient computation even in
high-dimensional spaces. 

In the dynamical systems setting considered in this work, the system trajectory plays the role of the dataset, and the horizon plays the role of its size. Paralleling the statistical learning setting, the complexity of the nonparametric adaptive input that we introduce grows with this horizon. To overcome this complication, we leverage the theory of random features to provide high-probability guarantees on the possibility of uniformly approximating the nonparametric input via a finite-dimensional expansion in random basis functions. Importantly, this approach leads to efficient update laws that match the computational complexity of parametric methods while retaining the expressivity of the RKHS.

We focus on two primary problem settings. The first setting is the classical problem of adaptive control with matched uncertainty, where the uncertainty is assumed to live in the span of the control matrix. Our second application is in adaptive state estimation, where we seek to learn a model of an unknown dynamics governing the evolution of a particular state variable. As a byproduct of our analysis, we exhibit a duality between these two problems reminiscent of the duality between LQR and Kalman filtering in linear control theory. In both settings, we assume that the unmodelled component can be written as the sum of a term that can be linearly parameterized with known physically-motivated basis functions and a term assumed to live in an RKHS. This setup captures the practically relevant setting where a learner can leverage some available physical knowledge of the system but also must perform estimation in a purely unstructured fashion to achieve ideal performance. 

The paper is organized as follows. In Section~\ref{sec:related}, we review related work and summarize our contribution. In Section~\ref{sec:problem}, we formulate the adaptive control and prediction problems. In Section~\ref{sec:np_rslts}, we develop a theory of nonparametric adaptive control, building upon a simple observation reminiscent of the ``kernel trick'' in machine learning. In Section~\ref{sec:rf}, we review the theory of random Fourier features, which we subsequently apply in Section~\ref{sec:p_results} to design practical adaptive algorithms that asymptotically drive the control or prediction error to a ball around zero. The radius of the ball scales with the approximation error of the random feature expansion, and we give an explicit bound on the number of features needed to ensure that the tracking or prediction error falls below a tolerance threshold $\varepsilon$ with high probability.
In Section~\ref{sec:simulations}, we first study the performance of the nonparametric method in comparison to its randomized approximations on a synthetic adaptive control problem. We subsequently illustrate the effectiveness of its randomized approximations in very high dimension
by constructing an adaptive predictor for a $60$-dimensional
Hamiltonian dynamical system describing the motion of a collection of particles interacting through a $1/r^2$ potential.
\section{Related Work and Summary of Contributions}
\label{sec:related}

\paragraph{Uniform approximation for adaptive control} Most related to the present contribution is a line of work initiated by~\cite{kurdila13rkhs} and followed by~\cite{bobade18rkhs}, who study adaptive control and estimation in RKHSs. In these works, a nonparametric input is treated as an ideal, non-implementable abstraction, and this abstract input is approximated via orthogonal projections or a fixed grid of radial basis functions. Asymptotic convergence results are shown for the approximations,  but no finite-sample theory is given, and the grid of centers is chosen in an \textit{ad-hoc} fashion. By gridding the space, these past approaches essentially reduce to a classic line of work by~\cite{sanner_nn}, who approximate an unknown dynamics uniformly with a sum of radial basis functions. These basis functions are spaced on a regular grid, and the grid resolution is chosen based on considerations from sampling theory to ensure a sufficient degree of uniform approximation for the control application. Importantly, while these gridding-based approaches are suitable and highly efficient for low-dimensional systems, they become intractable for higher-dimensional systems. From the perspective of constructing a regular grid, ``low-dimensional'' is often as restrictive as four-dimensional, which is easily surpassed by modern control applications.

Another closely related line of work is \citet{chowdhary12adaptive,chowdhary2015gp}, who propose to use Gaussian Process (GP) regression for model reference adaptive control.
The primary difference with our work is that we derive a control
law (cf. \Cref{sec:np_rslts}) that operates purely in continuous-time, which obviates the need to take a time derivative of the error signal as supervision. This is important in practice, since it is well-known that computing the time derivative of a signal (e.g. with finite difference approximations) can amplify measurement noise.
An additional difference is that our theory quantifies the relation between
the number of random features in the function approximation (which governs its quality) and the size of the ball around the desired trajectory to which the system will converge.
Finally, our work relies on random feature approximations~\citep{rahimi07randomfeatures} for tractability,
which is simpler to implement in practice than approaches based on sparse GPs.

\paragraph{Randomization and dimensionality-dependence} We show that a nonparametric controller can be implemented as the action of a certain kernel integral operator against a known signal over the system trajectory, and we provide an intuitive derivation via the celebrated ``kernel trick''. This result naturally leads to the randomized approximation methods developed here, which can be seen as a stochastic alternative to a fixed grid of basis functions. The main advantage of randomization is computational: due to concentration of measure, the number of basis functions needed for our construction grows polynomially in the state and input dimension of the underlying control problem. This permits our method to scale to much higher-dimensional systems than prior methods based on gridding, which require a number of basis functions that grows \textit{exponentially} in dimension. Moreover, our work provides a natural path towards developing a theory of adaptive control with more expressive function classes such as single-layer neural networks~\citep{bach_breaking, bengio_convex}, as well as alternative approximation schemes such as the Nystr{\"o}m method~\citep{online_kernel}.

\paragraph{Random feature approximations} Our randomized algorithm is based on random Fourier features~\citep{rahimi07randomfeatures,rahimi08uniform,rahimi08kitchensinks} and their extension to vector-valued functions~\citep{brault16randomfeatures,minh16operator}. We build heavily on the results of~\cite{rahimi08uniform}, who prove that the $L_\infty$ approximation error
over a compact set for a function $f$ in an RKHS $\calH$ 
decays as $O(1/\sqrt{K})$, where $K$ is the number of features drawn 
from a particular distribution induced by $\calH$. This rate matches that due to~\cite{barron93universal} for approximation of functions whose gradients have absolutely integrable Fourier transforms via sums of sigmoidal basis functions.

\paragraph{Control and robotic learning} In control and robotics applications, several authors have utilized 
random features for function approximation in learning
stable vector fields~\citep{sindhwani18vectorfields},
control contraction metrics~\citep{singh20learning},
Lyapunov functions~\citep{boffi20learningcertificates},
and in velocity gradient-based adaptation~\citep{boffi21regret}.
However, these works do not analyze the effect of the
approximation error introduced by random features 
on the control performance, nor do they provide any bounds
on the number of random features needed to achieve a 
specified level of uniform approximation. Adaptive control laws have also been developed for robotic manipulators by exploiting the structure of the governing Euler-Lagrange equations~\citep{slot_li_robot}; it is straightforward to extend our results to this setting, or to augment existing robotic adaptive control laws with a nonparametric component to improve robustness to unmodeled disturbances.

\paragraph{Generality of results} While the focus of this work is on nonparametric adaptive control and randomized approximation schemes, we have written our results generally to capture a variety of different
settings in adaptive control, including Lyapunov-based adaptive control~\citep{krstic95adaptivebook}, speed/velocity gradient methods~\citep{fradkov99,krstic95adaptivebook}, 
mirror descent~\citep{boffi_neco_imp_reg}, and contraction metrics~\citep{brett_adapt}. We believe that this unification of results represents one of the most general treatments of nonlinear adaptive control available in the literature, and see it to be of independent interest.
\section{Problem Formulation}
\label{sec:problem}
\paragraph{Adaptive control} We study nonlinear dynamical systems in \textit{matched uncertainty form}
\begin{equation}
    \label{eqn:gen_dyn}
    \dot{x} = f(x, t) + g(x, t)\left(u(x, t) - Y(x, t) \alpha_p - h(x)\right),
\end{equation}
where $f:\R^n\times\R_{\geq 0} \rightarrow \R^n$ is the ``nominal dynamics'' representing the behavior of the system in the absence of any inputs, $g : \R^n\times\R_{\geq 0}\rightarrow \R^{n\times d}$ is the control matrix describing how an input enters the system, $u:\R^n\times\R_{\geq 0}\rightarrow \R^d$ is the control input chosen by the learner, 
$Y : \R^n \times \R_{\geq 0} \rightarrow \R^{d \times p}$
is a matrix of basis functions describing the system's physical structure, $\alpha_p \in \R^p$ is a corresponding vector of physical parameters,
and $h \in \calH$ is an unknown dynamics in an operator-valued RKHS $\calH$ of functions mapping $\R^n\mapsto\R^d$~\citep{carmeli08operator}\footnote{A formal definition of an operator-valued RKHS will be provided in Section~\ref{sec:np_rslts}.}. Both $h$ and $\alpha_p$ are unknown, and the goal is to drive $x(t)$ to a bounded desired trajectory $x_d(t)$ by learning a suitable input $u(x, t)$ online. As a supervisory signal, the learner observes an error $e(t) \in \R^s$ at each $t$ with dynamics
\begin{align}
    \dot{e} = f_e(e, t) + g_e(x, t) (u(x, t) - Y(x, t) \alpha_p - h(x)), \label{eq:error_dynamics}
\end{align}
where $f_e : \R^s \times \R_{\geq 0} \rightarrow \R^s$ and $g_e : \R^n \times \R_{\geq 0} \rightarrow \R^{s \times d}$.
While the most natural
error signal is the trajectory tracking
error $e(t) = x(t) - x_d(t)$,
we formulate the error signal 
more abstractly to allow for controllers that
only actuate higher order derivatives of the state.
This is discussed more in \Cref{ex:controllable_lti}.

\begin{rmk}
Our formulation with $h$ autonomous can be relaxed by considering an RKHS of functions mapping $\R^{n+1} \mapsto \R^{d}$, i.e., by treating time explicitly as an input variable.
\end{rmk}

\paragraph{Adaptive prediction} We study nonlinear dynamical systems that can be additively decomposed 
\begin{equation*}
    \dot{x} = f(x, t) = Y(x, t)\alpha_p + h(x),
\end{equation*}
where $f:\R^n\times\R_{\geq 0}\rightarrow\R^n$ is an unknown dynamics 
composed of terms that have a similar interpretation to the control setting. The goal is to learn an approximation $\hat{f}: \R^n\times \R_{\geq 0}\rightarrow \R^n$ of the true dynamics $f$ by designing an estimator
\begin{equation}
    \label{eqn:dyn_predict}
    \dot{\hat{x}} = \hat{f}(\hat{x}, t) + k(\hat{x}, x(t))
\end{equation}
that will ensure $\hat{x}(t)$ asymptotically approaches $x(t)$. In \eqref{eqn:dyn_predict}, $k:\R^n\times\R^n\rightarrow\R^n$ is a feedback term satisfying $k(x, x) = 0$ for all $x$ that is used to ensure $\hat{x}(t)$ remains close to $x(t)$ during learning. In this setting, the error signal can be taken as the prediction error $e(t) = \hat{x}(t) - x(t)$. Moreover, the estimator state $\hat{x}(t)$ plays the role of $x(t)$ from the control setting, while $x(t)$ plays the role of the desired trajectory $x_d(t)$. When measurements are no longer available, the open-loop system $\dot{\hat{x}} = \hat{f}(\hat{x}, t)$ may be used to extrapolate the state to make predictions into the future. If measurements are only available at some discrete sampling frequency, the dynamics $\hat{f}$ can be used to interpolate the value of the state between sampling points. This discrete setting is expanded upon in \Cref{app:sample}.

\subsection{Notation} 
We consider algorithms that update estimates of the physical parameters
$\hat{\alpha}_p \in O_p \subseteq \R^p$ and model parameters (when applicable) $\hat{\alpha}_m \in O_m \subseteq \R^m$ online, where $O_p$ and $O_m$
are open convex subsets.
We fix twice differentiable mirror maps\footnote{See e.g. \cite[Section 4.1]{bubeck_fnt_book} for a definition.} (potential functions) $\psi_p : O_p \rightarrow \R$ (resp. $\psi_m : O_m \rightarrow \R$)
that are strongly convex with respect to a norm $\norm{\cdot}$ on $O_p$ 
(resp. $\norm{\cdot}'$ on $O_m$)
and have locally Lipschitz Hessians.
For a potential $\psi$, we let $\bregd{\alpha}{\ha} = \psi(\alpha) - \psi(\ha) - \nabla \psi(\ha)^\T(\ha - \alpha)$ denote the Bregman divergence associated with $\psi$. We use $\norm{\cdot}_2$ to denote the $\ell_2$ norm,
$\opnorm{\cdot}$ to denote the $\ell_2\rightarrow \ell_2$ operator norm of a matrix,
$B_2^n(R)$ to denote the closed $\ell_2$ ball of radius $R$ in $\R^n$, 
$\mathbb{S}^{n-1}$ to denote the unit sphere in $\R^n$, $\R_{\geq 0}$ to denote the non-negative reals, and 
$\Symm_{\geq 0}^{n \times n}$ to denote the set of symmetric positive semidefinite $n \times n$ matrices.
More generally, for a normed vector space $E$, $\norm{\cdot}_E$ denotes
its norm, and $B_E(R)$ denotes a closed ball in $E$ of radius $R$.
For a measure $\nu$, measurable space $\Theta$, and positive integer $q$,
the space $L_2^q(\Theta, \nu)$ denotes the real Hilbert space of
square integrable measurable functions $f : \Theta \rightarrow \R^q$ with norm
$\norm{f}_{L_2^q(\Theta, \nu)}^2 = \int_\Theta \norm{f(\theta)}_2^2 d\nu(\theta)$.
We will often drop the dependence on $q$ when it is clear from the context. Finally, for a positive definite metric $M : \R^n \rightarrow \Symm_{\geq 0}^{n \times n}$,
the Riemannian energy $E_M : \R^n \times \R^n \rightarrow \R_{\geq 0}$ 
is defined as:
\begin{align*}
    E_M(x, y) := \inf_{\gamma} \int_0^1 \gamma_s(s)^\T M(\gamma(s)) \gamma_s(s) ds, \:\: \gamma_s(s) = \frac{d \gamma}{ds}(s),
\end{align*}
where the infimum ranges over smooth curves $\gamma$ satisfying
$\gamma(0) = x$ and $\gamma(1) = y$.

\subsection{Assumptions} 
To make the above learning problems tractable and to simplify our presentation of results, we require some standard definitions and assumptions. The first requirement is regularity of the nominal dynamics, control matrix, and basis functions.

\begin{mydef}
\label{def:local_lip_local_bound}
Let $E_1$ and $E_2$ be normed vector spaces. 
A function $f(x, t)$ mapping $ E_1 \times \R_{\geq 0} \mapsto E_2$ is said to be locally Lipschitz in $x$ if for every finite $T > 0$ and $R > 0$, 
\begin{equation*}
    \sup_{t \in [0, T]} \sup_{\substack{\norm{x}_{E_1} \leq R,\\ \norm{y}_{E_1} \leq R, \\ x \neq y}} \frac{ \norm{f(x, t) - f(y, t)}_{E_2}}{\norm{x-y}_{E_1}} < \infty.
\end{equation*}
Furthermore, $f$ is said to be locally bounded in $x$ uniformly in $t$
if for every finite $R > 0$, 
\begin{equation*}
    \sup_{t \in \R_{\geq 0}} \sup_{\norm{x}_{E_1} \leq R} \norm{f(x, t)}_{E_2} < \infty.
\end{equation*}
\end{mydef}
\begin{assmp}[Dynamics regularity]
\label{assmp:basic}
The functions $f$, $g$, and $Y$ are known to the learner. Moreover, $f$, $g$, $Y$, and $h$ are locally Lipschitz in $x$ and locally bounded in $x$ uniformly in $t$.
\end{assmp}
Our second requirement is a set of reasonable conditions on the error to ensure it provides a suitable signal for learning.
\begin{assmp}[Error regularity]
\label{assmp:error}
$f_e$ and $g_e$ are locally Lipschitz in their first argument and locally bounded in their first argument uniformly in $t$. Moreover, the following three conditions hold:
\begin{enumerate}[(i)]
    \item In the absence of the unknown dynamics and any input, zero error is a fixed point,
\begin{align}
    \label{eqn:error_fixed_pint}
    f_e(0, t) &= 0 \:\: \text{for all}\:\: t \geq 0.
\end{align}
    \item Bounded error implies a bounded deviation from the desired trajectory,
\begin{align}
    \label{eq:bounded_error_bounded_x}
    \sup_{t \in [0, T]} \norm{e(t)}_2 < \infty &\text{ implies } \sup_{t \in [0, T]} \norm{x(t)-x_d(t)}_2 < \infty \:\: \text{for all}\:\: T > 0.
\end{align}
    \item A convergent error signal implies a convergent trajectory
\begin{align}
    \label{eq:error_signal_asymp}
    \lim_{t \rightarrow \infty} \norm{e(t)}_2 = 0 &\text{ implies } \lim_{t \rightarrow \infty} \norm{x(t)-x_d(t)}_2 = 0. 
\end{align}
\end{enumerate}
\end{assmp}
To demonstrate that such error signals can be constructed in practice, we provide a few simple illustrative examples.

\begin{ex}[Systems with regularity] 
Consider a system satisfying Assumption~\ref{assmp:basic}. Then $e(t) = x(t) - x_d(t)$ satisfies the requirements in Assumption~\ref{assmp:error}.
\end{ex}

\begin{ex}[Controllable linear time-invariant systems]
\label{ex:controllable_lti}
Consider the linear time-invariant system $f(x, t) = Ax$ and $g(x, t) = B$ with the pair $(A, B)$ controllable. Let $z(t) \in \R^n$ denote the state of the system expressed in control canonical form, and let $z_d(t)\in\R^n$ denote the corresponding desired trajectory. Define $e(t) = H(s)\left(z_1(t) - z_{d, 1}(t)\right)$ where $H(s)$ is a stable transfer function with at most $n-1$ poles and $z_i(t)$ denotes the $i^{\text{th}}$ component of $z$. Then $e(t)$ satisfies the requirements of Assumption~\ref{assmp:error}.
\end{ex}

The following stability assumption on the error model is key to our analysis. This assumption is equivalent to requiring that in the absence of the unknown dynamics and adaptive input, the system will nominally tend to the desired trajectory.
\begin{assmp}[Lyapunov stability of the error]
\label{assmp:lyapunov}
The error system \eqref{eq:error_dynamics} admits a continuously differentiable Lyapunov function $Q: \R^s\times \R_{\geq 0} \rightarrow \R$ satisfying for every $e \in \R^s$ and $t \geq 0$,
\begin{enumerate}[(i)]
\item $\nabla Q(e, t)$ and $\frac{\partial Q}{\partial t}(e, t)$ are 
locally bounded in $e$ uniformly in $t$,
\item $\nabla Q(e, t)$ is locally Lipschitz in $e$,
\item $\ip{\nabla Q(e, t)}{f_e(e, t)} + \frac{\partial Q}{\partial t}(e, t) \leq -\rho(\norm{e}_2)$, and
\item $\mu_1(\norm{e}_2) \leq Q(e, t) \leq \mu_2(\norm{e}_2)$,
\end{enumerate}
where $\rho, \mu_1,$ and $\mu_2$ denote class-$\calK_\infty$ functions.
\end{assmp}%

While we focus on Lyapunov stability of the error dynamics, our results encompass incremental forms of stability such as contraction~\citep{lohmiller98contraction}. 

\begin{rmk}[Contraction]
\label{rmk:lyap_contr}
We say that the error system is contracting in a metric $M:\R^s\times\R_{\geq 0} \rightarrow \Symm^{s\times s}_{\geq 0}$ if for some $\lambda > 0$,
\begin{align}
\frac{\partial f_e}{\partial e}(e, t)^\T M(e, t) + M(e, t) \frac{\partial f_e}{\partial e}(e, t) + \dot{M}(e, t) \preccurlyeq -2\lambda M(e, t), \:\: \forall e \in \R^s, t \in \R_{\geq 0}.
\end{align}
Taking the first variation of the Riemannian energy between the error $e$ and the zero trajectory $Q(e, t) = E_{M(\cdot, t)}(e, 0)$ shows that $\ip{\nabla Q(e, t)}{f_e(e, t)} + \frac{\partial Q}{\partial t}(e, t) \leq -2\lambda Q(e, t)$, so that the energy serves as an exponentially stable Lyapunov function. This correspondence will be used in the prediction setting with $e(t) = \hat{x}(t) - x(t)$.
\end{rmk}

\section{Nonparametric adaptive control and prediction}
\label{sec:np_rslts}
In this section, we present our primary result in the nonparametric setting.  Given a Lyapunov function for the error dynamics as stated in Assumption~\ref{assmp:lyapunov}, the standard procedure in adaptive nonlinear control is to approximate the unknown dynamics $h(x)$ appearing in \eqref{eqn:gen_dyn}~\&~\eqref{eq:error_dynamics} by an expansion in known basis functions $\Phi:\R^n\rightarrow\R^{d\times p}$~\citep{sanner_nn}
\begin{equation}
    \label{eqn:parametric_approx}
    \hat{h}(x,t) = \Phi(x)\hat{\alpha}(t),
\end{equation}
and to update the parameter estimates $\hat{\alpha}(t)\in\R^p$ according to a Lyapunov-based update law
\begin{equation}
    \label{eqn:linear_update_standard}
    \dot{\hat{\alpha}}(t) = -\gamma \Phi(x)^\T g_e(x, t)^\T \nabla Q (e, t),
\end{equation}
for $\gamma > 0$ a learning rate. 

\subsection{Nonparametric form} 
We start with the following simple observation about the construction in \eqref{eqn:parametric_approx}~\&~\eqref{eqn:linear_update_standard}, which is analogous to the ``kernel trick'' in machine learning.
\begin{obs}[Kernel trick]
\label{obs:kernel}
Assume $\hat{\alpha}(0) = 0$\footnote{Note that this is without loss of generality, since any non-zero $\hat{\alpha}(0)$ results in a non-zero
$\hat{h}(\cdot, 0)$ which can simply be absorbed into $h$.}. Then the adaptive approximation \eqref{eqn:parametric_approx} with parameters updated according to the algorithm \eqref{eqn:linear_update_standard} is equivalent to the nonparametric approximation
\begin{equation}
    \label{eqn:kernel_input}
    \hat{h}(x, t) = \int_0^t \kernel(x, x(\tau)) c(\tau)d\tau,
\end{equation}
where we have defined the kernel function $\kernel:\R^n\times\R^n\rightarrow\R^{d\times d}$ and coefficients $c(t)\in\R^d$ as:
\begin{align*}
    \kernel(x, y) &= \Phi(x)\Phi(y)^\T,\\
    c(t) &= -\gamma g_e(x(t), t)^\T \nabla Q(e(t), t).
\end{align*}
\end{obs}
The proof is simple and proceeds by formally writing the solution of \eqref{eqn:linear_update_standard} as an integral over time. Observation~\ref{obs:kernel} demonstrates that the function estimates formed by classical adaptive control algorithms only depend on inner products between the basis functions and do not, in principle, require any reference to a vector of parameter estimates. This implies that the basis functions need not be finite-dimensional so long as they admit a computationally inexpensive procedure for computing their inner products, which is precisely the case for an RKHS. 

\paragraph{Data-adapted centers}  Restricting to the case where $\kernel(\cdot, \cdot)$ is the Gaussian kernel, \eqref{eqn:kernel_input} can be seen as leaving a ``trail'' of Gaussians along the system trajectory $x(\tau)$ for $\tau < t$. In this sense, similar to kernel machines in statistical learning, \eqref{eqn:kernel_input} automatically constructs data-adapted centers at which to place spatially-localized basis functions. 

\paragraph{Complexity} The price paid for the expressivity in the representation \eqref{eqn:kernel_input} is that $\hat{h}(x, t)$ now obeys a partial differential equation that must be solved over a horizon of length $t$ at each $x \in \R^n$,
\begin{equation}
    \label{eqn:pde}
    \frac{\partial \hat{h}}{\partial t}(x, t) = \kernel(x, x(t))c(t).
\end{equation}
While \eqref{eqn:pde} is decoupled in space so that a global solve is not required, past work from time $\tau < t$ cannot be re-used at time $t$. Hence, unlike standard parametric methods that incur an $\mathcal{O}(1)$ cost at each timestep, solving \eqref{eqn:pde} for the value of $\hat{h}(x, t)$ at a given spatial location $x$ incurs an $\mathcal{O}(t)$ cost at each time $t$. For most applications, this is prohibitively expensive, and we now turn to efficient approximation schemes that circumvent this difficulty.

\subsection{Random feature space}
Observation~\ref{obs:kernel} motivates us to work with function classes
described by kernels. The following definition
introduces the notion of an operator-valued kernel.
\begin{mydef}[Operator-valued reproducing kernel, see e.g.,~\cite{carmeli08operator}]
\label{def:op_kernel}
A kernel $\kernel : \R^{n} \times \R^{n} \rightarrow \R^{d \times d}$ is said to be an operator-valued reproducing
kernel for an RKHS $\calH$ if 
\begin{enumerate}[(i)]
    \item For every $\{x_i\}_{i=1}^{N} \subseteq \R^n$ and $\{w_i\}_{i=1}^{N} \subseteq \R^d$, it holds 
that $\sum_{i,j=1}^{N} \ip{w_i}{\kernel(x_i, x_j) w_j} \geq 0$.
    \item $\kernel(\cdot, x)w \in \calH$ for every $x\in \R^n$ and $w\in\R^d$.
    \item $\calH$ can be written
    $\calH = \mathsf{cl}\left\{f \,\Bigg|\, \exists\: \{x_i\}_{i=1}^{n}, \{w_i\}_{i=1}^{n}\: s.t.\: f(\cdot) = \sum_{i=1}^n\kernel(\cdot, x_i)\omega_i \right\}$.
\end{enumerate}
\end{mydef} 
The adaptive algorithms we formulate will be valid for any RKHS $\calH$ with a known operator-valued kernel $\kernel$. However, we focus on RKHSs with specific structure that will enable the design of efficient randomized approximations. These function spaces are described by the following assumption.
\begin{assmp}[The function class $\mathcal{F}_2$, see e.g.,~\cite{bach_breaking}]
\label{assmp:rf_kernel}
The unknown dynamics $h$ lies in an RKHS $\calH$ with known operator-valued kernel $\kernel$. Moreover, $\kernel$ may be written in terms of a feature map $\Phi:\R^n\times\Theta\rightarrow\R^{d\times d_1}$ as
\begin{align}
    \kernel(x, y) = \int_{\Theta} \Phi(x, \theta) \Phi(y, \theta)^\T d\nu(\theta), \label{eq:kernel_form}
\end{align}
with $d_1\leq d$ and where $\nu$ is a known probability measure on a measurable space $\Theta$.
\end{assmp}
In Assumption~\ref{assmp:rf_kernel}, we have overloaded the definition of the feature map $\Phi$ as a generalization of the structure of $\kernel$ seen in Observation~\ref{obs:kernel}. Assumption~\ref{assmp:rf_kernel} is not very restrictive, as many rich kernels applied in practice -- such as the Gaussian and Laplace kernels -- can readily written in this form. In particular, the operator-valued generalization of Bochner's theorem~\citep{brault16randomfeatures} states that any translation-invariant kernel can be written in the form \eqref{eq:kernel_form} with a feature map
\begin{equation}
    \label{eqn:bochner}
    \Phi(x, \theta) = B(w) \cos(w^\T x + b),
\end{equation}
where $\Theta \subseteq\R^{n+1}$, $\theta = (w, b)$, $w\in\R^n$, $b\in\R$, and for suitable choices of $\nu$ and $B : \R^n\rightarrow \R^{d\times d_1}$.

Under Assumption~\ref{assmp:rf_kernel}, it is well-known (c.f.~\cite{bach_breaking}, Appendix A) that $h \in \calH$ can be written, for some square-integrable signed density $\alpha:\Theta\rightarrow\R^{d_1}$ with respect to the base measure $\nu$, as the integral
\begin{equation}
    \label{eqn:f2_func}
    h(\cdot) = \int_{\Theta}\Phi(\cdot, \theta)\alpha(\theta)d\nu(\theta), \:\: \norm{h}^2_{\calH} = \norm{\alpha}^2_{L_2(\Theta, \nu)}.
\end{equation}
The corresponding Hilbert space is referred to as $\mathcal{F}_2$~\citep{bach_breaking, bengio_convex}. $\calF_2$ is related to the Banach space of single-layer neural networks $\mathcal{F}_1$, which may be obtained by taking the union over all possible base measures for $\calF_2$. The space $\calF_2$ is convenient for our purposes because it allows us to treat the infinite-dimensional density over parameters $\alpha$ similar to a standard finite-dimensional vector of parameters. To do so, we introduce a second moment regularity condition that will ensure the nonparametric input leads to a stable and convergent feedback system.
\begin{assmp}[Second moment regularity of $\Phi$]
\label{assmp:second_moment_bound}
For every $x \in \R^n$, the second moment of the feature matrix is finite, i.e., $\int_\Theta \opnorm{\Phi(x, \theta)}^2 d\nu(\theta) < \infty$.
Furthermore, for every $R > 0$,
\begin{align*}
    \sup_{\substack{\norm{x}_2 \leq R, \norm{y}_2 \leq R,\\ x \neq y}} \frac{\left(\int_\Theta \opnorm{\Phi(x, \theta) - \Phi(y, \theta)}^2 d\nu(\theta)\right)^{1/2}}{\norm{x-y}_2} < \infty.
\end{align*}
\end{assmp}
To obtain accurate parametric approximations, we may sample points $\theta_i \in \Theta$ i.i.d. from the base measure $\nu$. This has the effect of discretizing the density $\alpha$ into a vector of parameters that can be learned using standard adaptive methods.

\subsection{Main results}
For simplicity of exposition, we restrict to the case where $\alpha_p = 0$ in \eqref{eqn:gen_dyn} to focus on convergence of the nonparametric input; the randomized approximations in Section~\ref{sec:p_results} will adapt over both physical and mathematical parameter estimates simultaneously. Moreover, we focus here on the setting of adaptive control. Later, the proof of Theorem~\ref{thm:dp_finite_approx} will highlight a duality between adaptive control and adaptive prediction that immediately implies an analogous result for prediction. 
The following theorem demonstrates that the nonparametric adaptation algorithm leads to a stable and convergent trajectory.
\begin{restatable}[Convergence]{mythm}{npconv}
\label{thm:nonparametric_conv}
Consider system \eqref{eqn:gen_dyn} under Assumptions~\ref{assmp:lyapunov},~\ref{assmp:rf_kernel}, and~\ref{assmp:second_moment_bound}. Fix $\alpha_p = 0$ and let $\gamma > 0$.
Then the adaptive control input
\begin{align*}
    u(x, t) &= - \gamma\int_0^t \kernel(x, x(\tau))g_e(x(\tau), \tau)^\T \nabla Q(e(\tau), \tau)d\tau
\end{align*}
ensures that both $x(t)$ and $e(t)$ exist and are uniformly bounded for all $t \geq 0$. Moreover, $u(\cdot, t) \in \calH$ for all $t \geq 0$ and $\lim_{t\rightarrow \infty} \norm{x(t) - x_d(t)}_2 = 0$.
\end{restatable}
Next, we study the interpolation properties of the input $u(x, t)$ along the desired trajectory. To this end, we 
strengthen Definition~\ref{def:local_lip_local_bound} to be uniform in $t$.
\begin{mydef}
\label{def:local_lipschitz_x_uniform_t}
Let $E_1$ and $E_2$ be normed vector spaces. 
A function $f(x, t)$ mapping $E_1 \times \R_{\geq 0} \mapsto E_2$
is said to be locally Lipschitz in $t$ uniformly in $x$ if the following
two conditions hold for every $R > 0$:
\begin{align*}
    \sup_{\norm{x}_{E_1}\leq R} \sup_{\substack{t_1, t_2 \in \R_{\geq 0}, \\t_1 \neq t_2}} \frac{\norm{f(x, t_1) - f(x, t_2)}_{E_2}}{\abs{t_1-t_2}} &< \infty, \\
    \sup_{t \in \R_{\geq 0}} \sup_{\substack{\norm{x_1}_{E_1} \leq R,\\ \norm{x_2}_{E_1} \leq R, \\x_1 \neq x_2}} \frac{\norm{f(x_1, t) - f(x_2, t)}_{E_2}}{\norm{x_1 - x_2}_{E_1}} &< \infty.
\end{align*}
\end{mydef}
With Definition~\ref{def:local_lipschitz_x_uniform_t} in hand, we may state the following theorem.
\begin{restatable}[Interpolation]{thm}{interpolation}
\label{thm:interpolation}
Consider the setting of Theorem~\ref{thm:nonparametric_conv}. Suppose furthermore that both $f_e(e, t)$ and $g_e(x, t)$ are locally Lipschitz in their first argument
uniformly in $t$.
Finally, suppose that for every $R > 0$,
\begin{align*}
    \int_\Theta \sup_{\norm{x}_2 \leq R} \opnorm{\Phi(x, \theta)}^2 d\nu(\theta) < \infty.
\end{align*}
Then the nonparametric input asymptotically interpolates the unknown in the span of the control matrix, $\lim_{t \rightarrow \infty} \norm{g_e(x(t), t)(u(x(t), t) - h(x(t)))}_2 = 0$.
\end{restatable}
Mirroring the finite-dimensional setting considered by~\citet{boffi_neco_imp_reg}, we now demonstrate that the adaptive input in Theorem~\ref{thm:nonparametric_conv} converges to the minimum RKHS-norm interpolating solution.
\begin{restatable}[Implicit regularization]{thm}{impreg}
\label{thm:imp_reg}
Consider the setting of Theorem~\ref{thm:nonparametric_conv}. Define the interpolating set over the trajectory
\begin{equation*}
    \mathcal{A} := \left\{\bar{h} \in \calH : \bar{h}(x(t)) = h(x(t)), \    \forall t\geq 0\right\},
\end{equation*}
and assume that $\lim_{t\rightarrow\infty}u(\cdot, t)\in \mathcal{A}$. Then,
\begin{equation}
   \lim_{t\rightarrow\infty}u(\cdot, t) \in \argmin_{\bar{h}\in\mathcal{A}}\:\norm{\bar{h}(\cdot)}_{\calH}.
\end{equation}
\end{restatable}
Given these results for the computationally expensive nonparametric input, we now turn to develop a theory of efficient randomized approximation schemes.
\section{Random feature approximation}
\label{sec:rf}
\subsection{Approximation theory}
We now demonstrate how the function space $\calF_2$ leads to efficient randomized approximation algorithms. These randomized algorithms will enable us to restore the computational advantages of classical finite-dimensional parametric approximations while retaining the expressiveness of the RKHS $\calF_2$ with high probability. Roughly speaking, the approach will be to apply the law of large numbers to the expectation \eqref{eqn:f2_func}, which leads to a finite-dimensional approximation
\begin{equation*}
    h(\cdot) \approx \frac{1}{K}\sum_{i=1}^K \Phi(\cdot, \theta_i)\alpha_i,
\end{equation*}
where the $\theta_i \sim \nu$ are drawn i.i.d.\ from the base measure $\nu$ and the $\alpha_i = \alpha(\theta_i) \in \R^{d_1}$ are treated as parameters to be learned. $K$ denotes the number of sampling points and will tune the accuracy of the approximation. We provide a bound on the number of random features $K$ needed to ensure that there exists a set of weights $\{\alpha_i\}$ capable of $\varepsilon$-uniformly approximating $h$ on a fixed compact set $X\subset\R^n$. To begin, let $B_\Phi(\delta)$ be any function that satisfies,
for any $\delta \in (0, 1)$,
\begin{align*}
    \Pr_{\theta \sim \nu}\left( \sup_{x \in X} \opnorm{\Phi(x, \theta)} > B_\Phi(\delta) \right) \leq \delta.
\end{align*}
Then, for any $\eta \in (0, 1)$, define a truncated version of $\Phi$ as
\begin{align*}
    \Phi_\eta(x, \theta) := \Phi(x, \theta) \ind\left\{ \opnorm{\Phi(x, \theta)} \leq B_\Phi(\eta) \right\}.
\end{align*}
We will be interested in approximating functions over the subset
\begin{equation*}
    \calF_2(B) = \left\{f(\cdot) = \int_{\Theta}\Phi(\cdot, \theta)\alpha(\theta)d\nu(\theta) \,\Bigg|\, \esssup_{\theta\in\Theta}\norm{\alpha(\theta)}_2 \leq B\right\} \subset \calF_2,
\end{equation*}
which is dense in $\calF_2$ as $B\rightarrow\infty$~\citep{rahimi08uniform}; this bound on the density $\alpha(\theta)$ is needed to obtain a uniform approximation result. With this notation in hand, we may extend the approximation theory of \citet{rahimi08uniform}
to vector-valued functions.
\begin{restatable}[Approximation error]{myprop}{uniform}
\label{prop:uniform_approx}
Let $X \subset \R^n$ be compact. Fix $\delta \in (0, 1)$, $B_h > 0$, $h \in \calF_2(B_h)$, and a positive integer $K$.
Let $\theta_1, ..., \theta_K$ be i.i.d.\ draws from $\nu$.
Put $\eta = \frac{\delta}{2K}$.
With probability at least $1-\delta$, there exist 
weights $\{\alpha_i\}_{i=1}^{K} \subset \R^{d_1}$ such that $\norm{\alpha_i}_2 \leq B_h$ for
$i=1, ..., K$, and
\begin{align*}
    \bignorm{\frac{1}{K} \sum_{i=1}^{K} \Phi(\cdot, \theta_i) \alpha_i - h }_\infty &\leq \frac{2}{K} \E \bignorm{\sum_{k=1}^{K} \varepsilon_i \Phi_\eta(\cdot, \theta_i)\alpha(\theta_i)}_\infty \\
    &\qquad + \sqrt{2} B_\Phi(\eta) B_h \sqrt{\frac{\log(2/\delta)}{K}} + B_h \sqrt{\frac{\delta\sup_{x \in X} \E\opnorm{\Phi(x, \theta)}^2 }{2K}}.
\end{align*}
Above, each $\varepsilon_i$ is an i.i.d.\ Rademacher random variable\footnote{That is, $\Pr(\varepsilon_i = 1) = \Pr(\varepsilon_i = -1) = 1/2$.} and $\norm{f}_{\infty} := \sup_{x\in X}\norm{f(x)}_2$.
\end{restatable}

In order to bound the Rademacher complexity term 
appearing in Proposition~\ref{prop:uniform_approx},
we now make a few more assumptions on the structure of $\Phi(x, \theta)$. These assumptions are motivated by the operator-valued Bochner's theorem \citep{brault16randomfeatures}.
\begin{assmp}
\label{assmp:weights_biases}
The feature space $\Theta$ is a subset of $\R^{n+1}$, so that $\theta \in \Theta$ may be written as $\theta = (w, b)$ with $w \in \R^n$ and $b\in\R$. Moreover, the feature map can be factorized as $\Phi(x, \theta) = \phi(w^\T x + b)M(w)$ for $M:\R^{n}\rightarrow \R^{d \times d_1}$ and a $1$-Lipschitz scalar function $\phi : \R \rightarrow [-1, 1]$.
\end{assmp}
Because $\abs{\phi} \leq 1$, we may take $B_\Phi(\delta)$ to be any function that satisfies
$\Pr( \opnorm{M(w)} > B_\Phi(\delta)) \leq \delta$.
Accordingly, we have
$\Phi_\eta(x, \theta) = M_\eta(w) \phi(w^\T x + b)$
with $M_\eta(w)$ defined as 
$M_\eta(w) := M(w) \ind\{ \opnorm{M(w)} \leq B_\Phi(\eta) \}$. With these extra assumptions in place, we can bound the Rademacher complexity term as follows.
\begin{restatable}[Rademacher complexity bound]{myprop}{rad}
\label{prop:rademacher_bound}
Let Assumption~\ref{assmp:weights_biases} hold, and denote $B_X := \sup_{x \in X} \norm{x}_2$. Then for any $\eta \in (0, 1)$,
\begin{align*}
  \frac{2}{K} \E \bignorm{ \sum_{i=1}^{K} \varepsilon_i \Phi_\eta(\cdot; \theta_i) \alpha(\theta_i)}_\infty &\leq \frac{4 B_h B_\Phi(\eta)}{\sqrt{K}} \left[ B_X \sqrt{\E \norm{w_1}_2^2} + \sqrt{d_1} \right].
\end{align*}
\end{restatable}
Combining Proposition~\ref{prop:uniform_approx} and Proposition~\ref{prop:rademacher_bound}, we have that
with probability $1-\delta$,
\begin{align}
    &\inf_{\{\alpha_i\}_{i=1}^{K} \subseteq \R^{d_1} : \norm{\alpha_i}_2 \leq B_h} \bignorm{\frac{1}{K} \sum_{k=1}^{K} \Phi(\cdot, \theta_i) \alpha_i - h}_\infty \nonumber \\
    &\leq \frac{B_h}{\sqrt{K}} \bigg[ 2 B_\Phi\left(\frac{\delta}{2K}\right) \left(2 B_X \sqrt{\E\norm{w_1}_2^2} + 2 \sqrt{d_1} + \sqrt{\log(2/\delta)}\right) + \sqrt{\frac{\delta}{2} \E\opnorm{M(w)}^2 } \bigg]. \label{eq:final_rf_bound}
\end{align}
To simplify this expression, we now look at some particular choices of kernels.

\subsection{Examples of Reproducing Kernels}
\label{sec:rf:examples}

In what follows, we consider a few examples of vector-valued kernels.
\subsubsection{Shift-invariant kernels}

First, we consider shift-invariant kernels from~\citet{brault16randomfeatures} and~\citet{minh16operator}.
Let $\mathsf{k}(x - z)$ be an arbitrary scalar shift-invariant kernel
and denote by $\mu$ the normalized inverse Fourier transform of $\mathsf{k}(\cdot)$.
We will assume generically that $\E_{w \sim \mu}\norm{w}_2^2 \asymp n$ where $\mu$ denotes the marginal of $\nu$ over $b$.

\paragraph{Decomposable kernels}
Let $\kernel(x, z) = A \mathsf{k}(x - z)$ for any positive semidefinite $A = BB^\T$.
Then $\Phi(x, \theta) = B\cos(w^\T x + b)$ and  $B_\Phi(\delta) = \opnorm{B}$.
Here, the approximation error bound \eqref{eq:final_rf_bound}
scales as $\frac{B_h \opnorm{B}}{\sqrt{K}} \left(B_X \sqrt{n} + \sqrt{d_1} \right)$.

\paragraph{Curl-free kernel}
Let $n=d$ and set $\kernel(x, z) = -\nabla^2 \mathsf{k}(x - z)$.
Then $A(w) = ww^\T$ and $\Phi(x, \theta) = w \cos(w^\T x + b)$. If $\mu \sim N(0, \sigma^2 I)$, then
$B_\Phi(\delta) = \sqrt{n} + 2 \sigma \sqrt{\log(1/\delta)}$
by standard Gaussian concentration results.
The approximation error bound \eqref{eq:final_rf_bound} then scales as $\frac{B_h (B_x \vee 1)}{\sqrt{K}}(n + \log{K})$.

\paragraph{Divergence-free kernel}
Again let $n=d$. Set $\kernel(x, z) = (\nabla^2 - I \Delta) \mathsf{k}(x-z)$,
where $\Delta$ is the Laplacian and $I$ is the identity matrix.
Then $A(w) = \norm{w}_2^2 P^\perp_{w}$,
where $P_{M}$ denotes the orthogonal projection onto the range of $M$ and
$P^\perp_{M} = I - P_{M}$. Hence,
$\Phi(x, \theta) = \norm{w}_2 P^\perp_{w} \cos(w^\T x + b)$.
If $\nu \sim N(0, \sigma^2 I)$, then
$B_\Phi(\delta) = \sqrt{n} + 2\sigma \sqrt{\log(1/\delta)}$.
The approximation error bound \eqref{eq:kernel_form}
also scales as $\frac{B_h (B_x \vee 1)}{\sqrt{K}}(n + \log{K})$.

\subsubsection{Other kernels}
We now consider some other possible choices of kernels.

\paragraph{Kernels leveraging prior physical information}
Any known physical structure can easily be combined with
reproducing kernels. As a concrete example,
suppose the state $x$ decomposes as $x = (x_1, x_2) \in \R^{n_1 + n_2}$, and
that the unknown dynamics factorizes as $h(x) = h_1(x_1) h_2(x_2)$, where $h_1 : \R^{n_1} \rightarrow \R^d$ is a known vector-valued function and
$h_2 : \R^{n_2} \rightarrow \R$ is an unknown function in an RKHS with scalar kernel $\mathsf{k}$.
Then we can set $\kernel((x_1, x_2), (z_1, z_2)) = h_1(x_1) h_1(z_1)^\T \mathsf{k}(x_2, z_2)$.
This type of structural simplification is common
in, e.g., robotic applications~\citep{sanner_robot}.

\paragraph{The neural tangent kernel}
The neural tangent kernel \citep{jacot18ntk} was recently developed as an approximation to infinitely wide deep neural networks. Consider a network $h(x, \theta)$, where
$x$ denotes the network input
and $\theta$ denotes the network parameters.
The NTK is defined as the following kernel:
\begin{align*}
    \kernel(x, z) = \E_{\theta \sim \mathcal{D}}\left[\frac{\partial h}{\partial \theta}(x, \theta)^\T \frac{\partial h}{\partial \theta}(z, \theta)\right],
\end{align*}
where $\mathcal{D}$ is the distribution used to initialize the
weights of the network.
Expressions of the NTK for various common architectures are available in closed form \citep{arora2019ntk}.
\section{Randomized adaptive control and prediction}
\label{sec:p_results}
We now demonstrate how the nonparametric input in Theorem~\ref{thm:nonparametric_conv} can be approximated using the uniform approximation theory of Section~\ref{sec:rf} to obtain adaptive control and prediction algorithms with high-probability guarantees of convergence. We state completely general results under the assumption that the unknown dynamics $h(\cdot)$ can be uniformly approximated to a desired degree of accuracy, similar to the classical results of~\citet{sanner_nn} but in a generalized context. Taking $h(\cdot)$ to lie in the function space $\calF_2$ and applying the results of Section~\ref{sec:rf} immediately gives a sufficient bound on the number of random features needed to track the desired trajectory to a given tolerance. 
\subsection{Deadzones}
Before we present our main approximate algorithms, we first introduce the notion of a deadzone. Since any finite-dimensional
approximation to $h(\cdot)$ will have some non-zero approximation error,
any adaptive algorithm cannot learn below this noise floor;
a deadzone allows us to disable adaptation when the only residual error
remaining is due to approximation error.

\begin{mydef}
\label{def:dead}
Let $\Delta > 0$.
A continuously differentiable function $\sigma_\Delta : \R_{\geq 0} \rightarrow \R$ is called
$\Delta$-admissible deadzone if:
\begin{enumerate}[(i)]
    \item $0 \leq \sigma_\Delta$ and $\sigma_\Delta(x) = 0$ for all $x \in [0, \Delta]$,
    \item $0 \leq \sigma'_\Delta$ and $\sigma'_\Delta(x) = 0$ for all $x \in [0, \Delta]$, 
    \item $\sigma'_\Delta$ is locally Lipschitz. \label{eq:condition_locally_lip}
\end{enumerate}
The function $\sigma_\Delta$ is called a $(\Delta, L, B)$-admissible deadzone if
condition \eqref{eq:condition_locally_lip} is replaced with the condition that
$\sigma'_\Delta$ is $L$-Lipschitz and $B$-bounded.
\end{mydef}

We now give some examples of $\Delta$-admissible deadzones.
The first example is a direct extension of the deadzone used
in \citet{sanner_nn}.
\begin{restatable}{ex}{sdeltadeadzone}
\label{prop:s_delta_deadzone}
Fix a scalar $\delta > 0$. Let $s_\delta : \R_{\geq 0} \rightarrow \R_{\geq 0}$
be defined as $s_\delta(x) := (x-\delta) \ind\{x > \delta\}$.
For any $\Delta > 0$, the function $x \mapsto s_{\sqrt{\Delta}}^2(\sqrt{x})$ is a
$(\Delta, 1/(2\Delta), 1)$-admissible deadzone.
\end{restatable}

An issue with a deadzone based on $s_\delta$ is that the Lipschitz
constant of the derivative diverges with vanishing $\Delta$. This makes it challenging to prove sharp ``approximate interpolation'' results similar to Theorem~\ref{thm:interpolation}.
To remedy this issue, we construct a deadzone
where the Lipschitz constant of the derivative is 
decoupled from $\Delta$. The following construction
is directly inspired by smooth approximations to the hinge loss
for support vector machines (see e.g. \citet{chapelle07svm}).
\begin{restatable}{ex}{sdeltagammadeadzone}
\label{prop:s_delta_gamma_deadzone}
Fix $\delta > 0$ and $\gamma > 0$.
Define $s_{\delta,\gamma}$ as:
\begin{align*}
    s_{\delta,\gamma}(x) := \begin{cases}
        0 &\text{if } x \leq \delta, \\
        \frac{(x-\delta)^2}{4\gamma} &\text{if } x \in (\delta, \delta+2\gamma), \\
        x - (\delta+\gamma) &\text{if } x \geq \delta + 2\gamma.
    \end{cases}
\end{align*}
For any $\Delta > 0$ and $\gamma > 0$, the function $s_{\Delta,\gamma}$ is a $(\Delta, 1/(2\gamma), 1)$-admissible deadzone.
\end{restatable} 
Worked details of Examples~\ref{prop:s_delta_deadzone} and~\ref{prop:s_delta_gamma_deadzone} may be found in Appendix~\ref{app:p_results}. Our results to come will be stated in terms of an arbitrary deadzone according to Definition~\ref{def:dead}, but concrete instantiations can be found via these prescriptions.

\subsection{Adaptive control}
We are now ready to present our main result
in the setting of approximate control. The following is a general result about adaptive control with uniform approximation that can be applied with an arbitrary choice of basis.
\begin{restatable}[Adaptive control with finite-dimensional approximation]{mythm}{acfiniteapprox}
\label{thm:ac_finite_approx}
Suppose that Assumption~\ref{assmp:lyapunov} holds.
Let $\alpha_{\ell,0} := \arg\min_{\alpha \in O_\ell} \psi_\ell(\alpha)$ for $\ell \in \{p, m\}$.
Fix $B_{\alpha_p} > 0$ satisfying $\bregdp{\alpha_p}{\alpha_{p,0}} \leq B_{\alpha_p}$,
$B_{\alpha_m} > 0$, and 
$R$ satisfying
\begin{align*}
    R > \mu_1^{-1}\left( Q(e(0), 0) + B_{\alpha_p} + B_{\alpha_m} \right).
\end{align*}
Suppose there exists a finite $C_e$ such that
for every $T > 0$:
\begin{align}
    \max_{t \in [0, T]} \norm{e(t)}_2 \leq R \text{ implies } \norm{x(T) - x_d(T)}_2 \leq C_e R. \label{eq:small_error_implies_small_state}
\end{align}
Let $\Psi : \R^n \rightarrow \R^{d \times m}$ be a locally Lipschitz feature map.
Define the constants
\begin{align*}
    B_d &:= \sup_{t \geq 0} \norm{x_d(t)}_2, \\
    B_x &:= C_e R + B_d, \\
    B_{g_e} &:= \sup_{t \geq 0} \sup_{\norm{x}_2 \leq B_x} \opnorm{g_e(x, t)}, \\
    B_{\nabla Q} &:= \sup_{t \geq 0} \sup_{\norm{e}_2 \leq R} \norm{\nabla Q(e, t)}_2, \\
    B_{\mathrm{approx}} &:= \inf_{\bregdm{\alpha_m}{\alpha_{m,0}} \leq B_{\alpha_m}}  \sup_{\norm{x}_2 \leq B_x} \norm{\Psi(x)\alpha_m - h(x)}_2.
\end{align*}
Let $\Delta$ be any positive constant satisfying
\begin{align*}
    \Delta \geq \mu_2(\rho^{-1}(2B_{g_e} B_{\nabla Q} B_{\mathrm{approx}})),
\end{align*}
and let $\sigma_\Delta$ be a $\Delta$-admissible deadzone.
Then the dynamical system
\begin{align*}
    \dot{x} &= f(x, t) + g(x, t)(u(x, t) - Y(x, t) \alpha_p - h(x)), \\
    \dot{e} &= f_e(e, t) + g_e(x, t)( u(x, t) - Y(x, t) \alpha_p - h(x)), \\
    u(x, t) &= Y(x, t) \hat{\alpha}_p + \Psi(x) \hat{\alpha}_m, \\
    \frac{d}{dt} \nabla \psi_p(\hat{\alpha}_p) &= -\sigma'_\Delta(Q(e, t)) Y(x, t)^\T g_e(e, t)^\T \nabla Q(e, t), \\
    \frac{d}{dt} \nabla \psi_m(\hat{\alpha}_m) &= - \sigma'_\Delta(Q(e, t)) \Psi(x)^\T g_e(e, t)^\T \nabla Q(e, t),
\end{align*}
with initial conditions $x(0) = x_0$,
$e(0) = m(x_0, 0)$, $\hat{\alpha}_p(0) = \alpha_{p,0}$,
and $\hat{\alpha}_m(0) = \alpha_{m,0}$ has a
solution $(x(t), e(t), \hat{\alpha}_p(t), \hat{\alpha}_m(t))$ that exists for all $t \geq 0$.
Furthermore,
\begin{align*}
    \limsup_{t \rightarrow \infty} \norm{e(t)}_2 \leq \mu_1^{-1}(\Delta).
\end{align*}
\end{restatable}
Theorem~\ref{thm:ac_finite_approx} can be used in conjunction with the results of Section~\ref{sec:rf} to obtain a high-probability guarantee for control, as illustrated by the following example.

\begin{ex}[Adaptive control with random features]
\label{ex:rf_adaptive}
Suppose for simplicity that $\bregdm{x}{y} = \frac{1}{2}\norm{x-y}_2^2$ is the Euclidean distance.
Fix a positive integer $K$,
and let $\delta \in (0, 1)$.
Assume that $h\in\calF_2(B_h)$ under 
Assumption~\ref{assmp:weights_biases}, 
and again for simplicity assume that 
the kernel is decomposable as in Section~\ref{sec:rf:examples}.
Set $B_{\alpha_m} = B_h^2/2$. Let $\{\theta_i\}_{i=1}^{K}$ be i.i.d.\ draws from $\nu$. Then, by Equation~\ref{eq:final_rf_bound},
with probability at least $1-\delta$
there exists $\alpha_m = (\alpha_{m,1}, ..., \alpha_{m,K}) \in \R^{K d_1}$ satisfying $\norm{\alpha_{m,i}}_2 \leq B_h/K$ for $i=1, ..., K$ and
\begin{align*}
    \sup_{\norm{x}_2 \leq B_x} \norm{ h(x) - \Psi\left(x; \{\theta_i\}_{i=1}^{K}\right) \alpha_m }_2 \leq \frac{C(h, \delta) (B_x \sqrt{n} + \sqrt{d_1})}{\sqrt{K}},
\end{align*}
with $\Psi(x; \{\theta_i\}_{i=1}^{K}) = \begin{bmatrix} \Phi(x, \theta_1), ...,\Phi(x, \theta_K) \end{bmatrix} \in \R^{d \times K d_1}$.
Here, $C(h, \delta) > 0$ is a constant that depends only on $h$ and $\delta$. 
Note that
\begin{align*}
    \bregdm{\alpha_m}{0} = \frac{1}{2}\norm{\alpha_m}_2^2 = \frac{1}{2}\sum_{i=1}^{K} \norm{\alpha_{m,i}}_2^2 \leq \sum_{i=1}^{K} \frac{B_h^2}{2K^2} = \frac{B_h^2}{2K} \leq B_{\alpha_m},
\end{align*}
so that $B_{\mathrm{approx}} \leq \frac{C(h, \delta) (B_x \sqrt{n} + \sqrt{d_1})}{\sqrt{K}}$.
Hence, to ensure $\limsup_{t\rightarrow\infty}\norm{e(t)}_2 \leq \varepsilon$ for some $\varepsilon > 0$, it suffices to take $K$ satisfying
\begin{align*}
    K \geq \frac{4 B_{g_e}^2 B_{\nabla Q}^2 C(h, \delta)^2 (B_x \sqrt{n} + \sqrt{d_1})^2}{\rho^2(\mu_2^{-1}(\mu_1(\varepsilon)))}.
\end{align*}
Suppose that
$\mu_1(x) = \mu x$, $\mu_2(x) = L x$, and $\rho(x) = \beta x$\footnote{For $V(t)$ a quadratic Lyapunov function certifying exponential stability, it is a simple calculation to show that one can take $Q(t) = \sqrt{V(t)}$ to obtain such linear functions for $\mu_1, \mu_2$ and $\rho$.}. Then this bound simplifies to
\begin{align*}
    K \geq \frac{4}{\beta^2 \varepsilon^2} \left(\frac{L}{\mu}\right)^2 B_{g_e}^2 B_{\nabla Q}^2 C(h, \delta)^2 (B_x \sqrt{n} + \sqrt{d_1})^2.
\end{align*}
\end{ex}

\paragraph{Approximation region} For simplicity of presentation, we have chosen the approximation region in Theorem~\ref{thm:ac_finite_approx}
large enough to cover the variation of the error signal throughout
adaptation. Alternatively, the approximation region can be specified \textit{a-priori}, and sliding mode control can be used to force the system to stay inside the approximation region. Such a formulation requires additional technical assumptions on the error dynamics.

\paragraph{Contraction} Assume that the error dynamics is contracting. Then we may take $Q(e, t)$ to be the Riemannian energy as in Remark~\ref{rmk:lyap_contr} and set $\psi_\ell(\cdot) = \frac{1}{2}\norm{\cdot}_2^2$
for $\ell \in \{p, m\}$ to recover the contraction metric-based adaptation law due to~\cite{brett_adapt}
\begin{equation*}
    \dot{\hat{\alpha}}_m = -\Psi(x)^\T g_e(x, t)^\T M(e, t)\gamma_s(e, 0, t).
\end{equation*}
Here, $\gamma_s(e, 0, t)$ denotes the tangent vector to a geodesic in the metric $M(e, t)$ between $e$ and the origin at the endpoint $e$ (a similar metric-based update also holds for $\hat{\alpha}_p$).

\paragraph{Mirror descent} By analogy to mirror descent, the choice of potential functions $\psi_p(\cdot)$ and $\psi_m(\cdot)$ can be used to regularize the learned physical and random feature models, or can be used to improve convergence when adapted to the problem geometry~\citep{boffi_neco_imp_reg}. The random sinusoidal features considered in Section~\ref{sec:rf} are uniformly bounded in $\ell_\infty$ norm independent of the number of parameters. This observation suggests that, for a large number of features, a potential function strongly convex with respect to the $\ell_1$ norm such as the hypentropy potential due to~\citet{pmlr-v117-ghai20a} may lead to improved performance.

\paragraph{Interpolation} We conclude our treatment of adaptive control by presenting an approximate version of Theorem~\ref{thm:interpolation}, which demonstrates
how the approximation error from finite-dimensional truncation
translates into an interpolation error for the learned dynamics approximation.
Specifically, if Theorem~\ref{thm:ac_finite_approx}
is invoked with a $(\Delta, L, B)$-admissible deadzone,
then the following result shows that the interpolation error is bounded by $O\left(\sqrt{\mu_1^{-1}(\Delta)(1+L)}\right)$. This motivates the construction in Example~\ref{prop:s_delta_gamma_deadzone}.

\begin{restatable}[Approximate interpolation]{mythm}{acapproxinterp}
\label{thm:ac_approx_interp}
Suppose the hypotheses of Theorem~\ref{thm:ac_finite_approx} hold.
Let $\sigma_\Delta$ denote a
$(\Delta, L, B)$-admissible deadzone,
and assume that $f_e$, $g_e$, and $Y$ are locally Lipschitz in 
their first arguments uniformly in $t$.
Then there exist constants $C_1 > 0$ and $C_2 > 0$ not depending on
$\Delta$ such that
\begin{align*}
    \limsup_{t \rightarrow \infty} \norm{g_e(x(t), t)(u(x(t), t) - Y(x(t), t) \alpha_p - h(x(t)))}_2 \leq C_1 \sqrt{\mu_1^{-1}(\Delta)(1+L)} + C_2 \mu_1^{-1}(\Delta).
\end{align*}
\end{restatable}

\subsection{Adaptive prediction}
Similar to Theorem~\ref{thm:ac_finite_approx}, the following theorem designs a predictor by leveraging the ability to uniformly approximate the unknown dynamics to a suitable degree of accuracy.

\begin{restatable}[Adaptive prediction with uniform approximation]{mythm}{dpfiniteapprox}
\label{thm:dp_finite_approx}
Suppose that the trajectory $x(t)$ of the system $\dot{x} = f(x, t)$ 
is uniformly bounded.
Choose a continuous and locally Lipschitz $k(\hat{x}, x)$ such that $f(\hat{x}, t) + k(\hat{x}, x(t))$ is
contracting in a metric $M : \R^n \times \R_{\geq 0} \rightarrow \Symm^{n \times n}_{\geq 0}$ with rate $\lambda > 0$,
and suppose that the metric $M$ satisfies $\mu I \preccurlyeq M(\hat{x}, t) \preccurlyeq L I$ for all $\hat{x}$ and $t$.
Let $\gamma(\cdot; \hat{x}, x, t):[0, 1]\rightarrow\R^n$ denote a geodesic between $\hat{x}$ and $x$ in the metric $M(\hat{x}, t)$, and let $\gamma_s(s; \hat{x}, x, t)$ denote the derivative of $s \mapsto \gamma(s; \hat{x}, x, t)$. 
Suppose that the map 
$(\hat{x}, t) \mapsto \norm{\gamma_s(0; \hat{x}, x(t), t)}_2$
is locally bounded in $\hat{x}$ uniformly in $t$.
Fix any $B_{\alpha_p} > 0$ satisfying $\bregdp{\alpha_p}{\alpha_{p,0}} \leq B_{\alpha_p}$,
any $B_{\alpha_m} > 0$, and 
any $R$ satisfying
\begin{align*}
    R > \sqrt{\frac{Q(\hat{x}(0), 0) + B_{\alpha_p} + B_{\alpha_m}}{\mu}}, \:\: Q(\hat{x}, t) := E_{M(\cdot, t)}(\hat{x}, x(t)).
\end{align*}
Let $\Psi : \R^n \rightarrow \R^{d \times m}$ be a locally Lipschitz feature map. 
Define the following constants
\begin{align*}
    B_x &:= \sup_{t \geq 0} \norm{x(t)}_2, \\
    B_{\hat{x}} &:= R + B_x, \\
    B_\gamma &:= \sup_{t \geq 0} \sup_{\norm{\hat{x}}_2 \leq B_{\hat{x}}} \norm{\gamma_s(0; \hat{x}, x(t), t)}_2, \\
    B_{\mathrm{approx}} &:= \inf_{\bregdm{\alpha_m}{\alpha_{m,0}} \leq B_{\alpha_m}}  \sup_{\norm{\hat{x}}_2 \leq B_{\hat{x}}} \norm{\Psi(\hat{x})\alpha_m - h(\hat{x})}_2.
\end{align*}
Choose any $\Delta$ satisfying $\Delta \geq \frac{L^2 B_\gamma B_{\mathrm{approx}}}{\lambda \mu}$,
and let $\sigma_\Delta$ be a $\Delta$-admissible deadzone.
Then the dynamical system
\begin{align*}
    \dot{\hat{x}} &= \hat{f}(\hat{x}, \hat{\alpha}_p, \hat{\alpha}_m, t) + k(\hat{x}, x(t)), \\
    \hat{f}(\hat{x}, \hat{\alpha}_p, \hat{\alpha}_m, t) &= Y(\hat{x}, t)\hat{\alpha}_p + \Psi(\hat{x}) \hat{\alpha}_m, \\
        \frac{d}{dt} \nabla \psi_p(\hat{\alpha}_p) &= -\sigma'_\Delta(Q(\hat{x}, t)) Y(\hat{x}, t)^\T \nabla Q(\hat{x}, t), \\
    \frac{d}{dt} \nabla \psi_m(\hat{\alpha}_m) &= - \sigma'_\Delta(Q(\hat{x}, t)) \Psi(\hat{x})^\T \nabla Q(\hat{x}, t),
\end{align*}
with initial conditions $\hat{x}(0) = \hat{x}_0$, 
$\hat{\alpha}_p(0) = \alpha_{p,0}$, and $\hat{\alpha}_m(0) = \alpha_{m,0}$ has a solution that exists for all $t \geq 0$. Furthermore, 
\begin{align*}
    \limsup_{t \rightarrow \infty} \norm{\hat{x}(t) - x(t)}_2 \leq \sqrt{\frac{\Delta}{\mu}}.
\end{align*}
\end{restatable}

\paragraph{Constructing a metric}
\Cref{thm:dp_finite_approx} requires a metric $M(\hat{x}, t)$
such that $\bar{f}(\hat{x}, t) := f(\hat{x}, t) + k(\hat{x}, x(t))$
is contracting. One such metric can always be obtained by taking $k(\hat{x}, x) = -\zeta (\hat{x} - x)$, in which case $\frac{\partial \bar{f}}{\partial \hat{x}}(\hat{x}, t) = \frac{\partial f}{\partial \hat{x}}(\hat{x}, t) - \zeta I$. If we further assume that $\frac{\partial f}{\partial \hat{x}}$ is locally bounded in $\hat{x}$ uniformly in $t$ and that $x(t)$ is uniformly bounded, then there exists a finite $\zeta \in (0, \infty)$ such that
$\bar{f}$ is contracting in the identity metric $M(\hat{x}, t) = I$.
In the case where $f$ is known, $k$ can be tailored to the system physics to obtain improved convergence~\citep{chung2009cooperative}.

\paragraph{Duality} The proof of Theorem~\ref{thm:dp_finite_approx} highlights a duality between the nonlinear adaptive control and nonlinear adaptive prediction problems reminiscent of the duality between LQR and Kalman filtering in linear control theory. Intuitively, any model capable of predicting the time evolution of a system could be used to control the system. Conversely, a model that can be used to control a system could instead be used to predict its evolution.

\paragraph{Interpolation} Theorem~\ref{thm:dp_finite_approx} assumes that the true system state $x(t)$ is measured continuously and concludes that the learned prediction $\hat{x}(t)$ will asymptotically become consistent with the observed measurements up to a level specified by the accuracy of the uniform approximation. Applying duality, the interpolation result in Theorem~\ref{thm:ac_approx_interp} shows that the learned model $\hat{f}(\hat{x}, \hat{\alpha}_p, \hat{\alpha}_m, t)$ becomes approximately consistent with the true model along the trajectory $x(t)$. 

\paragraph{Discrete sampling} In practical applications, measurements of the true system state are obtained at discrete instants, and an open-loop predictor with fixed parameters is used to extrapolate beyond them. The parameters are then updated according to a discretized adaptation law when measurements are received. In 
Appendix~\ref{app:sample}, we demonstrate how the nominal contraction properties required by Theorem~\ref{thm:dp_finite_approx} can be preserved with discrete measurements by taking the feedback term $k(\hat{x}, x)$ to have a sufficiently high contraction rate in comparison to the spacing between measurements $\Delta t$.
\section{Simulations}
\label{sec:simulations}
We now study the empirical performance of the nonparametric method and its randomized approximation. In the control setting we directly compare the kernel and approximate inputs. In prediction we illustrate the ability of the random feature approximation to scale to high-dimensional systems. In addition, we study the convergence of the prediction and interpolation errors as a function of $K$.

\subsection{Adaptive control}
\label{sec:exp:control}
Here we consider a synthetic example in adaptive control to compare the nonparametric adaptive input to its randomized approximation. 

\paragraph{System dynamics} We study the stable linear time-invariant system
\begin{align}
    \dot{x} = A \left(x - \frac{3}{2}\ind\right) + u(x, t) - h(x), \:\: x \in \R^5, \:\: h(x) = \sin(x) \mathrm{erf}(x), \label{eq:adaptive_control_exp}
\end{align}
where $A$ is a known matrix
with eigenvalues lying entirely in the left half-plane and $\mathbf{1}$ denotes the vector of ones. The operations defining $h$ are applied elementwise to each coordinate.
The error signal is set to $e(t) = x(t) - \frac{3}{2}\mathbf{1}$, and the desired trajectory is constant at the nominal equilibrium point
$x_d(t) = \frac{3}{2}\mathbf{1}$. This system admits a Lyapunov function $Q(x, t) = \frac{1}{2} (x-x_d(t))^\T P (x-x_d(t))$,
where $P$ is the unique positive definite solution to the Lyapunov matrix equation
$A^\T P + P A = - I$.

\paragraph{Implementation}
We apply a nonparametric input generated by the Gaussian kernel
\begin{align*}
    \kernel(x, y) = \exp\left(-\frac{\norm{x-y}_2^2}{2\sigma^2}\right) I, \:\: \sigma = 0.1.
\end{align*}
For its randomized approximation, we use the 
random Fourier features described in Section~\ref{sec:rf:examples}. Both the randomized and nonparametric adaptive laws are obtained by forward Euler integration with a fixed timestep
$\Delta t = 0.001$. At each time, the kernel input \eqref{eqn:kernel_input}
is evaluated via a Riemann sum approximation at the same resolution,
\begin{align*}
    u(x, t) = \int_0^t K(x, x(\tau))c(\tau)d\tau \approx \sum_{i=0}^{n_t}K(x, x(t_i))c(t_i)\Delta t
\end{align*}
with $n_t = t/\Delta t$. This corresponds to solving the pointwise-decoupled partial differential equation
\begin{equation*}
    \frac{\partial u}{\partial t}(x, t) = K(x, x(t))c(t)
\end{equation*}
again via forward Euler integration with a timestep $\Delta t$.

\begin{figure}[!t]
    \centering
    \begin{tabular}{ll}
        \begin{overpic}[width=.475\textwidth]{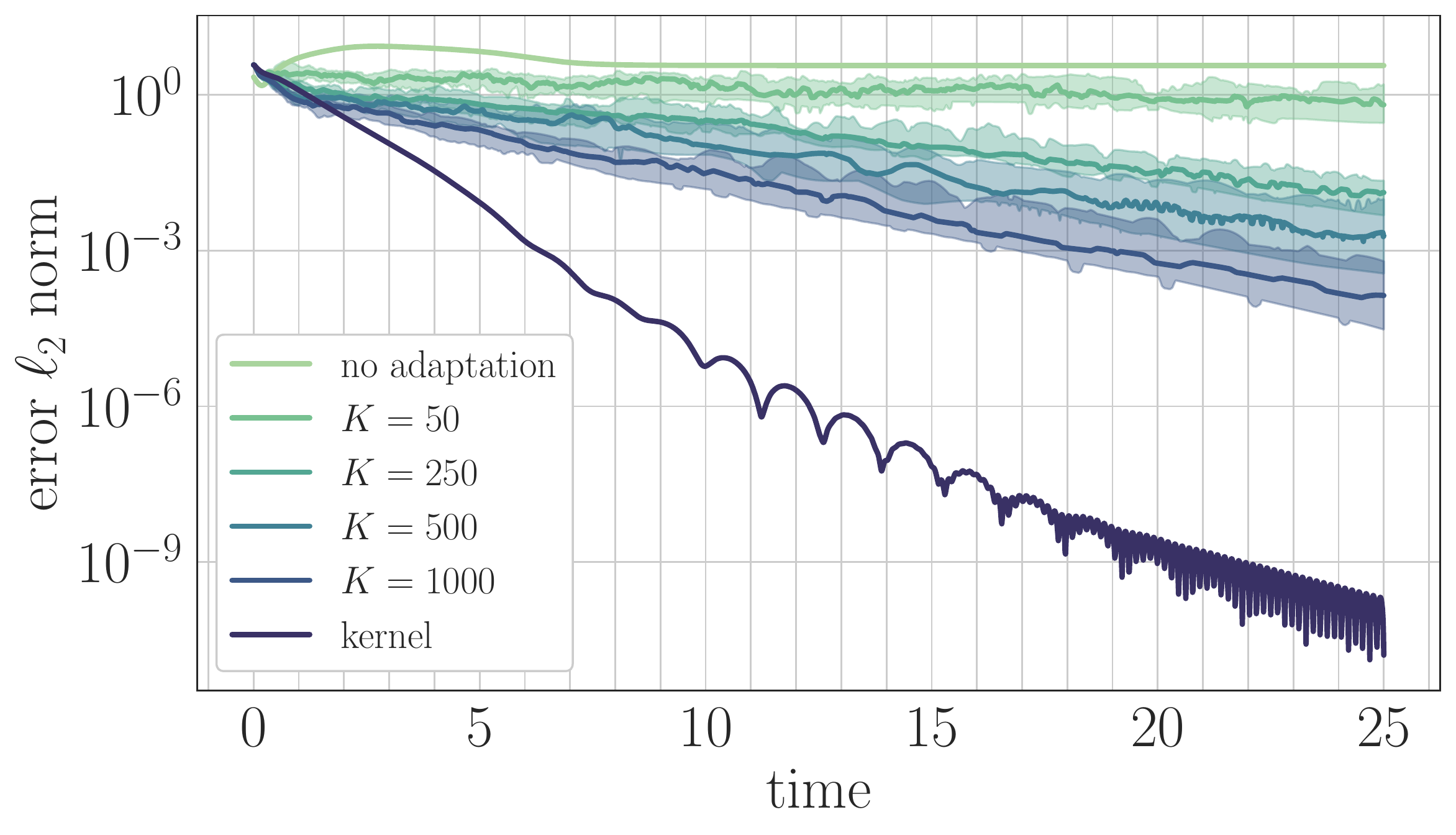}%
        \put(5, 60){\textbf{A}}
        \end{overpic}&  
        \begin{overpic}[width=.475\textwidth]{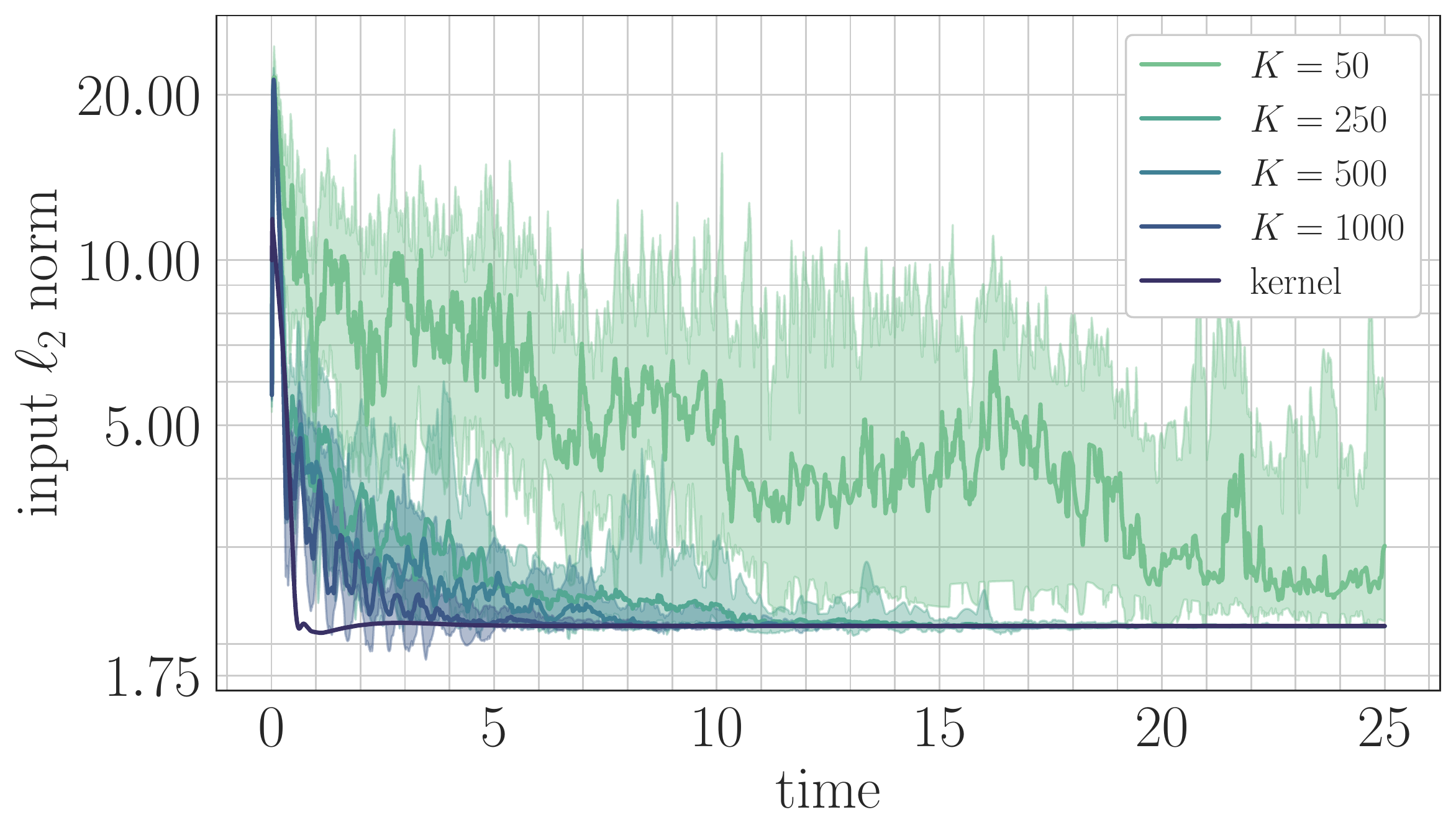}%
        \put(5, 60){\textbf{B}}
        \end{overpic}\\\\
        \multicolumn{2}{c}{%
        \begin{overpic}[width=.475\textwidth]{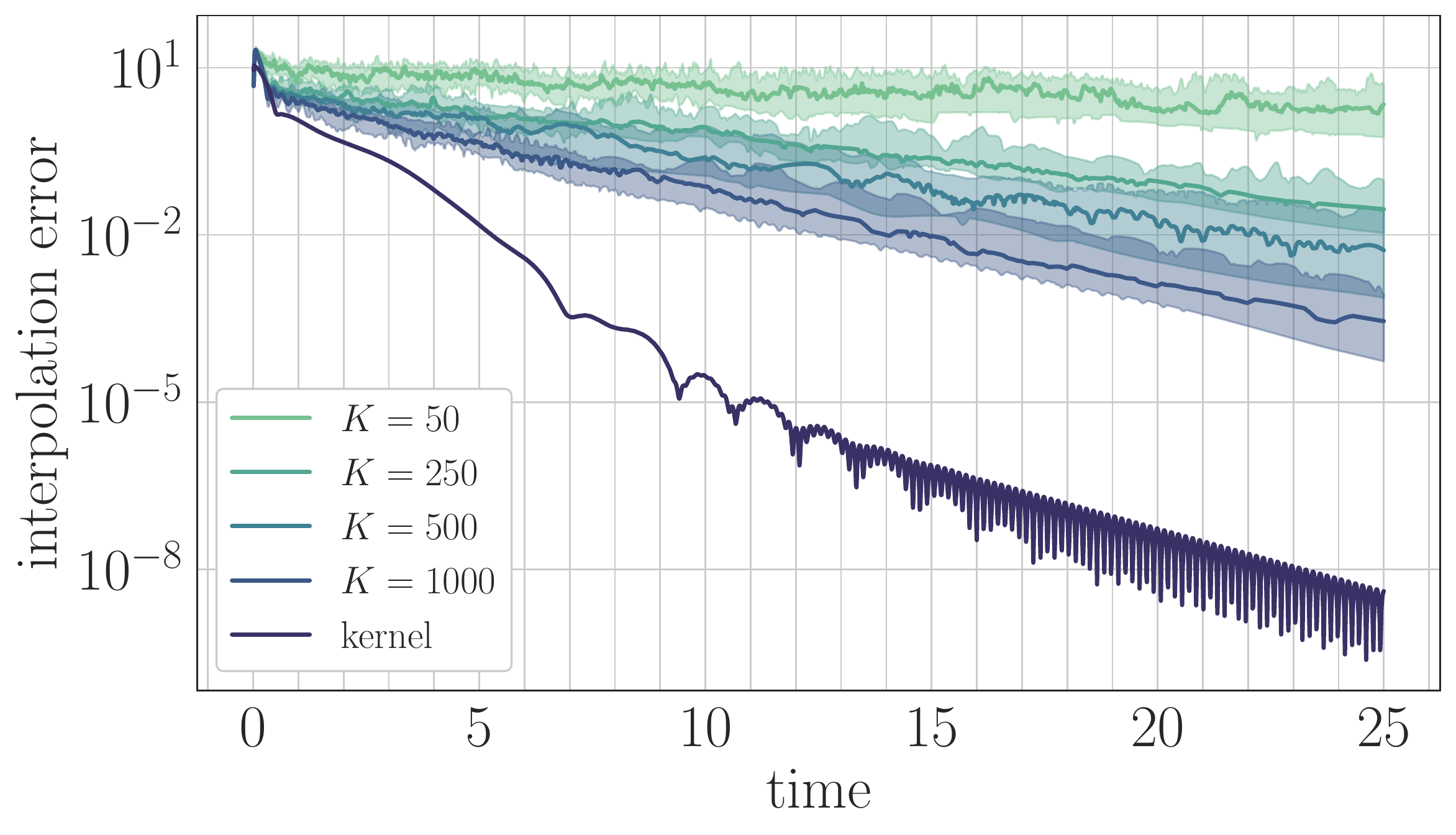}%
        \put(5, 60){\textbf{C}}
        \end{overpic}}        
    \end{tabular}
    \caption{\textbf{Adaptive control.} (A) Tracking error as a function of time. Error bars display the $20\%/80\%$ quantiles over $20$ trials (draws of the $\theta_i$) for each choice of $K$. Solid lines display the median. The tracking error decreases monotonically with the number of features, and the kernel input obtains the best performance by several orders of magnitude. (B) Magnitude of the adaptive input over time. The kernel input obtains the best performance despite using the lowest input magnitude. (C) Interpolation error as a function of time. Similar to the tracking error, the interpolation error decreases monotonically with increasing $K$, and the kernel input obtains the best performance by several orders of magnitude.}
    \label{fig:control}
\end{figure}

\paragraph{Results (Figure~\ref{fig:control})}%
Error bars around the random feature curves display the $20\%$ and $80\%$ quantiles, while the solid central curves display the corresponding median. In comparison to each value of $K$, the kernel input obtains the best tracking performance both transiently and asymptotically by several orders of magnitude. The tracking error at each fixed time decreases monotonically as a function of $K$ (Figure~\ref{fig:control}A).
The overall magnitude of the adaptive control input $\norm{u(x, t)}_2$ decreases monotonically as a function of $K$, and the kernel input consistently applies the lowest magnitude input despite obtaining the best performance (Figure~\ref{fig:control}B).
Similar to the tracking error, the kernel input obtains the best dynamics approximation by several orders of magnitude, and the dynamics interpolation error decreases monotonically as a function of $K$ for each fixed time (Figure~\ref{fig:control}C).

\subsection{Adaptive prediction}
The infinite-dimensional input considered in Theorem~\ref{thm:nonparametric_conv} enjoys guarantees that are independent of the system dimension. As shown by \eqref{eq:final_rf_bound}, the accuracy of the random feature approximation only depends polynomially on the system dimension. These observations suggest that the nonparametric input and its randomized approximations should scale well to high-dimensional systems.

\paragraph{Failures of uniform gridding} 
Any gridding-based approach must depend \textit{exponentially} on the system dimension, and as a consequence suffers from the curse of dimensionality. 
Modern robotic systems, for instance, often have state dimension in the twenties or thirties, which renders such approaches inapplicable for robotic control. 
For illustration, consider a uniform gridding method as suggested by the calculations in~\cite{sanner_nn}. For 
a nine-dimensional system, placing only ten basis functions in each direction would require one billion total basis functions, a computationally and statistically intractable number. 
Here, we study the efficiency of our randomized method in forming a predictive model of a sixty-dimensional system, and find that our randomized approach leads to good accuracy.

\paragraph{$m$-body system} Consider a system of $m$ point masses interacting via Newtonian gravitation in $d$ dimensions, and denote by $q_i\in\R^d$ the position of mass $i$ and $p_i\in\R^d$ the momentum of mass $i$. Assuming equal masses, such a system admits a Hamiltonian in non-dimensionalized units
\begin{equation*}
    H\left(\left\{p_i\right\}_{i=1}^m, \left\{q_i\right\}_{i=1}^m\right) = \sum_{i=1}^m\frac{\norm{p_i}_2^2}{2} - \sum_{i < j}^m\frac{1}{\norm{q_i - q_j}_2},
\end{equation*}
and a corresponding symplectic dynamics $\dot{q}_i = \frac{\partial H}{\partial p_i}$, $\dot{p}_i = -\frac{\partial H}{\partial q_i}$. We denote by $x = (q^\T, p^\T)^\T \in \R^{2md}$ with $q \in \R^{md}$ and $p \in \R^{md}$ vectors containing the stacked $q_i$ and $p_i$ over $i$.

\begin{figure}[!t]
    \centering
    \begin{tabular}{cc}
        \begin{overpic}[width=.475\textwidth]{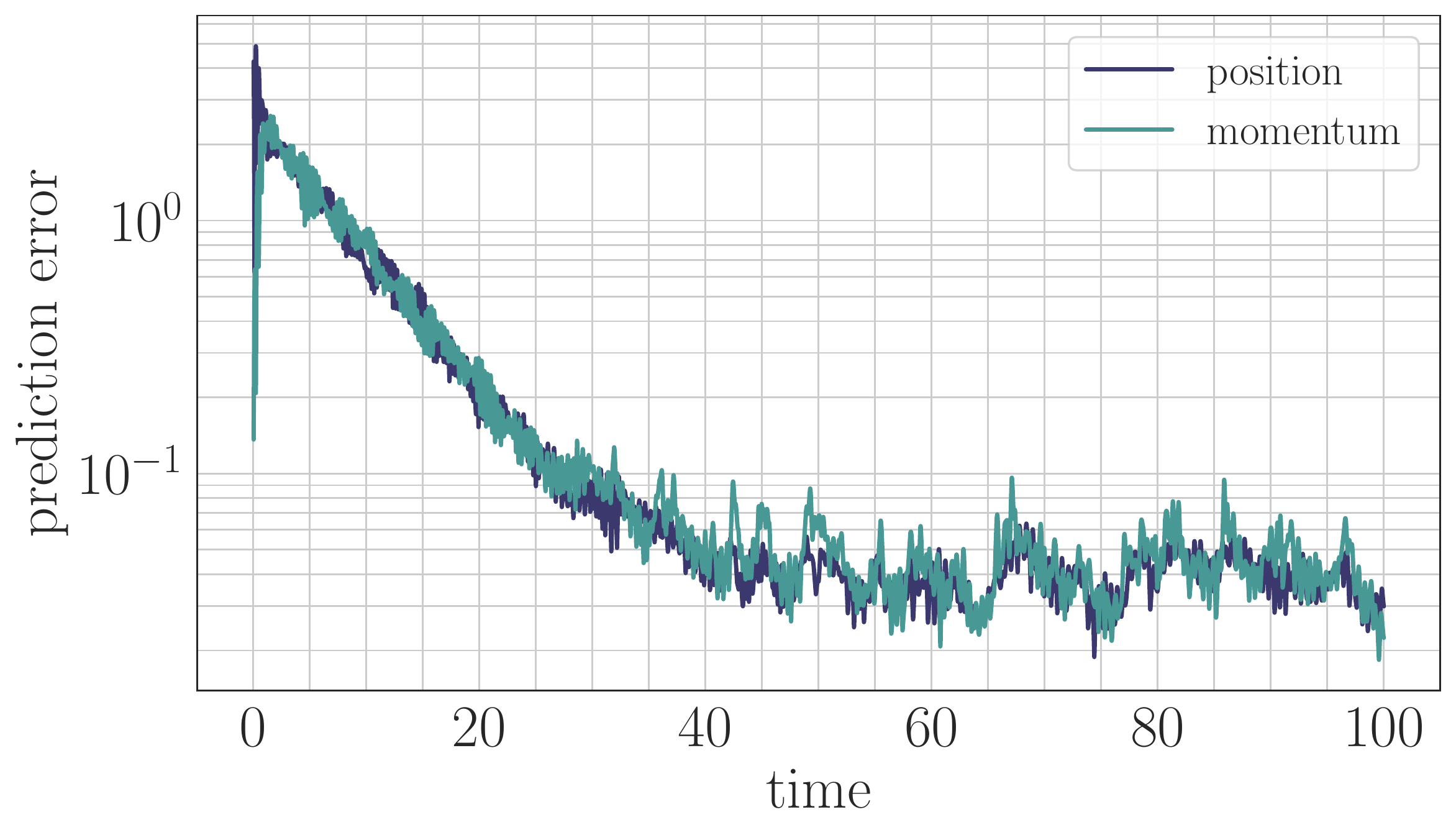}%
        \put(5, 55){\textbf{A}}
        \end{overpic}&  
        \begin{overpic}[width=.475\textwidth]{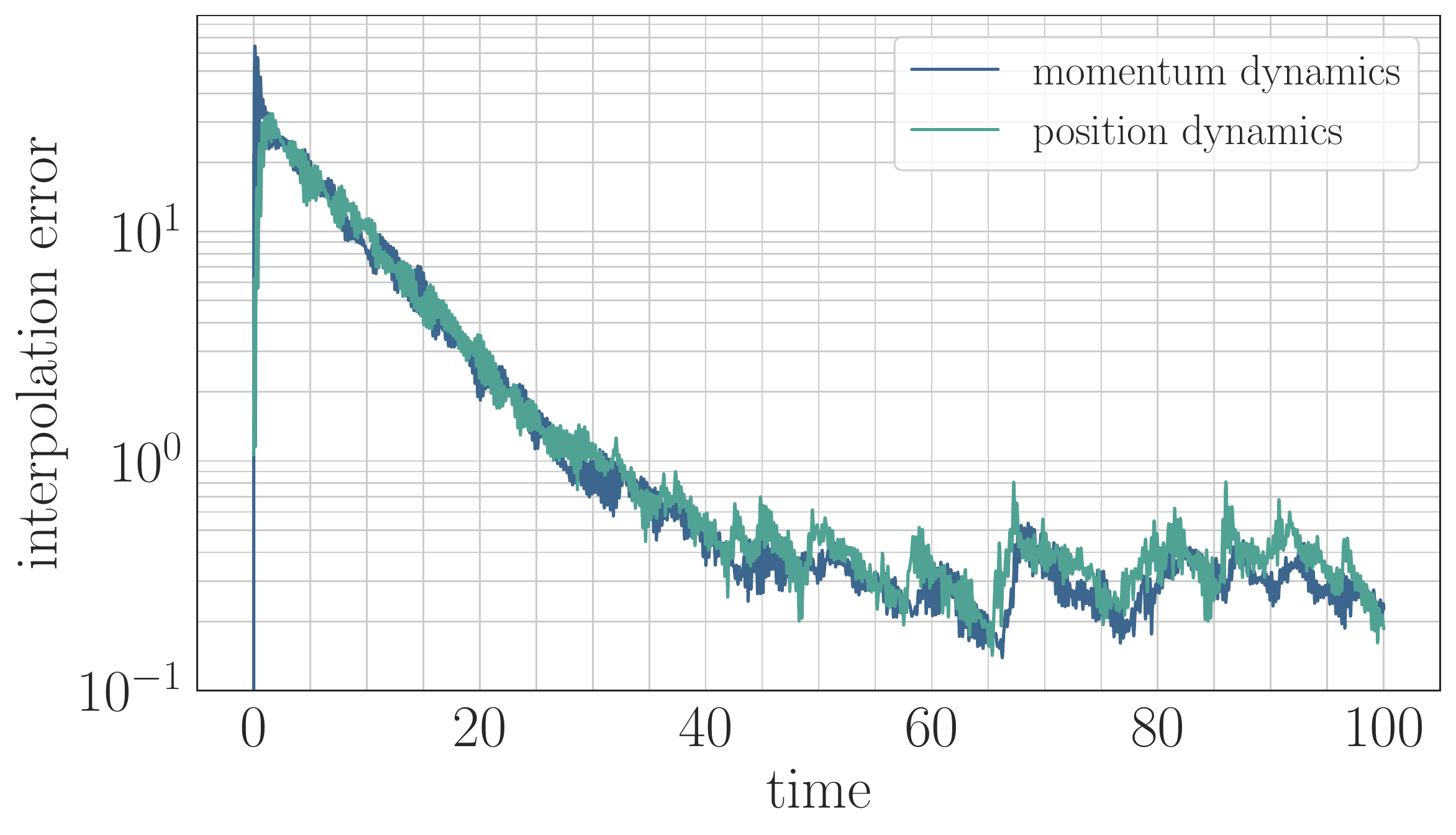}%
        \put(5, 55){\textbf{B}}
        \end{overpic}\\
        \begin{overpic}[width=.442\textwidth]{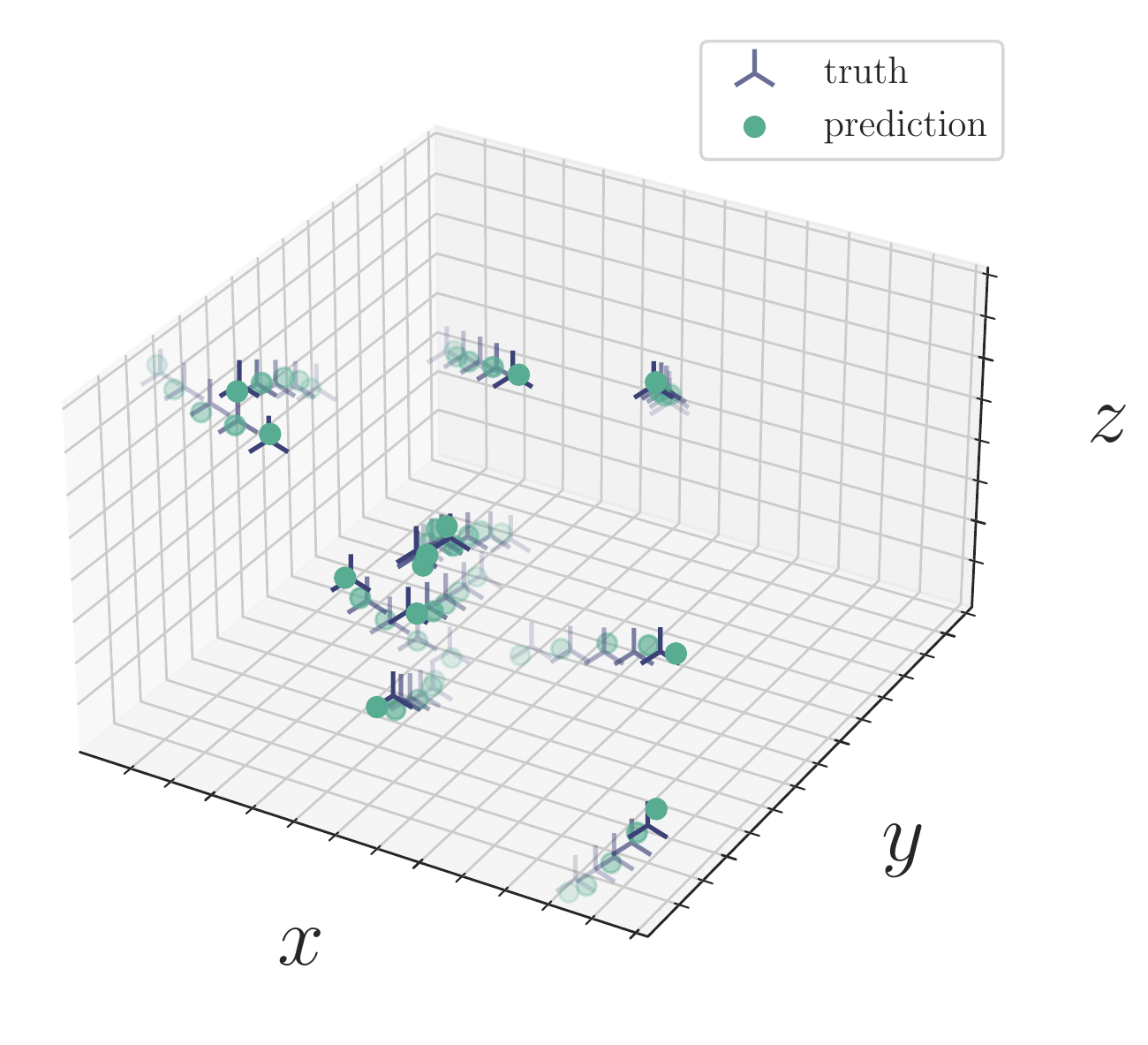}%
        \put(5, 75){\textbf{C}}
        \end{overpic}&  
        \begin{overpic}[width=.442\textwidth]{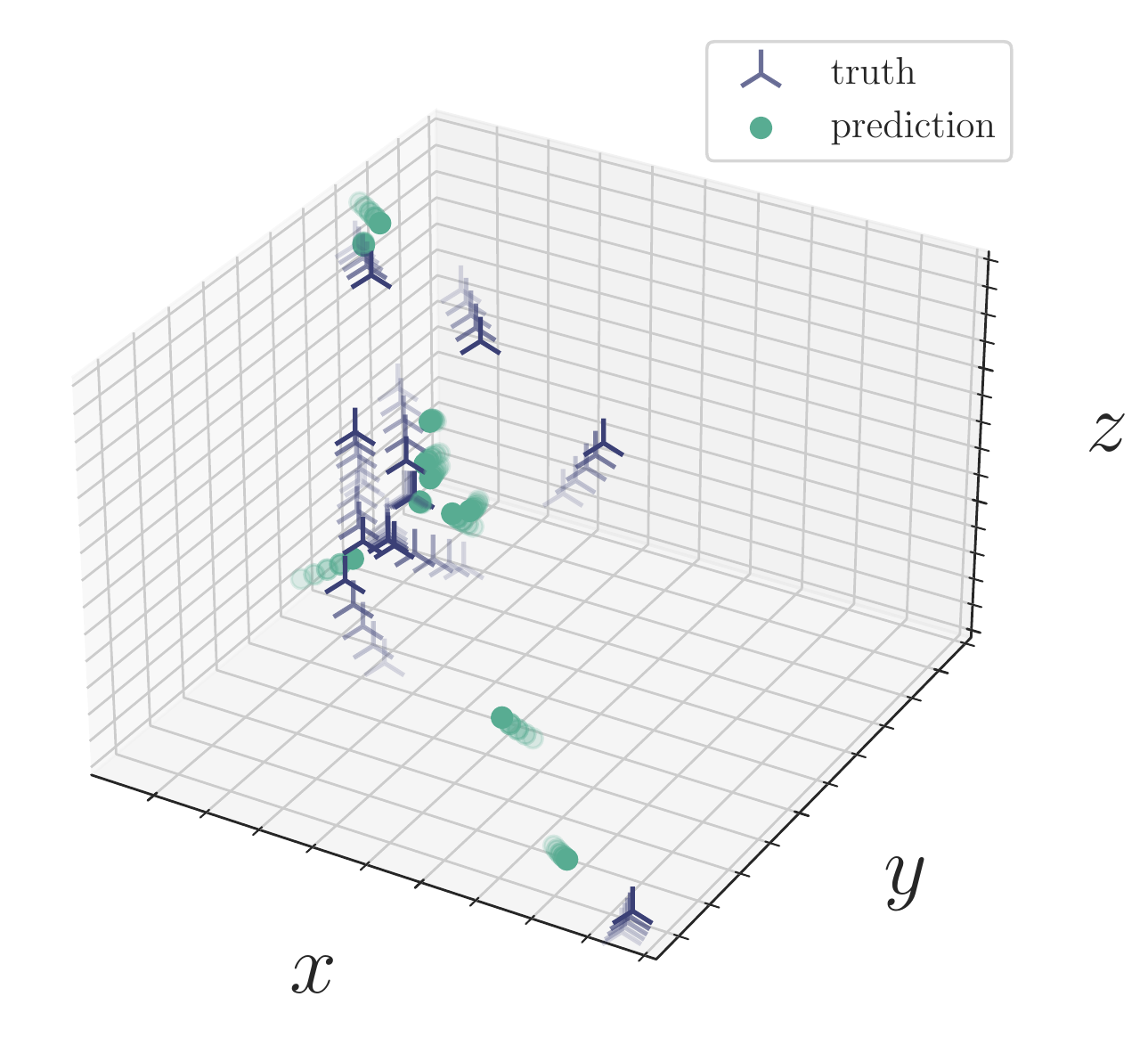}%
        \put(5, 75){\textbf{D}}
        \end{overpic}\\\\
        \multicolumn{2}{c}{%
        \begin{overpic}[width=.475\textwidth]{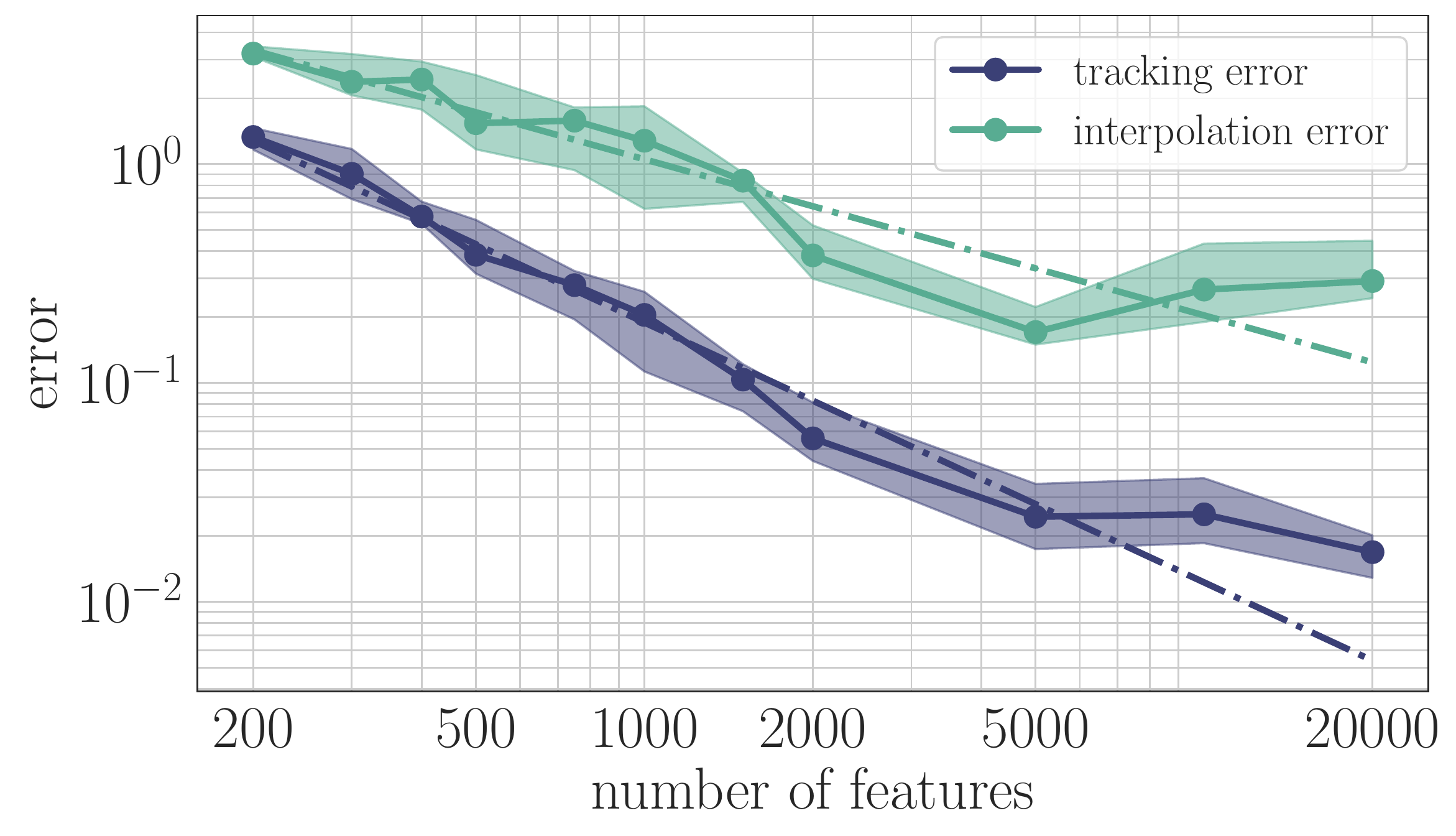}%
        \put(5, 60){\textbf{E}}
        \end{overpic}}
    \end{tabular}
    \caption{\textbf{Adaptive prediction.} All adaptive trajectories use $K=2500$ unless otherwise stated.
    (A/B) Prediction error (A) and dynamics interpolation error (B) over time. Both errors smoothly decreases to a ball around zero.
    (C/D) Example position prediction trajectory for the particles with learning (C) and without learning (D). Low-opacity fade denotes particle trajectories in time. The learned system accurately predicts the ground truth, while the system without learning fails to accurately capture the particle motion.
    (E) Prediction and interpolation errors at $t=100$ as a function of the number of features $K$ (solid: median, error bars: 20\% / 80\% quantiles over $10$ trials per $K$ value). As $K$ increases, the asymptotic prediction error decreases as a power law $\sim K^{-\xi}$, and the interpolation error decreases as a distinct power law $\sim K^{-\zeta}$. Best-fit power laws obtained via nonlinear least squares are shown in dashed with $\xi \approx 1.28 \pm 0.03$ and $\zeta \approx 0.77 \pm 0.03$.}
    \label{fig:prediction}
\end{figure}

\paragraph{Hamiltonian estimation} Let $\hat{q}_i \in \R^d$ and $\hat{p}_i \in \R^d$ denote estimates of the coordinates and momenta of the masses. Similar to the original offline method developed in~\cite{bottou_hamiltonian} and the online approach due to~\cite{boffi_neco_imp_reg}, consider learning a model of the Hamiltonian $\widehat{H}\left(\left\{\hat{p}_i\right\}_{i=1}^m, \left\{\hat{q}_i\right\}_{i=1}^m, t\right)$ by evolving the state estimates according to
\begin{align*}
    \dot{\hat{q}}_i &= \frac{\partial \widehat{H}}{\partial \hat{p}_i}\left(\left\{\hat{p}_i\right\}_{i=1}^m, \left\{\hat{q}_i\right\}_{i=1}^m, t\right) + k\cdot\left(q_i(t) - \hat{q}_i\right),\\
    \dot{\hat{p}}_i &= -\frac{\partial \widehat{H}}{\partial \hat{q}_i}\left(\left\{\hat{p}_i\right\}_{i=1}^m, \left\{\hat{q}_i\right\}_{i=1}^m, t\right) + k\cdot\left(p_i(t) - \hat{p}_i\right),
\end{align*}
where $k > 0$ denotes a measurement gain and where $q_i(t)$ and $p_i(t)$ denote measurements of the true system state. The error signals $\tilde{q}(t) = \hat{q}_i(t) - q_i(t)$ and $\tilde{p}(t) = \hat{p}_i(t) - p_i(t)$ can be used to update the Hamiltonian estimate $\widehat{H}$ until $\hat{q}_i(t)$ and $\hat{p}_i(t)$ become consistent with $q_i(t)$ and $p_i(t)$.

\paragraph{Symplectic kernel} Define the symplectic matrix 
\begin{equation*}
    J = \begin{pmatrix} 0 & I \\ -I & 0\end{pmatrix}, \:\:\: \text{so that} \: \: \: \dot{x} = J\nabla_x H(x)
\end{equation*} 
and let $\hat{x} = (\hat{q}^\T, \hat{p}^\T)^\T$ with $\hat{q}\in\R^{md}$ and $\hat{p}\in\R^{md}$ the stacked vectors of $\hat{q}_i$ and $\hat{p}_i$ over $i$. We search for the Hamiltonian estimate $\widehat{H}$ over an RKHS $\calH_{\mathsf{k}}$ corresponding to a scalar-valued translation-invariant kernel $\mathsf{k}:\R\rightarrow\R$. Similar to the curl-free kernel seen in Section~\ref{sec:rf:examples}, we define the \textit{symplectic kernel}
\begin{equation}
    \label{eqn:ham_kernel}
   \kernel(x, y) = -J\nabla^2 \mathsf{k}(x-y)J^\T,
\end{equation}
which describes the RKHS corresponding to the dynamics $J\nabla_{\hat{x}} \widehat{H}(\hat{x})$ for $\widehat{H}\in\calH_{\mathsf{k}}$. Taking $\mathsf{k}(\cdot)$ to be the Gaussian kernel, we may write \eqref{eqn:ham_kernel} as
\begin{equation*}
   \kernel(x, y) = -J\E[ww^\T\cos(w^\T x + b)\cos(w^\T x + b)]J^\T
\end{equation*}
with the expectation taken over $w \sim \mathsf{N}(0, \sigma_w^2 I)$ and  $b\sim\mathsf{Unif}(0, 2\pi)$. Let $\Psi:\R^{2md} \rightarrow \R^K$ denote a vector of random features. We may take each component $\Psi_i(\hat{x}) = \cos(w_i^\T \hat{x} + b_i)$ with the $(w_i, b_i)$ i.i.d. samples and write, for $\gamma > 0$ a learning rate,
\begin{align*}
    \widehat{H}(\hat{x}, t) &= \Psi(\hat{x})^\T \hat{\alpha}(t),\\
    \dot{\hat{\alpha}}(t) &= -\gamma \left(\left[\nabla_{\hat{p}}\Psi(\hat{x})\right]^\T \tilde{q}(t) - \left[\nabla_{\hat{q}}\Psi(\hat{x})\right]^\T\tilde{p}(t)\right).
\end{align*}

\paragraph{Results (Figure~\ref{fig:prediction})} 
We consider the sixty-dimensional ten body problem ($m=10$) in three dimensions ($d=3$).
With $K=2500$ features, the prediction and interpolation errors for the positions $\hat{q}(t) - q(t)$, momenta $\hat{p}(t) - p(t)$, and corresponding dynamics $\nabla_{\hat{p}}\widehat{H}$ and $-\nabla_{\hat{q}}\widehat{H}$ are driven to a small ball around zero (Figure~\ref{fig:prediction}A/B).
In the early stages of learning ($t \lesssim 5$), the trajectory prediction oscillates around the target trajectory. As learning proceeds, the prediction becomes smoother and accurately tracks the true system trajectory (Figure~\ref{fig:prediction}C).
As the number of random features $K$ increases, the sizes of the asymptotic balls in both the prediction and interpolation errors decrease as a power law in $K$ (Figure~\ref{fig:prediction}D). 

\paragraph{Power law exponents} Let $\xi$ denote the exponent $\limsup_{t\rightarrow\infty} \norm{\hat{x}(t) - x(t)}_2 \sim K^{-\xi}$ in the power law for the prediction error, and let $\zeta$ denote an analogous quantity for the interpolation error. 
Nonlinear least-squares fits lead to estimates ($\pm$ denotes 95\% confidence intervals) $\xi \approx 1.28 \pm 0.03$ and $\zeta \approx 0.77 \pm 0.03$.
For the adaptive predictor considered in this section, a Lyapunov function for the nominal error dynamics is the quadratic $V(t) = \norm{e(t)}_2^2$. 
Moreover, due to the feedback term $k\cdot\left(x(t) - \hat{x}(t)\right)$, the nominal dynamics is exponentially stable with rate $k$, and we may take $\rho(\norm{e}_2) \propto \norm{e}_2^2$.
This setting was considered in Example~\ref{ex:rf_adaptive} and leads to the analytical predictions $\xi= 1/2$ and $\zeta = 1/4$, where $\zeta = \xi/2$ follows after an application of Theorem~\ref{thm:ac_approx_interp}. 
The rates we obtain empirically are 
faster than the $\mathcal{O}\left(1/\sqrt{K}\right)$ Monte-Carlo rate for random feature approximations
predicted by our theory.
One plausible explanation for this observation is that more features are required to see the $\mathcal{O}\left(1/\sqrt{K}\right)$ tail behavior, as suggested by the flattening of the curve observed near $K \approx 20,000$.

\subsection{Adaptive control with deep neural networks}
\begin{figure}[!t]
    \centering
    \begin{tabular}{cc}
        \begin{overpic}[width=.475\textwidth]{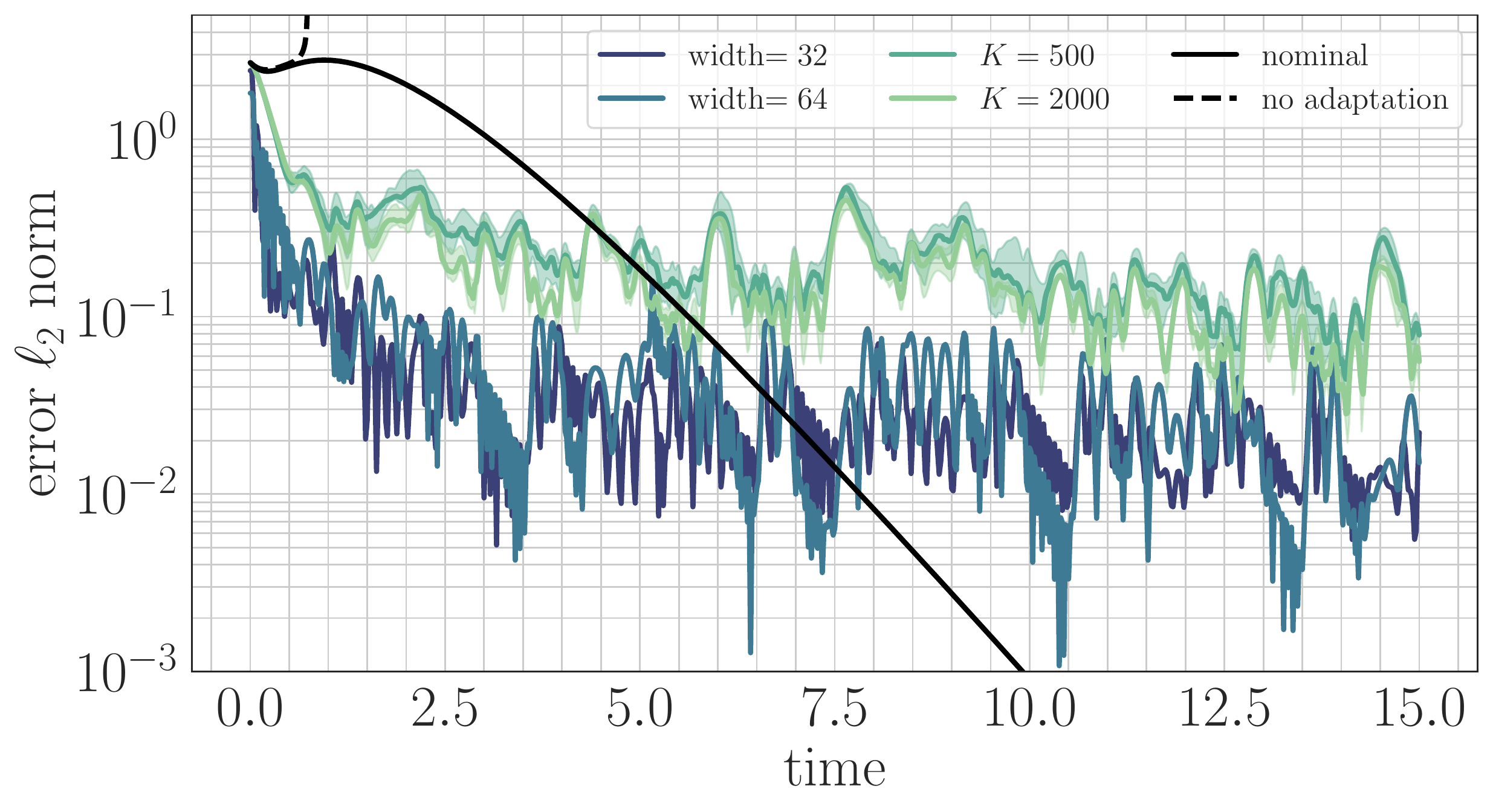}%
        \put(5, 55){\textbf{A}}
        \end{overpic}&  
        \begin{overpic}[width=.475\textwidth]{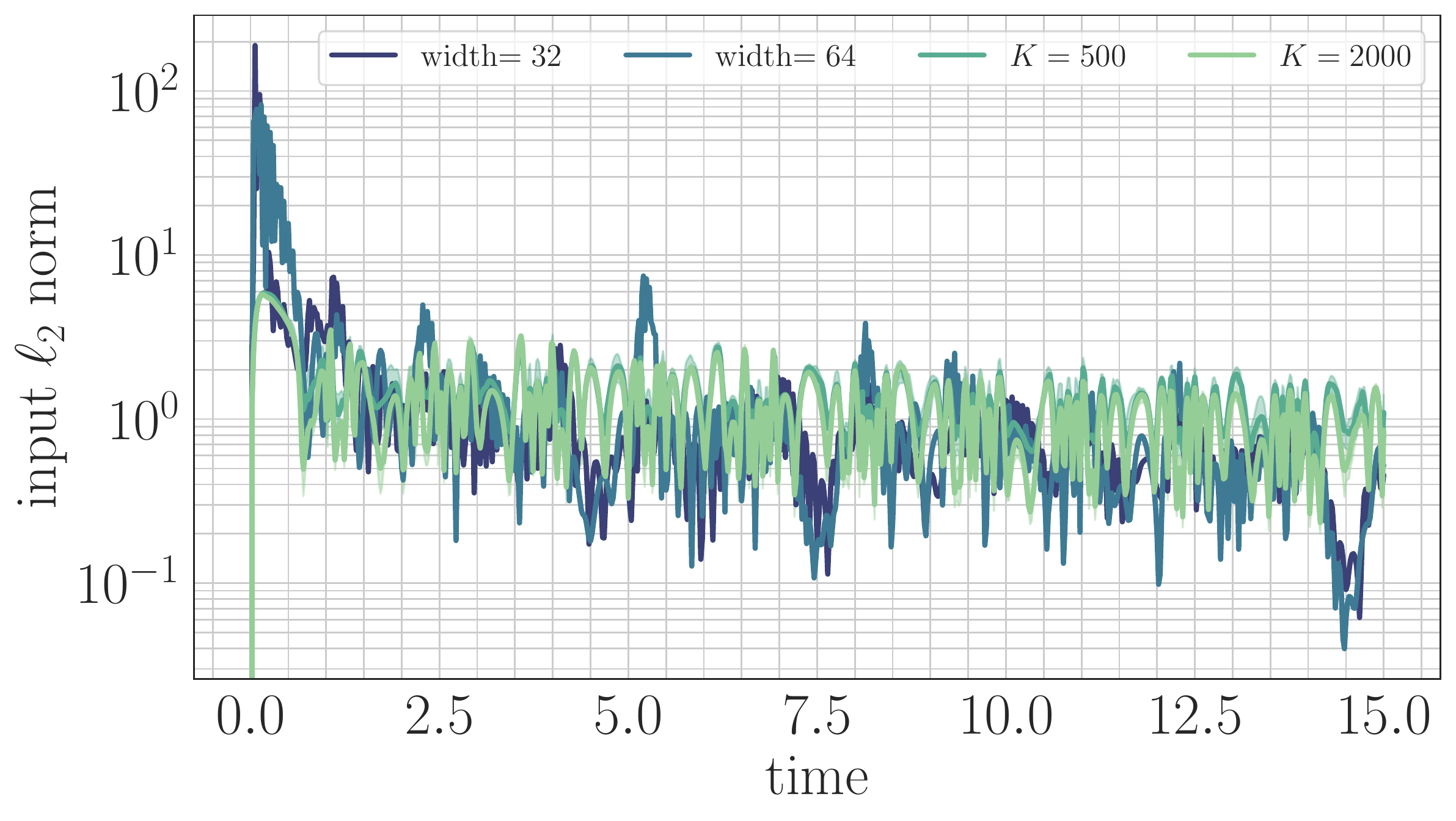}%
        \put(0, 55){\textbf{B}}
        \end{overpic}\\
        \begin{overpic}[width=.475\textwidth]{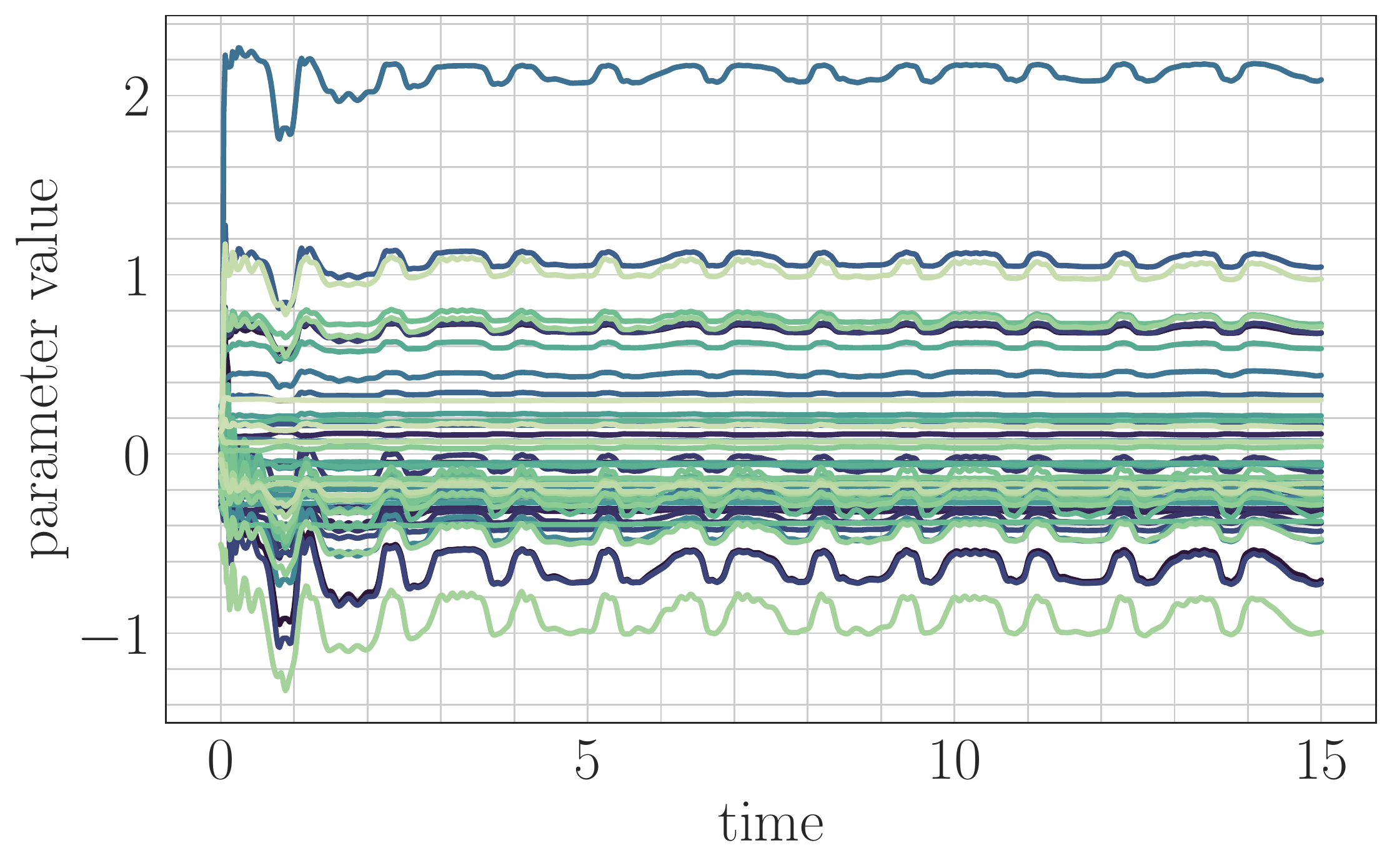}%
        \put(5, 60){\textbf{C}}
        \end{overpic}&  
        \begin{overpic}[width=.475\textwidth]{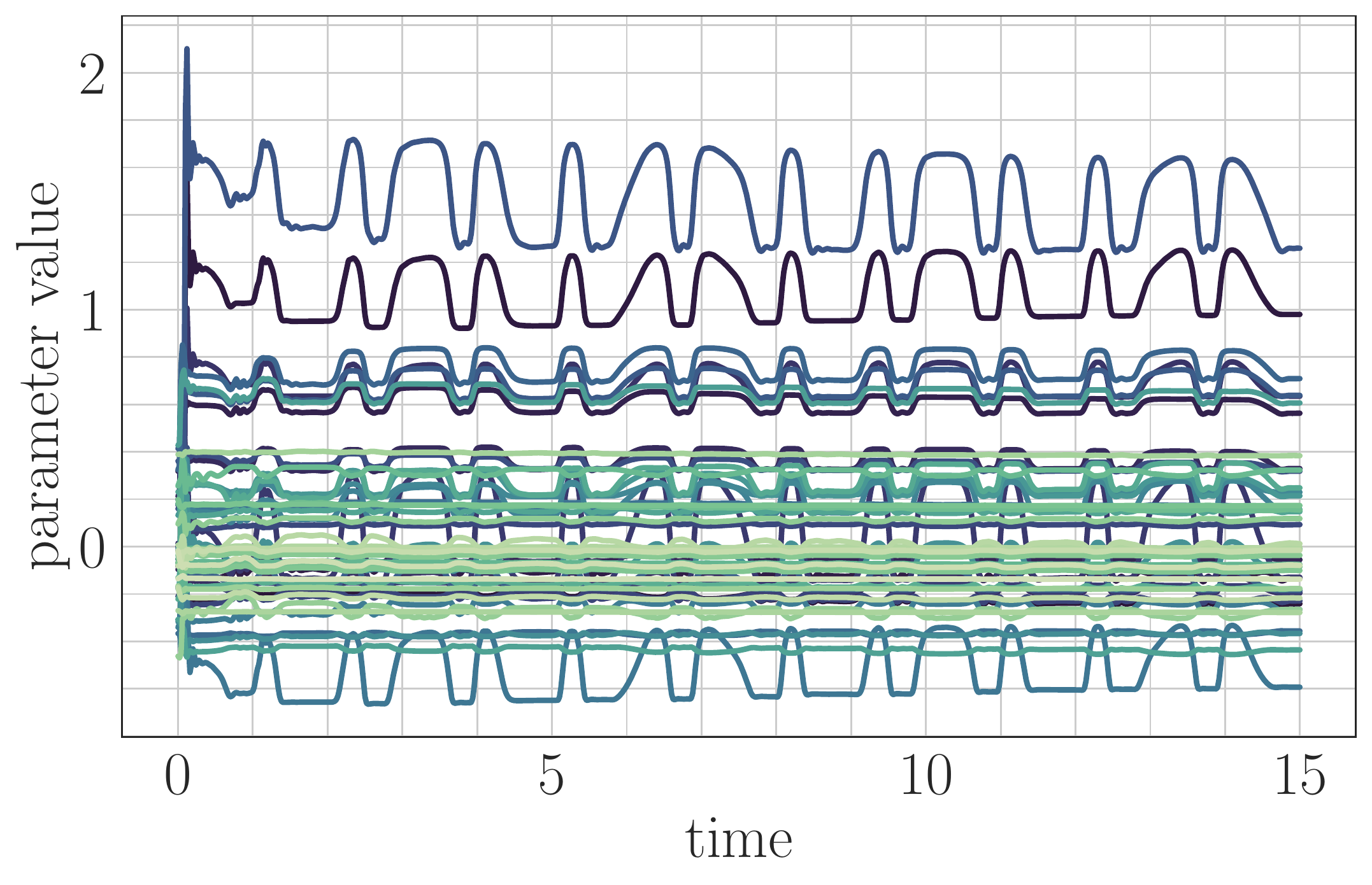}%
        \put(0, 60){\textbf{D}}
        \end{overpic}
    \end{tabular}
    \caption{\textbf{Adaptive control with multilayer networks.} (A) Tracking error $\norm{x(t) - x_d(t)}_2$ for the system without adaptation, the nominal dynamics, an adaptive system with a neural network approximation, and an adaptive system with a random feature approximation. The system without adaptation is driven unstable, while the adaptive systems with neural network and random feature approximations both regulate the actual trajectory to a ball around the desired trajectory. (B) Input norm $\norm{u(x(t), t)}_2$. The neural network system undergoes an initial transient with large input. (C/D) Linear (C) and hidden layer (D) weights over time for the multilayer representation with width of $32$ neurons. The hidden layer weights change significantly from their initialization, indicating that the network is operating outside of the kernel regime.}
    \label{fig:network_control}
\end{figure}

The preceding sections established the adaptive controllability of high-dimensional systems via randomized approximations of kernel machines. 
These random feature methods can be viewed as linearizations of neural networks~\citep{ghorbani_linearized_2020, jacot18ntk}, and have been shown to suffer from the curse of dimensionality when the target function only depends on a few relevant directions in the input space~\citep{bach_breaking, ghorbani_when_2020}.
Neural networks with a single hidden layer do not exhibit the same difficulties, which raises the question if it is possible to use deep neural networks for adaptive control.

\paragraph{A gradient flow} The Lyapunov-based adaptive law~\eqref{eqn:linear_update_standard} that forms the basis for the learning rules in Theorems~\ref{thm:ac_finite_approx}~\&~\ref{thm:dp_finite_approx} may also be written as the gradient flow~\citep{fradkov99}
\begin{equation}
    \label{eqn:velocity_gradient}
    \dot{\hat\alpha}(t) = -\gamma\nabla_{\hat{\alpha}}\dot{Q}(e(t), \hat{\alpha}(t), t).
\end{equation}
In~\eqref{eqn:velocity_gradient}, $\dot{Q}(e(t), \hat\alpha, t)$ denotes the time derivative of the Lyapunov function $Q$ for the nominal error dynamics along the actual error trajectory $e(t)$.
This formulation of the method shows that the parameters are updated to promote stability by enforcing negativity of $\dot{Q}(e(t), \hat\alpha(t), t)$.
Importantly,~\eqref{eqn:velocity_gradient} leads to a simple algorithm for non-linearly parameterized function approximators such as neural networks.
Choosing as adaptive input $u(x, t) = \varphi(x, t, \hat{\alpha}(t))$ a neural network and omitting time arguments for brevity,~\eqref{eqn:velocity_gradient} becomes
\begin{equation}
    \label{eqn:velocity_gradient_nlin}
    \dot{\hat{\alpha}} = -\gamma \left(\nabla_{\hat{\alpha}}\varphi(x, t, \hat\alpha)\right)^\T g_e(x, t)^\T \nabla Q(e, t).
\end{equation}
In~\eqref{eqn:velocity_gradient_nlin}, the linear basis functions $Y(x, t)$ from~\eqref{eqn:linear_update_standard} have been replaced by the Jacobian of the neural network evaluated at the current parameter estimates.
 
In general, it is challenging to obtain practical theoretical guarantees for~\eqref{eqn:velocity_gradient_nlin} due to the nonconvexity, time-dependence, and feedback properties of the resulting online optimization problem.
Nevertheless, for expressive classes of functions that empirically exhibit benign optimization landscapes such as neural networks, it is plausible that there exists a target set of parameters $\alpha$ that can render $\dot{Q}(e(t), \alpha, t)$ negative definite and that these parameters can be found via gradient-based optimization.

\paragraph{An unstable system} To test the adaptive law~\eqref{eqn:velocity_gradient_nlin}, we study a more difficult variant of~\eqref{eq:adaptive_control_exp} with a time-varying desired trajectory and an unknown dynamics that renders the system unstable in the absence of adaptation,
\begin{equation}
    \label{eq:adaptive_control_exp_net}
    \begin{aligned}
        \dot{x} &= A(x - x_d(t)) + \dot{x}_d(t) + u(x, t) - h(x), \:\:\: x \in \R^5,\\
        x_d(t) &= \sin\left(2\pi t + \cos\left(\sqrt{2}\pi t\right)\right),\\
        h_i(x) &= \frac{1}{4}x_i^4,
    \end{aligned}
\end{equation}
where $A$ is a known stable matrix and $x_d(t)$ denotes the desired trajectory. We take $Q(x, t) = \frac{1}{2}\left(x - x_d(t)\right)^\T P \left(x - x_d(t)\right)$ with $A^\T P + P A = -I$ as in Section~\ref{sec:exp:control}. We consider single hidden-layer neural networks with width of $32$ or $64$ neurons and the \texttt{swish} activation function. For comparison, we use the same random Fourier feature approximation of the Gaussian kernel as in Section~\ref{sec:exp:control}. We set $\gamma = 20$ for the random feature adaptation law and $\gamma = 10$ for the neural network.

\paragraph{Results (\Cref{fig:network_control})} Both the random feature and neural network representations effectively stabilize the system and regulate the actual trajectory to a ball around the desired trajectory (Figure~\ref{fig:network_control}A). 
The neural network obtains slightly improved performance over the random feature method for both choices of the width despite having similar or fewer parameters. 
This can be traced to learning of the hidden-layer weights, which indicates that the network is operating outside of the kernel regime (Figure~\ref{fig:network_control}C/D).

In exchange for this improved performance, we find that the neural network adaptation law~\eqref{eqn:velocity_gradient_nlin} is significantly more brittle to choice of hyperparameters: while the random feature approximation is provably stable for any choice of learning rate, the neural network adaptation law empirically renders the system unstable for many choices of learning rate.
Moreover, we find that the stable range of learning rates depends on the network architecture. Increasing the network depth without careful tuning of the learning rate often leads to instability.
Signatures of this phenomenon can be seen in the input norm $\norm{u(x, t)}_2$ even for trajectories that remain stable (Figure~\ref{fig:network_control}B), where the magnitude of the neural network input is seen to exceed that of the random feature input by one or two orders of magnitude during an initial transient.

\paragraph{Discussion} These observations are consistent with both the theory presented in this work and the approximation properties of neural networks. 
Neural networks, in principle, can perform better than kernel methods due to greater expressivity.
Nevertheless, due to the worst-case difficulty of the corresponding online nonconvex optimization, a stability proof as provided in this work for kernel methods is likely out of reach.
For adaptive control systems where stability of the closed-loop dynamics is necessary, these considerations may render kernel methods a more desirable choice than deeper architectures.
Nevertheless, understanding if the adaptive law~\eqref{eqn:velocity_gradient_nlin} can be modified to ensure stability of the closed-loop dynamics -- or, to the same end, if a neural network-based adaptive system can be augmented with a kernel-based approach -- are interesting directions of future research.
\section{Conclusions and future directions}
In this work, we introduced a novel nonparametric method for adaptive control and prediction that estimates the unknown dynamics over a reproducing kernel Hilbert space. By restricting to the space $\calF_2$, we analyzed efficient finite-dimensional randomized approximations that scale well to high dimension.
A promising future direction of work is to study the Banach space $\mathcal{F}_1$ of single-layer neural networks of the form $h(\cdot) = \int_{\Theta} \Phi(\cdot, \theta)\mu(d\theta)$ for a signed Radon measure $\mu$~\citep{bach_breaking, bengio_convex}.
The space $\calF_1$ admits convergence analyses for gradient-based optimization via the theory of Wasserstein gradient flows, as well as efficient approximation via particle methods~\citep{MeiE7665, rotskoff2019trainability}. Such approaches could in principle be generalized to the adaptive control setting considered here via the gradient flow algorithm~\eqref{eqn:velocity_gradient_nlin}.

\section{Acknowledgments}
NMB thanks Eric Vanden-Eijnden and Joan Bruna for many instructive discussions on the function spaces $\calF_1$ and $\calF_2$. All authors thank Pannag Sanketi and Vikas Sindhwani for helpful feedback.

\bibliography{paper}

\appendix
\section{Discrete sampling}
\label{app:sample}
Assume that measurements of the true system state $\left\{x(t_i)\right\}_{i=0}^{\infty}$ are received at potentially non-uniformly spaced intervals $t_i = t_0 + \sum_{i'=0}^{i-1} \Delta t_{i'}$. Denote $\hat{x}_i = \hat{x}(t_i)$ and let $\phi_{t_i+\Delta t_i}(\hat{x}_i)$ denote the flow from time $t_i$ to time $t_{i+1} = t_i + \Delta t_i$ of the
system $\dot{\hat{x}} = f(\hat{x}, t)$ starting at $\hat{x}(t_i) = \hat{x}_i$. We are interested in the contraction properties of the hybrid system
\begin{align*}
    \hat{x}_{i+1/2} &= \phi_{t_i + \Delta t_i}(\hat{x}_i), \:\: \hat{x}_{i+1} = k_i(\hat{x}_{i+1/2}, x_{i+1}),
\end{align*}
where $x_{i+1}$ denotes the measurement $x(t_{i+1})$. The following result is similar to~\citet[Eq. 6]{process_control}.
\begin{restatable}[]{myprop}{sample}
\label{prop:disc_meas}
Suppose that there exists some $\Theta : \R^n\times\R_{\geq 0} \rightarrow \R^{n\times n}$ and $0 < \beta < 1$ such that $y_{i+1} = k_i(y_i, x)$ is contracting as a discrete-time dynamical system with rate $\beta$ for any $x$, i.e.,
\begin{equation*}
    F_i := \Theta(y_{i+1}, t_{i+1})\frac{\partial k_i}{\partial y}(y_i, x)\Theta(y_i, t_i)^{-1}, \:\: F_i^\T F_i \preccurlyeq \beta I.
\end{equation*}
Assume that $k_i(x, x) = x$ for all $x \in \R^n$, and denote by
\begin{align*}
    \bar{\lambda}_i &= \sup_{t\in [t_i, t_{i+1}]}\lambda_{\max}\Bigg\{\Symm\left(\dot{\Theta}(\hat{x}(t), t) + \Theta(\hat{x}(t), t)\frac{\partial \hat{f}}{\partial \hat{x}}(\hat{x}(t), t)\Theta(\hat{x}(t), t)^{-1}\right)\Bigg\}
\end{align*}
the maximum expansion rate of the open loop dynamics between $t_i$ and $t_{i+1}$ in the metric $M(\hat{x}, t) = \Theta(\hat{x}, t)^\T \Theta(\hat{x}, t)$. Then the Riemannian energy in the metric $M(\hat{x}, t)$ obeys
\begin{equation*}
    E(\hat{x}_{i+1}, x_{i+1}) \leq \beta e^{\bar{\lambda}_i \Delta t_i} E(\hat{x}_i, x_i).
\end{equation*}
\end{restatable}
\begin{proof}
Let $M_{i}(\cdot) = M(\cdot, t_{i})$.
Let $t_{i+1/2} = t_i + \Delta t_i^{-}$ denote the instant before the measurement. Let $\gamma^{i+1}:[0, 1] \rightarrow \R^n$ denote a geodesic in the metric $M_{i+1}(\cdot)$ between $\hat{x}_{i+1}$ and $x_{i+1}$,
and let $\gamma^{i+1}_s = \frac{d}{ds} \gamma^{i+1}$.  The Riemannian energy under the metric $M_{i+1}$ is then
\begin{align*}
    E(\hat{x}_{i+1}, x_{i+1}) &= \int_0^1 \gamma_s^{i+1}(s)^\T M_{i+1}\left(\gamma^{i+1}(s)\right)\gamma_s^{i+1}(s) ds.
\end{align*}
Now let $\gamma^{i+1/2}: [0, 1]\rightarrow \R^n$ denote a geodesic in the metric $M_{i+1/2}(\cdot)$ between $\hat{x}_{i+1/2}$ and $x_{i+1}$. Observe that because $k_i(x, x) = x$ for all $x$, $\hat{x}_{i+1/2} = x_{i+1}$ is a fixed point. Then, by contraction of $k(\hat{x}, x)$ in $\hat{x}$ with rate $0 < \beta < 1$ (cf.~\citet{lohmiller98contraction},~\citet{pham08}):
\begin{align*}
    \int_0^1 \gamma_s^{i+1}(s)^\T M_{i+1}\left(\gamma^{i+1}(s)\right)\gamma_s^{i+1}(s) \:ds \leq \beta \int_0^1 \gamma_s^{i+1/2}(s)^\T M_{i+1/2}\left(\gamma^{i+1/2}(s)\right)\gamma_s^{i+1/2}(s) \:ds.
\end{align*}
Let $\gamma^i : [0, 1]\rightarrow \R^n$ denote a geodesic in the metric $M(\hat{x}, t)$ between $\hat{x}_i$ and $x_i$, let $\psi_t(s) = \Theta\left(\gamma^{i}(s), t\right)\gamma_s^{i}(s)$, and define 
\begin{align*}
    J_t(s) &= \dot{\Theta}\left(\gamma^{i}(s), t\right) + \Theta\left(\gamma^{i}(s), t\right)\frac{\partial f}{\partial x}\left(\gamma^{i}(s), t\right)\Theta\left(\gamma^{i}(s), t\right)^{-1}.
\end{align*}
Observe that $\frac{d}{dt}\psi_t(s) = J_t(s)\psi_t(s)$. Hence,
\begin{align*}
    \frac{d}{dt}\left[\gamma_s^{i}(s)^\T M\left(\gamma^{i}(s), t\right)\gamma_s^{i}(s)\right] &= 2 \psi_t(s)^\T J_t(s) \psi_t(s)\\
    &\leq \bar{\lambda} \psi_t(s)^\T \psi_t(s)\\
    &= \bar{\lambda}_i \gamma_s^{i}(s)^\T M\left(\gamma^{i}(s), t\right)\gamma_s^{i}(s).
\end{align*}
Then, by the comparison lemma,
\begin{align*}
    \gamma_s^{i+1/2}(s)^\T M_{i+1/2}\left(\gamma^{i+1/2}(s)\right)\gamma_s^{i+1/2}(s) &\leq e^{\bar{\lambda}_i \Delta t_i}\gamma_s^{i}(s)^\T M_i\left(\gamma^{i}(s)\right)\gamma_s^{i}(s).
\end{align*}
Plugging this in to our previous bound,
\begin{align*}
    E(\hat{x}_{i+1}, x_{i+1}) &\leq \beta e^{\bar{\lambda}_i \Delta t_i}\int_0^1 \gamma_s^{i}(s)^\T M_i\left(\gamma^{i}(s)\right)\gamma_s^{i}(s) \:ds.
\end{align*}
Observing that $\int_0^1 \gamma_s^{i}(s)^\T M\left(\gamma^{i}(s), t\right)\gamma_s^{i}(s)ds = E(\hat{x}_i, x_i)$ completes the proof.
\end{proof}

\section{Preliminary results}

Let $E$ and $E'$ be normed vector spaces.
Denote by $\calL(E, E')$ the space
of linear operators from $E$ to $E'$ 
equipped with the operator norm.
A function $f(x, t)$ mapping $E \times \R_{\geq 0} \mapsto F$
with $E$ and $F$ normed vector spaces is said to be
locally bounded in $x$ if for every $R > 0$ and $T > 0$, 
\begin{align*}
    \sup_{t \in [0, T]} \sup_{\norm{x}_E \leq R} \norm{f(x, t)}_{F} < \infty.
\end{align*}

\begin{myprop}
\label{prop:local_composition}
Let $\{E_i\}_{i=1}^{2}$, $\{F_i\}_{i=1}^{2}$ be normed vector spaces and let
$f_i : E_i \times \R_{\geq 0} \rightarrow F_i$ for $i \in \{1, 2\}$
be locally 
Lipschitz. Then the following hold
\begin{enumerate}[(i)]
\item If $E_1 = E_2$ and $F_1 = F_2$, then the sum
$(x, t) \mapsto f_1(x, t) + f_2(x, t)$
is locally Lipschitz.
\item If $E_1 = E_2$, $F_1 = \calL(F_2, F_3)$,
and both $f_1$ and $f_2$ are locally bounded,
then the 
product $(x, t) \mapsto f_1(x, t) f_2(x, t)$ is locally 
Lipschitz and locally bounded.
\item If $F_1 = E_2$ and both $f_1$ and $f_2$ are locally bounded, then the composition $(x, t) \mapsto f_2(f_1(x, t), t)$ is locally Lipschitz and locally bounded.
\end{enumerate}
\end{myprop}
\begin{proof}
Let $R$ and $T$ be arbitrary positive constants.
Let $C = C(R, T) > 0$ be a finite positive constant such that:
\begin{align*}
    \sup_{t \in [0, T]} \norm{f_1(x, t) - f_1(y, t)}_{F_1} &\leq C \norm{x-y}_{E_1} \:\: \forall x,y \in B_{E_1}(R), \\
    \sup_{t \in [0, T]} \norm{f_2(x, t) - f_2(y, t)}_{F_2} &\leq C \norm{x-y}_{E_2} \:\: \forall x,y \in B_{E_2}(R).
\end{align*}

Now, let $x,  y \in B_{E_1}(R)$ and $t \in [0, T]$
be arbitrary.

\paragraph{The sum $(x, t) \mapsto f_1(x, t) + f_2(x, t)$.}
Observe that
\begin{align*}
    \norm{f_1(x, t) + g_1(x, t) - (f_2(y, t) + g_2(y, t))}_{F_1} &\leq \norm{f_1(x, t) - f_1(y, t)}_{F_1} + \norm{f_2(x, t) - f_2(y, t)}_{F_1} \\
    &\leq 2C \norm{x-y}_{E_1}.
\end{align*}

\paragraph{The product $(x, t) \mapsto f_1(x, t) f_2(x, t)$.}
Let $C' = C'(R, T) > 0$ be a finite positive constant such that:
\begin{align*}
    \sup_{t \in [0, T]} \sup_{\norm{x}_{E_1} \leq R} \max\{\norm{f_1(x, t)}_{\calL(F_2, F_3)}, \norm{f_2(x, t)}_{F_2}\} \leq C'.
\end{align*}
Now observe that
\begin{align*}
    \norm{f_1(x, t)f_2(x, t) - f_1(y, t)f_2(y, t)}_{F_3} &\leq \norm{f_1(x, t)-f_1(y,t)}_{\calL(F_2, F_3)}\norm{f_2(x,t)}_{F_2} \\
    &\qquad + \norm{f_1(y,t)}_{\calL(F_2, F_3)}\norm{f_2(x,t)-f_2(y,t)}_{F_2} \\
    &\leq 2C C' \norm{x-y}_{E_1}.
\end{align*}
The fact that the composition is locally bounded is immediate.

\paragraph{The composition $(x, t) \mapsto f_2(f_1(x, t), t)$.}
First, let $C' = C'(R, T)$ be such that:
\begin{align*}
    \sup_{t \in [0, T]} \sup_{\norm{x}_{E_1}\leq R} \norm{f_1(x, t)}_{F_1} \leq C'.
\end{align*}
Next, let $C'' = C''(R, T)$ be such that:
\begin{align*}
    \sup_{t \in [0, T]} \norm{f_2(x, t) - f_2(y, t)}_{F_2} \leq C'' \norm{x - y}_{E_2} \:\: \forall x, y \in B_{E_2}(C').
\end{align*}
Then we have:
\begin{align*}
    \norm{f_2(f_1(x, t), t) - f_2(f_1(y, t), t)}_{F_2} \leq C'' \norm{f_1(x, t) - f_1(y, t)}_{F_1} \leq C C'' \norm{ x-y}_{E_1}
\end{align*}
This shows that the composition is locally Lipschitz.
The fact that the composition is locally bounded is immediate.

\end{proof}

Now, let $E$ and $F$ be normed vector spaces and
let $U \subseteq E$.
A function $ f: U \rightarrow F$ is said to be globally
Lipschitz (uniformly Lipschitz) if
\begin{align*}
    \sup_{x, y \in U, x \neq y} \frac{\norm{f(x) - f(y)}_F}{\norm{x-y}_E} < \infty.
\end{align*}
A function $f : U \rightarrow F$ is said to be globally bounded
(uniformly bounded) if:
\begin{align*}
    \sup_{x \in U} \norm{f(x)}_F < \infty.
\end{align*}

\begin{myprop}
\label{prop:global_lipschitz_composition}
Let $\{E_i\}_{i=1}^{2}$ and $\{F_i\}_{i=1}^{2}$ be collections of normed vector spaces.
Let $f_i : U_i \rightarrow F_i$ with $U_i \subseteq E_i$ for $i \in \{1, 2\}$ be globally Lipschitz. Then the following hold
\begin{enumerate}
    \item If $E_1 = E_2$ and $F_1 = F_2$, then the sum $x \mapsto f_1(x) + f_2(x)$ is globally Lipschitz.
    \item If $E_1 = E_2$, $F_1 = \calL(F_2, F_3)$, and both $f_1$ and $f_2$ are globally bounded, then the product $x \mapsto f_1(x) f_2(x)$ is globally Lipschitz and globally bounded.
    \item If $F_1 = E_2$ and both $f_1$ and $f_2$ are globally bounded, then the composition $(x, t) \mapsto f_2(f_1(x))$ is globally Lipschitz and globally bounded.
\end{enumerate}
\end{myprop}
\begin{proof}
Nearly identical proof as Proposition~\ref{prop:local_composition}
\end{proof}

The following result generalizes Barbalat's lemma to the case when the limiting
value of a function $f$ only converges to a ball. We first state Barbalat's lemma, and then state our generalization.
\begin{myprop}[Barbalat's lemma]
Let $f \in C^1(\R_{\geq 0}, \R)$ satisfy $\lim_{t\rightarrow \infty} f(t) < \infty$. Further assume that $f'$ is uniformly continuous. Then $\lim_{t\rightarrow \infty} f'(t) = 0$.
\end{myprop}
\begin{myprop}[Generalized Barbalat's lemma]
\label{prop:generalized_barbalat}
Let $f\in C^1(\R_{\geq 0}, \R)$ satisfy
$\limsup_{t \rightarrow \infty} \abs{f(t) - \alpha} \leq \varepsilon$ for some $\alpha \in \R$ and $\varepsilon \geq 0$. Further assume that $f'$ is $L$-Lipschitz.
Then,
\begin{align*}
    \limsup_{t \rightarrow \infty} \abs{f'(t)} \leq 2 \sqrt{\varepsilon L}.
\end{align*}
\end{myprop}
\begin{proof}
Suppose for a contradiction that
$\limsup_{t \rightarrow \infty} \abs{f'(t)} > 2\sqrt{\varepsilon L}$.
Then there exists an increasing sequence $\{t_n\}_{n \geq 1}$ with $t_n \rightarrow \infty$ such that $\abs{f'(t_n)} > 2\sqrt{\varepsilon L}$ for all $n \geq 1$.
Define $\delta := 2\sqrt{\varepsilon L}/L$. Then for any $n \geq 1$, we have
\begin{align*}
    \bigabs{\int_{t_n}^{t_n+\delta} f'(t) dt} &= \bigabs{ \delta f'(t_n) + \int_{t_n}^{t_n+\delta} (f'(t) - f'(t_n)) dt}, \\
    &\geq \delta \abs{f'(t_n)} - \int_{t_n}^{t_n+\delta} \abs{f'(t)-f'(t_n)} dt, \\
    &> \delta 2 \sqrt{\varepsilon L} - L \int_{t_n}^{t_n+\delta} \abs{t_n-t} dt, \\
    &= \delta 2\sqrt{\varepsilon L} - \frac{L}{2}\delta^2, \\
    &= 2\varepsilon.
\end{align*}
This lower bound implies that for any $n \geq 1$
\begin{align*}
    \abs{ f(t_n + \delta) - f(t_n) } &= \bigabs{\int_{t_n}^{t_n+\delta} f'(t) dt} > 2\varepsilon.
\end{align*}
This bound implies
\begin{align*}
    2 \varepsilon &< \limsup_{n \rightarrow \infty} \abs{f(t_n + \delta) - f(t_n)},\\
    &\leq \limsup_{t \rightarrow \infty} \abs{f(t+\delta) - f(t)},\\
    &\leq \limsup_{t \rightarrow \infty} \abs{f(t+\delta)-\alpha} + \limsup_{t \rightarrow \infty} \abs{f(t) - \alpha}, \\
    &\leq 2\varepsilon,
\end{align*}
which yields a contradiction.
\end{proof}

In adaptive control, a typical use of Barbalat's lemma
is to conclude (via deadzones) that the error signal tends to a small value. In the sequel, we will use 
Barbalat's lemma in conjunction with the
generalized Barbalat's lemma (Proposition~\ref{prop:generalized_barbalat})
to argue that both the error signal and the \emph{time derivative} of the error
signal are small. The time derivative of the error
signal can be written as a nominal term plus the 
error of the adaptive signal. By controlling
this quantity, we will be able to show that the
error of the adaptive signal is small as well, 
allowing us to prove approximate interpolation type results (Theorem~\ref{thm:ac_approx_interp}).

\paragraph{Sharpness of the bound} Proposition~\ref{prop:generalized_barbalat}
is sharp in the following sense.
Fix any $\varepsilon > 0$ and $\omega \in \R$,
and define $f(t) := \varepsilon \sin\left(\sqrt{\frac{\omega}{\varepsilon}} t\right)$.
Clearly $\limsup_{t \rightarrow \infty} \abs{f(t)} = \varepsilon$, and furthermore
\begin{align*}
    f'(t) = \sqrt{\varepsilon \omega} \cos\left(\sqrt{\frac{\omega}{\varepsilon}} t\right), \:\: f''(t) = - \omega \sin\left(\sqrt{\frac{\omega}{\varepsilon}}t\right).
\end{align*}
This shows that the smallest valid global
Lipschitz constant for $f'$ is $\omega$.
Furthermore,
\begin{equation*}
    \limsup_{t \rightarrow \infty} \abs{f'(t)} = \sqrt{\varepsilon \omega}.
\end{equation*}

\section{Omitted proofs for Section~\ref{sec:np_rslts}}

\subsection{Proof of Theorem~\ref{thm:nonparametric_conv}}
We first state the following technical lemma.
\begin{mylemma}
\label{lem:existence}
Let $E$ denote the Banach space $E := \R^n \times \R^s \times L_2(\Theta, \nu)$ equipped with the norm $\norm{(x, e, \hat{\alpha})}_E := \max\{ \norm{x}_2, \norm{e}_2, \norm{\hat{\alpha}}_{L_2(\Theta, \nu)} \}$. Write $z = (x, e, \hat{\alpha})$ for $z \in E$ and define the function $F : E \times \R_{\geq 0} \rightarrow E$ as:
\begin{align*}
    F(z, t) := \begin{bmatrix}
    f(x, t) + g(x, t)\left( \int_\Theta \Phi(x, \theta) \hat{\alpha}(\theta) d\nu(\theta) - h(x) \right) \\
    f_e(e, t) + g_e(x, t)\left( \int_\Theta \Phi(x, \theta) \hat{\alpha}(\theta) d\nu(\theta) - h(x)\right) \\
    -\gamma \Phi(x, \cdot)^\T g_e(x, t)^\T \nabla Q(e, t) 
    \end{bmatrix}.
\end{align*}
Then, under Assumption \ref{assmp:second_moment_bound}, $F(z, t)$ is locally Lipschitz in $z$ with respect to $\norm{\cdot}_E$. That is, for each $R > 0$ and $T > 0$,
letting $B_E(R) := \{ z \in E : \norm{z}_E \leq R\}$,
\begin{equation*}
    \sup_{t\in [0,T]} \sup_{z_1,z_2 \in B_E(R)} \frac{\norm{F(z_1, t) - F(z_2, t)}_E}{\norm{z_1 - z_2}_E} < \infty.
\end{equation*}
\end{mylemma}
\begin{proof}
By the composition rules for locally Lipschitz functions (cf. Proposition~\ref{prop:local_composition}),
it suffices to show that 
the functions $\psi_1 : E \rightarrow \R^d$
and $\psi_2 : E \rightarrow \mathcal{L}(\R^{d_1}, L_2(\Theta, \nu))$ defined by
\begin{align*}
    \psi_1((x, e, \hat{\alpha})) &:= \int_\Theta \Phi(x, \theta) \hat{\alpha}(\theta) d\nu(\theta), \\
    \psi_2((x, e, \hat{\alpha}))(q) &:= \Phi(x, \cdot)^\T q \:\: \forall q \in \R^{d_1}.  
\end{align*}
are locally Lipschitz and locally bounded. We view both $\psi_1$ and $\psi_2$ as functions defined on $E$, consistent with their appearance in the definition of $F(z, t)$; however, clearly $\psi_1$ is independent of $e$ and $\psi_2$ is independent of both $e$ and $\hat{\alpha}$.

Because $\psi_1$ and $\psi_2$ do not depend on time $t$,
locally Lipschitz implies locally bounded.
We first show that $\psi_1$ is locally Lipschitz.
Fix an $R > 0$ and let $z_1 = (x_1, e_1, \hat{\alpha}_1)$, $z_2 = (x_2, e_2, \hat{\alpha}_2)$ be contained in $B_E(R)$.
By Assumption~\ref{assmp:second_moment_bound},
there exists a $C = C(R) > 0$ such that the following conditions hold:
\begin{align*}
    \sup_{x \in B_2^n(R)} \int_\Theta \opnorm{\Phi(x, \theta)}^2 d\nu(\theta) &\leq C^2, \\
    \int_\Theta \opnorm{\Phi(x_1, \theta) - \Phi(x_2, \theta)}^2 d\nu(\theta) &\leq C^2 \norm{x_1 - x_2}_2^2 \:\:\forall x_1,x_2 \in B_2^n(R).
\end{align*}
By the triangle inequality and Cauchy-Schwarz,
\begin{align*}
    &\norm{ \psi_1(z_1) - \psi_1(z_2) }_2 \\
    &\leq \sqrt{\int_\Theta \opnorm{\Phi(x_1, \theta) - \Phi(x_2, \theta)}^2 d\nu(\theta)} \sqrt{ \int_\Theta \norm{\hat{\alpha}_1(\theta)}_2^2 d\nu(\theta)} \\
    &\qquad + \sqrt{\int_\Theta \opnorm{\Phi(x_1, \theta)}^2 d\nu(\theta)} \sqrt{ \int_\Theta \norm{\hat{\alpha}_1(\theta) - \hat{\alpha}_2(\theta)}_2^2 d\nu(\theta) } \\
    &\leq CR \norm{x_1-x_2}_2 + C\norm{\alpha_1 - \alpha_2}_{L_2(\Theta, \nu)} \\
    &\leq C(1+R) \norm{z_1-z_2}_E.
\end{align*}
This shows that $\psi_1$ is locally Lipschitz.
To show that $\psi_2$ is locally Lipschitz, by Cauchy-Schwarz,
\begin{align*}
    \norm{\psi_2(z_1) - \psi_2(z_2)}_{\calL(\R^{d_1}, L_2(\Theta, \nu))} &= \sup_{\norm{q}_2 = 1} \norm{ (\Phi(x_1, \cdot) - \Phi(x_2, \cdot))^\T q }_{L_2(\Theta, \nu)} \\
    &= \sup_{\norm{q}_2 = 1} \left( \int_\Theta \norm{ (\Phi(x_1, \theta) - \Phi(x_2, \theta))^\T q}_2^2 d\nu(\theta) \right)^{1/2} \\
    &\leq \left( \int_\Theta \opnorm{\Phi(x_1, \theta) - \Phi(x_2, \theta)}^2 d\nu(\theta) \right)^{1/2} \\
    &\leq C \norm{x_1 - x_2}_2 \leq C \norm{z_1 - z_2}_E.
\end{align*}
\end{proof}

We now require the following result concerned with existence and uniqueness of solutions to ordinary differential equations defined on Banach spaces. This result will be used in conjunction with Lemma~\ref{lem:existence} to assert the existence of our nonparametric input over an interval of time. Via a Lyapunov argument, we can then extend the interval to infinity.
\begin{myprop}[Existence of a maximal solution (see e.g., Proposition 11.8 of \citet{driver04book})]
\label{prop:banach_existence}
Let $E$ be a Banach space, $U$ be an open subset of $E$, $T \subseteq \R$ be an interval of time containing $0$, and $F: U\times T \rightarrow E$ be a continuous vector field on $E$. Assume that $F$ is locally Lipschitz in the following sense.
For every $x_0 \in U$ and compact $I \subseteq T$,
there exists finite positive $L = L(x_0, I)$ and $R = R(x_0, I)$
such that:
\begin{align*}
    \sup_{t \in I} \norm{f(x, t) - f(y, t)}_E \leq \norm{x - y}_E \:\: \forall x, y \in B_E(x_0, R).
\end{align*}
Then for each $x_0 \in U$, there exists a maximal interval $I(x_0) = (a(x_0), b(x_0)) \subseteq T$ with $a(x_0) \in [-\infty, 0)$ and
$b(x_0) \in (0, +\infty]$ such that the ordinary differential equation
\begin{equation*}
    \dot{x}(t) = F(x(t), t),\:\:\: x(0) = x_0
\end{equation*}
has a unique continuously differentiable solution $x: I(x_0) \rightarrow U$.
\end{myprop}

We may now state our proof of the main nonparametric theorem.
\npconv*
\begin{proof}
By Assumption~\ref{assmp:rf_kernel}, there exists a signed density $\alpha(\theta) \in L_2(\Theta, \nu)$ such that
\begin{equation*}
    h(\cdot) = \int_{\Theta}\Phi(\cdot, \theta)\alpha(\theta)d\nu(\theta), \:\: \norm{h}_{\calH}^2 = \norm{\alpha}^2_{L_2(\Theta, \nu)}.
\end{equation*}
Define the signed density $\hat{\alpha} : \Theta\times\R_{\geq 0}\rightarrow\R^{d_1}$ by
$\hat{\alpha}(\cdot, 0) = 0$ and the pointwise update for $\theta \in \Theta$,
\begin{equation*}
    \frac{\partial\hat{\alpha}}{\partial t}(\theta, t) = -\gamma \Phi(x(t), \theta)^\T g_e(x(t), t)^\T \nabla Q(e(t), t).
\end{equation*}
Observe that by Lemma~\ref{lem:existence} and Proposition~\ref{prop:banach_existence}, there exists some maximal $T_{\max} \in (0, \infty]$ such that the curve $t\mapsto (x(t), e(t), \hat{\alpha}(t))$ exists, is unique, and is continuously differentiable. Moreover, we may write the input as
\begin{align*}
    u(x, t) = \int_\Theta \Phi(x, \theta)\hat{\alpha}(\theta, t) d\nu(\theta).
\end{align*}
By means of contradiction, let us suppose that $T_{\max} < \infty$. For $t \in [0, T_{\max})$, define $\tilde{\alpha}(\theta, t) := \hat{\alpha}(\theta, t) - \alpha(\theta)$ so that
\begin{equation*}
    u(\cdot, t) - h(\cdot) = \int_{\Theta}\Phi(\cdot, \theta)\tilde{\alpha}(\theta, t)d\nu(\theta).
\end{equation*}
Now consider the Lyapunov-like function $V : [0, T_{\max}) \rightarrow \R$,
\begin{equation*}
    V(t) = Q(e(t), t) + \frac{1}{2\gamma}\norm{\tilde{\alpha}(\cdot, t)}_{L_2(\Theta, \nu)}^2.
\end{equation*}
We note that because $L_2(\Theta, \nu)$ is a real Hilbert space,
the map $u \mapsto \norm{u}_{L_2(\Theta,\nu)}^2$ is (Fr{\'{e}}chet) differentiable with derivative $h \mapsto 2 \ip{u}{h}_{L_2(\Theta, \nu)}$ Therefore, by the differentiability of the curve $t \mapsto \tilde{\alpha}(\cdot, t)$ and the chain rule, we have:
\begin{align*}
    \frac{d}{dt} \int_\Theta \norm{\tilde{\alpha}(\theta, t)}^2_2 d\nu(\theta) = 2\bigip{\tilde{\alpha}(\cdot, t)}{\frac{\partial \hat{\alpha}}{\partial t}(\cdot, t)}_{L_2(\Theta, \nu)}.
\end{align*}

Computing the time derivative, for any $t \in [0, T_{\max})$,
\begin{align*}
    \dot{V}(t) &= \frac{\partial Q}{\partial t}(e(t), t) + \nabla Q(e(t), t)^\T \left(f_e(e(t), t) + g_e(x(t), t)\left(u(x(t), t) - h(x(t))\right)\right) \\
    &\qquad + \frac{1}{\gamma}\bigip{\tilde{\alpha}(\cdot, t)}{\frac{\partial \hat{\alpha}}{\partial t}(\cdot, t)}_{L_2(\Theta, \nu)},\\
    &\leq -\rho\left(\norm{e(t)}_2\right) + \nabla Q(e(t), t)^\T g_e(e(t), t)\left(u(x(t), t) - h(x(t))\right) \\
    &\qquad + \frac{1}{\gamma}\bigip{\tilde{\alpha}(\cdot, t)}{\frac{\partial \hat{\alpha}}{\partial t}(\cdot, t)}_{L_2(\Theta, \nu)},
\end{align*}
where we have applied Assumption~\ref{assmp:lyapunov}. Now, observe that
\begin{align*}
    \bigip{\tilde{\alpha}(\cdot, t)}{\frac{\partial \hat{\alpha}}{\partial t}(\cdot, t)}_{L_2(\Theta, \nu)} = -\gamma \int_{\Theta}\ip{\tilde{\alpha}(\theta, t)}{\Phi(x(t), \theta)^\T g_e(x(t), t) \nabla Q(e(t), t)}d\nu(\theta)
\end{align*}
so that the last two terms in $\dot{V}(t)$ cancel, and hence:
\begin{equation*}
    \dot{V}(t) \leq -\rho(\norm{e(t)}_2).
\end{equation*}
Now, because $\dot{V}(t) \leq 0$ for all $t \in [0, T_{\max})$, 
$V(t) \leq V(0)$. 
Therefore, since $Q(e(t), t) \geq \mu_1(\norm{e(t)}_2)$
and $\mu_1$ is a class $\calK_\infty$ function,
\begin{align*}
    \sup_{t\in [0, T_{\max})}\norm{e(t)}_2 < \infty, \:\: \sup_{t\in [0, T_{\max})}\norm{\hat{\alpha}(\cdot, t)}_{L_2(\Theta, \nu)} < \infty.
\end{align*}
Furthermore, $\sup_{t \in [0, T_{\max})} \norm{x(t)-x_d(t)}_2 < \infty$
by requirement \eqref{eq:bounded_error_bounded_x} on the error signal,
and since $x_d$ is uniformly bounded, we also have that
$\sup_{t \in [0, T_{\max})} \norm{x(t)}_2 < \infty$.
This contradicts that $T_{\max}$ is finite, so we conclude that $T_{\max} = \infty$. This implies that 
\begin{align*}
    \sup_{t\geq 0} \max\{\norm{x(t)}_2, \norm{e(t)}_2, \norm{\tilde{\alpha}(\cdot, t)}_{L_2(\Theta,\nu)}\} < \infty,
\end{align*} 
so that $u(\cdot, t) \in \calH$ for all $t \geq 0$. This proves the first two claims. Now, integrating both sides of $\dot{V}(t)$,
\begin{equation*}
    \int_0^\infty \rho\left(\norm{e(t)}_2\right)dt \leq V(0).
\end{equation*}
To complete the proof, we now need to show that $t \mapsto \rho(\norm{e(t)}_2)$ is uniformly
continuous on $[0, \infty)$ and apply Barbalat's lemma.

We first show that $e(t)$ is uniformly Lipschitz in $t$. 
To do so, we bound $\sup_{t \geq 0} \norm{\dot{e}(t)}_2$ and apply $\norm{e(t_1) - e(t_2)}_2 \leq  \sup_{t \geq 0} \norm{\dot{e}(t)}_2 \abs{t_1 - t_2}$.
Let $C_\Phi := \sup_{t \geq 0} \left( \int_\Theta \norm{\Phi(x(t), \theta)}_2^2 d\nu(\theta) \right)^{1/2}$.
Because $x(t)$ is uniformly bounded, $C_\Phi$ is finite by Assumption~\ref{assmp:second_moment_bound}.
Next,
\begin{align*}
    \norm{ u(x(t), t) - h(x(t)) }_2 
    \leq \left( \int_{\Theta} \opnorm{\Phi(x(t), \theta)}^2 d\nu(\theta) \right)^{1/2} \norm{\tilde{\alpha}(\cdot, t)}_{L_2(\Theta, \nu)} 
    \leq C_\Phi \sqrt{2\gamma V(0)}. 
\end{align*}
Now, observe that
\begin{align*}
    \norm{\dot{e}(t)}_2 &\leq \norm{f_e(e(t), t)}_2 + \opnorm{g_e(x(t), t)} \norm{ u(x(t), t) - h(x(t))}_2 \\
    &\leq \norm{f_e(e(t), t)}_2 + \opnorm{g_e(x(t), t)} C_\Phi \sqrt{2\gamma V(0)}.
\end{align*}
Because both $f_e$ and $g_e$ are locally bounded in $x$ uniformly in $t$, $\norm{\dot{e}(t)}_2$ is uniformly bounded in $t$. 
Therefore, $t \mapsto \norm{e(t)}_2$ is uniformly Lipschitz and 
$t \mapsto \norm{e(t)}_2$ is uniformly continuous.
Now, because $\rho$ is continuous, it is uniformly continuous on
the range of $t \mapsto \norm{e(t)}_2$. Since the composition of two uniformly continuous functions remains uniformly continuous, 
$t \mapsto \rho(\norm{e(t)})$ is uniformly continuous.
By Barbalat's lemma, this implies that $\lim_{t \rightarrow \infty} \rho( \norm{e(t)}_2 ) = 0$. By continuity of $\rho$ and the fact that $\rho(a) = 0$ if and only if $a = 0$, we conclude that $\lim_{t \rightarrow \infty} \norm{e(t)}_2 = 0$.
From the requirement \eqref{eq:error_signal_asymp}
on the error signal, we conclude that $\lim_{t \rightarrow \infty} \norm{x(t)-x_d(t)}_2$.
\end{proof}

\subsection{Proof of Theorem~\ref{thm:interpolation}}

\interpolation*
\begin{proof}
Recall that the error dynamics satisfy:
\begin{align*}
    \dot{e}(t) = f_e(e(t), t) + g_e(x(t), t)(u(x(t), t) - h(x(t))).
\end{align*}
From the proof of Theorem~\ref{thm:nonparametric_conv},
$\lim_{t \rightarrow \infty} e(t) = 0$.
If we show in addition that $t \mapsto \dot{e}(t)$ is uniformly Lipschitz,
then by Barbalat's lemma (applied to each coordinate), 
$\lim_{t \rightarrow \infty} \dot{e}(t) = 0$.
Since $f_e(0, t) = 0$ and $f_e$ is locally Lipschitz in $e$ uniformly in $t$,
$\lim_{t \rightarrow \infty} \dot{e}(t) = 0$ implies
that $\lim_{t \rightarrow \infty} \norm{g_e(x(t), t)(u(x(t), t) - h(x(t)))}_2 = 0$.

It remains to show that $t \mapsto \dot{e}(t)$ is uniformly Lipschitz.
By the composition rule (cf. Proposition~\ref{prop:global_lipschitz_composition}),
it suffices to show that the functions:
\begin{align*}
    t \mapsto f_e(e(t), t), \:\: t \mapsto g_e(x(t), t), \:\: t \mapsto u(x(t), t), \:\: t \mapsto h(x(t)),
\end{align*}
are all uniformly Lipschitz and bounded.
From the proof of Theorem~\ref{thm:nonparametric_conv}, 
both $t \mapsto e(t)$ and $t \mapsto x(t)$ are uniformly bounded, and
$t \mapsto e(t)$ is uniformly Lipschitz. A nearly identical argument shows that
$t \mapsto x(t)$ is also uniformly Lipschitz.
Therefore, since $f_e$, $g_e$, and $h$ are all locally Lipschitz and locally bounded uniformly in $t$,
it is clear that $t \mapsto f_e(e(t), t)$,
$t \mapsto g_e(x(t), t)$, and $t \mapsto h(x(t))$ are all uniformly Lipschitz and uniformly bounded.

To see that $t \mapsto u(x(t), t)$ is also uniformly Lipschitz, we first 
choose a finite constant $C > 0$ such that
\begin{align*}
    \sup_{t \geq 0} \max\{ \norm{x(t)}_2, \opnorm{g_e(x(t), t)}, \norm{\nabla Q(e(t), t)}_2 \} \leq C.
\end{align*}
Now observe that
for every $\theta$ and $t$,
\begin{align*}
    \bignorm{ \frac{\partial \hat{\alpha}}{\partial t}(\theta, t) }_2 &= \gamma \bignorm{\Phi(x(t), \theta)^\T g_e(x(t), t)^\T \nabla Q(e(t), t) }_2 \\
    &\leq \gamma \opnorm{\Phi(x(t), \theta)} \opnorm{g_e(x(t), t)} \norm{\nabla Q(e(t), t)}_2 \\
    &\leq \gamma C^2 \opnorm{\Phi(x(t), \theta)}. 
\end{align*}
Put $C_\Phi := \left( \int_\Theta \sup_{\norm{x}_2 \leq C} \opnorm{\Phi(x, \theta)}^2 d\nu(\theta) \right)^{1/2}$, which is finite by assumption.
Fix $t_1, t_2$, and for $i \in \{1, 2\}$ define:
\begin{align*}
    u_i := u(x(t_i), t_i), \:\:
    \Phi_i(\cdot) := \Phi(x(t_i), \cdot), \:\:
    \hat{\alpha}_i(\cdot) := \hat{\alpha}(\cdot, t_i).
\end{align*}
We have:
\begin{align*}
    \norm{ \hat{\alpha}_1 - \hat{\alpha}_2 }_{L_2(\Theta, \nu)} &=
    \left( \int_\Theta \norm{\hat{\alpha}(\theta, t_1) - \hat{\alpha}(\theta, t_2)}_2^2 d\nu(\theta) \right)^{1/2} \\
    &\leq \left( \int_\Theta \bignorm{\int_{t_1}^{t_2} \frac{\partial \hat{\alpha}}{\partial t}(\theta, t) dt }^2_2 d\nu(\theta) \right)^{1/2} \\
    &\leq \left( \int_\Theta \left( \int_{t_1}^{t_2} \bignorm{\frac{\partial \hat{\alpha}}{\partial t}(\theta, t)}_2 dt \right)^2 d\nu(\theta) \right)^{1/2} \\
    &\leq \gamma C^2 \left( \int_\Theta \sup_{\norm{x}_2 \leq C} \opnorm{\Phi(x, \theta)}^2 d\nu(\theta) \right)^{1/2} \abs{t_1 - t_2} \\
    &\leq \gamma C^2 C_\Phi \abs{t_1-t_2}.
\end{align*}
Next, let
$C'_\Phi$ be a finite constant such that
\begin{align*}
    \left(\int_\Theta \opnorm{\Phi(x, \theta) - \Phi(y, \theta)}^2 d\nu(\theta)\right)^{1/2} \leq C'_\Phi \norm{x-y}_2 \:\:\forall x, y \in B_2^n(C). 
\end{align*}
Then,
\begin{align*}
    \norm{u_1-u_2}_{2} &\leq \int_\Theta \norm{\Phi_1(\theta)\hat{\alpha}_1(\theta) - \Phi_2(\theta)\hat{\alpha}_2(\theta)}_2 d\nu(\theta) \\
    &\leq \int_\Theta \opnorm{\Phi_1(\theta) - \Phi_2(\theta)}\norm{\hat{\alpha}_1(\theta)}_2 d\nu(\theta) + \int_\Theta \opnorm{\Phi_2(\theta)} \norm{\hat{\alpha}_1(\theta) - \hat{\alpha}_2(\theta)}_2 d\nu(\theta) \\
    &\leq C \left( \int_\Theta \opnorm{\Phi(x(t_1), \theta) - \Phi(x(t_2),\theta)}^2 d\nu(\theta)\right)^{1/2} + C_\Phi \norm{\hat{\alpha}_1 - \hat{\alpha}_2}_{L_2(\Theta, \nu)} \\
    &\leq C C'_\Phi \norm{x(t_1) - x(t_2)}_2 + \gamma C^2 C^2_\Phi \abs{t_1-t_2} \\
    &\leq (C^2 C'_\Phi + \gamma C^2 C^2_\Phi) \abs{t_1-t_2}.
\end{align*}
This shows that $t \mapsto u(x(t), t)$ is uniformly Lipschitz.
To conclude, we argue that $t \mapsto u(x(t), t)$ is uniformly bounded:
\begin{align*}
    \norm{ u(x(t), t) }_2 &\leq \int_\Theta \opnorm{ \Phi(x(t), \theta) } \norm{\hat{\alpha}(t)}_2 d\nu(\theta) \\
    &\leq \sqrt{\int_\Theta \opnorm{\Phi(x(t), \theta)}^2 d\nu(\theta)} \norm{\hat{\alpha}(t)}_2 \\
    &\leq C_\Phi \norm{\hat{\alpha}(t)}_2.
\end{align*}
The proof of Theorem~\ref{thm:nonparametric_conv} 
shows that $\norm{\tilde{\alpha}(t)}_2$ is uniformly bounded,
and therefore so is $\norm{\hat{\alpha}(t)}_2$ by the triangle inequality.
\end{proof}

\subsection{Proof of Theorem~\ref{thm:imp_reg}}
\impreg*
\begin{proof}
From Theorem~\ref{thm:nonparametric_conv}, $u(\cdot, t) \in \calH$ for all $t \geq 0$.
Let $\bar{h}(\cdot) \in \calH$ be arbitrary. Then by 
Assumption~\ref{assmp:rf_kernel}
there exists $\bar{\alpha} \in L_2(\Theta, \nu)$ such that 
\begin{equation*}
    \bar{h}(x) = \int_\Theta \Phi(x, \theta)\bar{\alpha}(\theta)d\nu(\theta).
\end{equation*} 
Consider the Lyapunov-like function $V: \R_{\geq 0}\rightarrow \R$,
\begin{equation*}
    V(t) = \frac{1}{2}\norm{u(\cdot, t) - \bar{h}}_{\calH}^2 = \frac{1}{2}\norm{\hat{\alpha}(\cdot, t) - \bar{\alpha}}_{L_2(\Theta, \nu)}^2,
\end{equation*}
where $\hat{\alpha}(\cdot, t) \in L_2(\Theta, \nu)$ was defined in the proof of Theorem~\ref{thm:nonparametric_conv} by the partial differential equation 
\begin{equation*}
    \frac{\partial\hat{\alpha}}{\partial t}(\theta, t) = -\gamma \Phi(x(t), \theta)^\T g_e(x(t), t)^\T \nabla Q (e(t), t), \:\: \hat{\alpha}(\theta, 0) = 0.
\end{equation*}
Computing the time derivative of $V$,
\begin{align*}
    \dot{V}(t) = \bigip{\hat{\alpha}(\cdot, t) - \bar{\alpha}}{\frac{\partial \hat{\alpha}}{\partial t}(\cdot, t)}_{L_2(\Theta, \nu)}.
\end{align*}
Integrating both sides of the above from $0$ to $t$,
\begin{equation*}
    \frac{1}{2}\norm{\hat{\alpha}(\cdot, t) - \bar{\alpha}}_{L_2(\Theta, \nu)}^2 = \frac{1}{2}\norm{\bar{\alpha}}_{L_2(\Theta, \nu)}^2 + \int_0^t \bigip{\hat{\alpha}(\cdot, \tau)-\bar{\alpha}}{\frac{\partial \hat{\alpha}}{\partial t}(\cdot, \tau)}_{L_2(\Theta, \nu)}d\tau,
\end{equation*}
Define $\hat{\alpha}_{\infty}(\theta)$ to be the density such that $\lim_{t\rightarrow\infty} u(\cdot, t) = \int_{\Theta} \Phi(\cdot, \theta)\hat{\alpha}_{\infty}(\theta)d\nu(\theta)$. 
Taking the limit as $t\rightarrow\infty$ of both sides,
\begin{align}
    \frac{1}{2}\norm{\hat{\alpha}_{\infty} - \bar{\alpha}}_{L_2(\Theta, \nu)}^2 &= \lim_{t\rightarrow\infty}\frac{1}{2}\norm{\hat{\alpha}(\cdot, t) - \bar{\alpha}}_{L_2(\Theta, \nu)}^2 \nonumber \\
    &= \frac{1}{2}\norm{\bar{\alpha}}_{L_2(\Theta, \nu)}^2 + \int_0^\infty \bigip{\hat{\alpha}(\cdot, \tau)-\bar{\alpha}}{\frac{\partial \hat{\alpha}}{\partial t}(\cdot, \tau)}_{L_2(\Theta, \nu)}d\tau.     \label{eqn:imp_reg_intermediate}
\end{align}
Now take $\bar{h}(\cdot) \in \mathcal{A}$. Observe that, by definition of $\calA$,
for any $\tau \geq 0$,
\begin{align*}
    \bigip{\bar{\alpha}}{ \frac{\partial \hat{\alpha}}{\partial t}(\cdot, \tau)}_{L_2(\Theta, \nu)}  &= -\gamma\int_{\Theta}\bar{\alpha}(\theta)^\T \Phi(x(\tau), \theta)^\T g_e(x(\tau), t)^\T \nabla Q(e(\tau), \tau)\\
    &= -\gamma \bar{h}(x(\tau))^\T g_e(x(\tau), \tau)^\T \nabla Q(e(\tau), \tau)\\
    &= -\gamma h(x(\tau))^\T g_e(x(\tau), \tau)^\T \nabla Q(e(\tau), \tau).
\end{align*}
Hence, \eqref{eqn:imp_reg_intermediate} may be re-written,
\begin{align*}
    \frac{1}{2}\norm{\hat{\alpha}_\infty - \bar{\alpha}}_{L_2(\Theta, \nu)}^2 &= \frac{1}{2}\norm{\bar{\alpha}}_{L_2(\Theta, \nu)}^2 + \int_0^\infty \bigip{\hat{\alpha}(\cdot, \tau)}{\frac{\partial \hat{\alpha}}{\partial t}(\cdot, \tau)}_{L_2(\Theta, \nu)}d\tau\\
    &\qquad + \gamma\int_0^\infty h(x(\tau))^\T g(x(\tau), \tau)^\T \nabla Q (x(\tau), \tau)d\tau,
\end{align*}
which has eliminated the dependence of the right-hand side on $\bar{\alpha}$ except for in the first term.
Let $\bar{\calA} := \{ \bar{\alpha} \in L_2(\Theta, \nu) : h(\cdot) = \int_\Theta \Phi(\cdot, \theta) \bar{\alpha}(\theta) d\nu \in \calA \}$.
Since $\hat{\alpha}_\infty \in \bar{\calA}$ by assumption, taking
the arg min over both sides of the above equation,
\begin{align*}
    \hat{\alpha}_\infty \in \argmin_{\bar{\alpha} \in \bar{\calA}} \norm{\bar{\alpha}}_{L_2(\Theta, \nu)}.
\end{align*}
The claim now follows by the correspondence
between $L_2(\Theta,\nu)$ and $\calH$.
\end{proof}

\section{Omitted proofs for Section~\ref{sec:rf}}

\subsection{Proof of Proposition~\ref{prop:uniform_approx}}
\uniform*
\begin{proof}
We first define a truncated target function $h_\eta(x)$ and its truncated approximation $\hat{h}_\eta(x; \{\theta_i\}_{i=1}^{K})$
\begin{align*}
    h_\eta(x) &:= \int_{\Theta} \Phi_\eta(x, \theta) \alpha(\theta) \: d\nu(\theta), \\
    \hat{h}_\eta(x; \{\theta_i\}_{i=1}^{K}) &:= \frac{1}{K} \sum_{i=1}^{K} \Phi_\eta(x, \theta_i) \alpha(\theta_i).
\end{align*}
Clearly, for each $x \in \R^n$,
\begin{align*}
    \E_{\{\theta_i\}_{i=1}^{K}} \hat{h}_\eta(x; \{\theta_i\}_{i=1}^{K}) &= h_\eta(x).    
\end{align*}
Now, consider two sets $\{\theta_i\} \subseteq \Theta$ and $\{\tilde{\theta}_i\} \subseteq \Theta$ that differ in only one index $i$. Observe that
\begin{align*}
    \norm{ \hat{h}_\eta(\cdot; \{\theta_i\}_{i=1}^{K}) - \hat{h}_\eta(\cdot; \{\tilde{\theta}_i\}_{i=1}^{K})}_\infty \leq \frac{2 B_\Phi(\eta) B_h}{K}.
\end{align*}
Hence, by McDiarmid's inequality, with probability at least $1-\delta/2$,
\begin{align*}
    \norm{\hat{h}_\eta(\cdot; \{\theta_i\}_{i=1}^{K}) - h_\eta(\cdot)}_\infty &\leq \E \norm{\hat{h}_\eta(\cdot; \{\theta_i\}_{i=1}^{K}) - h_\eta(\cdot)}_\infty +  \sqrt{2} B_\Phi(\eta) B_h \sqrt{\frac{\log(2/\delta)}{K}} \\
    &\leq \frac{2}{K} \E \bignorm{\sum_{i=1}^{K} \varepsilon_i \Phi_\eta(\cdot; \theta_i) \alpha_i}_\infty + \sqrt{2} B_\Phi(\eta) B_h \sqrt{\frac{\log(2/\delta)}{K}},
\end{align*}
where the last inequality follows by a standard symmetrization argument.
Define the event $\calE$ as
\begin{align*}
    \calE := \left\{ \max_{i=1, ..., K} \sup_{x \in X} \opnorm{\Phi(x, \theta_i)} \leq B_\Phi(\eta) \right\}.
\end{align*}
By our assumption on $B_\Phi$ and a union bound, we have that $\Pr(\calE^c) \leq \delta/2$.
Furthermore, $\Phi(\cdot, \cdot)$ and $\Phi_\eta(\cdot, \cdot)$ agree on $\calE$ by definition, so that
\begin{align*}
    \ind\{\calE\}\bignorm{  \frac{1}{K} \sum_{i=1}^{K} \Phi(\cdot, \theta_i) \alpha_i - h(\cdot) }_\infty &=\ind\{\calE\}\bignorm{\frac{1}{K} \sum_{i=1}^{K} \Phi_\eta(\cdot, \theta_i) \alpha_i - h(\cdot) }_\infty \\
    &\leq \bignorm{\frac{1}{K} \sum_{i=1}^{K} \Phi_\eta(\cdot, \theta_i) \alpha_i - h_\eta(\cdot) }_\infty + \norm{ h(\cdot) - h_\eta(\cdot) }_\infty.
\end{align*}
We now focus on bounding the term on the right-hand side.
We write
\begin{align*}
    h(x) - h_\eta(x) = \int_{\Theta} \ind\{ \opnorm{\Phi(x, \theta)} > B_\Phi(\eta) \} \Phi(x, \theta) \alpha(\theta) d\nu(\theta).
\end{align*}
This implies the estimate
\begin{align*}
    \norm{h(\cdot) - h_\eta(\cdot)}_\infty &\leq \sup_{x \in X} B_h \E_{\theta \sim \nu} \ind\{ \opnorm{\Phi(x, \theta)} > B_\Phi(\eta) \} \opnorm{\Phi(x, \theta)} \\
    &\leq B_h \sqrt{\eta} \sqrt{ \sup_{x \in X} \E\opnorm{\Phi(x, \theta)}^2 }\\
    &= B_h \sqrt{\frac{\delta\sup_{x \in X} \E\opnorm{\Phi(x, \theta)}^2 }{2K}}.
\end{align*}
The claim now follows by a union bound.
\end{proof}

\subsection{Proof of Proposition~\ref{prop:rademacher_bound}}
To state the proof of the proposition, we will require the following useful result.
\begin{mylemma}[\cite{maurer16vectorcontraction}, Corollary 4]
\label{lem:mauer}
Let $\calX$ be any set, let $(x_1, \hdots, x_n)\in\calX^n$, let $\calF$ be a class of functions $f: \calX \rightarrow \ell_2$, and let $h_i : \ell_2 \rightarrow \R$ have Lipschitz constant $L$. Then,
\begin{equation*}
\E \sup_{f\in\calF}\sum_i \varepsilon_i h_i(f(x_i)) \leq \sqrt{2}L\E \sup_{f\in\calF} \sum_{i,k}\varepsilon_{i, k}f_k(x_i)
\end{equation*}
where the $\varepsilon_{ik}$ are i.i.d.\ Rademacher random variables and $f_k(x_i)$ is the $k^\text{th}$ component of $f(x_i)$.
\end{mylemma}
\noindent We may now proceed with the proof.
\rad*
\begin{proof}
Put $\alpha_i = \alpha(\theta_i)$ and $M_{\eta,i} := M_\eta(w_i)$. We write, by definition of the $\norm{\cdot}_\infty$-norm and duality,
\begin{align*}
    \E \bignorm{ \sum_{i=1}^{K} \varepsilon_i \Phi_\eta(x; \theta_i) \alpha_i }_\infty &= \E \sup_{x \in X} \sup_{\psi \in \mathbb{S}^{d_1-1}} \sum_{i=1}^{K} \varepsilon_i \psi^\T M_{\eta,i} \alpha_i \phi(w_i^\T x + b_i).
\end{align*}
Towards applying Lemma~\ref{lem:mauer}, for a tuple $(x, \psi) \in X \times \mathbb{S}^{d_1-1}$,
define 
\begin{equation*}
    f_{x,\psi}(w, b) := \begin{pmatrix} w^\T x + b \\ \psi \end{pmatrix}.
\end{equation*}
Next, define $h_i : \R \times \mathbb{S}^{d_1-1} \rightarrow \R$ as 
\begin{equation*}
    h_i(v_1, v_2) := v_2^\T M_{\eta,i} \alpha_i \phi(v_1).
\end{equation*}
We need to show that $h_i$ is Lipschitz continuous. Let $v = (v_1, v_2)$, $w = (w_1, w_2)$, and observe that
\begin{align*}
    \abs{h_i(v_1, v_2) - h_i(w_1, w_2)} \leq \sqrt{2} B_h B_\Phi(\eta) \norm{v-w}_2,
\end{align*}
where we have applied the triangle inequality. Now, let $\{\xi_i\}_{i=1}^{K} \subseteq \{\pm 1\}$
and $\{\zeta_i\}_{i=1}^{K} \subseteq \{\pm 1\}^{d_1}$ be independent random vectors with i.i.d.\ Rademacher random variables as entries. Then, by Lemma~\ref{lem:mauer},
\begin{align*}
    \E \sup_{x \in X} \sup_{\psi \in \mathbb{S}^{d_1-1}} \sum_{i=1}^{K} \varepsilon_i  \psi^\T M_{\eta,i} \alpha_i \phi(w_i^\T x + b_i) &= \E \sup_{x \in X, \psi \in \mathbb{S}^{d_1-1}} \sum_{i=1}^{K} \varepsilon_i h_i(F_{x, \psi}(w_i, b_i)), \\
    &\leq 2 B_h B_\Phi(\eta) \E \sup_{x \in X, \psi \in \mathbb{S}^{d_1-1}} \sum_{i=1}^{K} \bigip{\begin{pmatrix}\xi_i\\\zeta_i\end{pmatrix}}{\begin{pmatrix} w_i^\T x + b_i \\ \psi \end{pmatrix}}, \\
    &= 2 B_h B_\Phi(\eta) \left[ \E \sup_{x \in X} \sum_{i=1}^{K} \xi_iw_i^\T x  + \E \sup_{\psi \in \mathbb{S}^{q-1}} \sum_{i=1}^{K} \zeta_i^\T \psi \right], \\
    &\leq 2 B_h B_\Phi(\eta) \left[ B_X \E \bignorm{ \sum_{i=1}^{K} \xi_iw_i}_2 + \E \bignorm{\sum_{k=1}^{K} \zeta_i}_2 \right], \\
    &\leq 2 \sqrt{K} B_h B_\Phi(\eta) \left[ B_X \sqrt{\E \norm{w_1}_2^2} + \sqrt{d_1} \right].
\end{align*}
This completes the proof.
\end{proof}

\section{Omitted proofs for Section~\ref{sec:p_results}}
\label{app:p_results}

\subsection{Details of Example~\ref{prop:s_delta_deadzone}}

First, for any $\delta > 0$, we define the function
\begin{align*}
    \bar{s}_\delta(x) := \frac{s_\delta(x)}{x}.
\end{align*}
\begin{mylemma}
\label{lem:sbar_delta_sqrt_lipschitz}
For any $\delta > 0$, the function $x \mapsto \bar{s}_\delta(\sqrt{x})$ is $\frac{1}{2\delta^2}$-Lipschitz on $\R_{\geq 0}$.
\end{mylemma}
\begin{proof}
Fix $x, y \in \R_{\geq 0}$.
Without loss of generality, suppose that $x \leq y$ (otherwise, we may flip the roles of $x$ and $y$).
If $y \leq \delta^2$, then
$\bar{s}_\delta(\sqrt{x}) = \bar{s}_\delta(\sqrt{y}) = 0$,
in which case the claim is trivial.

Now suppose that $x \geq \delta^2$.
On $[\delta, \infty)$,
we have that $\bar{s}_\delta$ coincides with
$x \mapsto 1-\delta/x$, and hence
\begin{align*}
    \bar{s}_\delta(\sqrt{x}) - \bar{s}_\delta(\sqrt{y}) &= 1 - \frac{\delta}{\sqrt{x}} - \left( 1 - \frac{\delta}{\sqrt{y}}\right) 
    = \delta \left( \frac{1}{\sqrt{y}} - \frac{1}{\sqrt{x}} \right).
\end{align*}
The function $x \mapsto 1/\sqrt{x}$ is $\frac{1}{2\delta^3}$-Lipschitz on $[\delta^2, \infty)$, and hence
\begin{align*}
    \abs{\bar{s}_\delta(\sqrt{x}) - \bar{s}_\delta(\sqrt{y})} \leq \frac{1}{2\delta^2} \abs{x - y}.
\end{align*}
Finally, we suppose that
$x \leq \delta^2 \leq y$.
By concavity of the square root on $\R_{\geq 0}$,
\begin{align*}
  \sqrt{y} \leq \delta + \frac{1}{2\delta}(y - \delta^2).
\end{align*}
Therefore,
\begin{align*}
    \abs{\bar{s}_\delta(\sqrt{x}) - \bar{s}_\delta(\sqrt{y})} = \abs{\bar{s}_\delta(\sqrt{y})} &= 1 - \frac{\delta}{\sqrt{y}} 
    = \frac{\sqrt{y}-\delta}{\sqrt{y}} 
    \leq \frac{\sqrt{y} - \delta}{\delta} 
    \leq \frac{y - \delta^2}{2\delta^2} 
    \leq \frac{y - x}{2\delta^2} 
    = \frac{\abs{y-x}}{2\delta^2}.
\end{align*}
\end{proof}

\sdeltadeadzone*
\begin{proof}
It is straightforward to check that $\frac{d}{dx} s^2_{\sqrt{\Delta}}(\sqrt{x}) = \bar{s}_{\sqrt{\Delta}}(\sqrt{x})$.
Conditions (i) and (ii) are immediately satisfied. To check condition (iii), observe that by Lemma~\ref{lem:sbar_delta_sqrt_lipschitz}, 
$\bar{s}_{\sqrt{\Delta}}(\sqrt{x})$ is $1/(2\Delta)$-Lipschitz.
Finally, $\bar{s}_{\sqrt{\Delta}}(\sqrt{x}) \leq 1$ for all $x \geq 0$.
\end{proof}

\subsection{Details of Example~\ref{prop:s_delta_gamma_deadzone}}

\sdeltagammadeadzone*
\begin{proof}
It is easy to check that the derivative $s'_{\Delta,\gamma}$ exists, is continuous,
and is given by:
\begin{align*}
    s'_{\Delta,\gamma}(x) = \begin{cases}
    0 &\text{if } x \leq \Delta, \\
    \frac{x-\Delta}{2\gamma} &\text{if } x \in (\Delta, \Delta+2\gamma), \\
    1 &\text{if } x \geq \Delta+2\gamma.
    \end{cases}
\end{align*}
It is also easy to check that this derivative is $\frac{1}{2\gamma}$-Lipschitz
and bounded by $1$.
\end{proof}

\subsection{Proof of Theorem~\ref{thm:ac_finite_approx}}

\begin{myprop}
\label{prop:ode_rhs_lipschitz}
Fix any $\Delta > 0$.
Let $\sigma_\Delta$ be $\Delta$-admissible.
Define
$F : \R^n \times \R^s \times O_p \times O_m \times \R_{\geq 0} \rightarrow \R^n \times \R^s \times \R^p \times \R^m$
as:
\begin{align*}
    F(x, e, \hat{\alpha}_p, \hat{\alpha}_m, t) &:= \begin{bmatrix}
    f(x, t) + g(x, t)(Y(x, t) \tilde{\alpha}_p + \Psi(x) \hat{\alpha}_m - h(x)), \\
    f_e(e, t) + g_e(x, t)(Y(x, t) \tilde{\alpha}_p + \Psi(x) \hat{\alpha}_m - h(x)), \\
    -\sigma'_\Delta(Q(e, t)) [\nabla^2 \psi_p(\hat{\alpha}_p)]^{-1} Y(x, t)^\T g_e(e, t)^\T \nabla Q(e, t), \\
    -\sigma'_\Delta(Q(e, t)) [\nabla^2 \psi_m(\hat{\alpha}_m)]^{-1} \Psi(x)^\T g_e(e, t)^\T \nabla Q(e, t).
    \end{bmatrix}.
\end{align*}
The function $F(x, e, \hat{\alpha}_p, \hat{\alpha}_m, t)$ is locally
Lipschitz in $(x, e, \hat{\alpha}_p, \hat{\alpha}_m)$.
\end{myprop}
\begin{proof}
The functions
$f, f_e, g, g_e, h, Y, \Psi, \nabla Q, B_h$, and $\sigma'_\Delta$ are all
locally Lipschitz and locally bounded
by assumption.
As long as we can check that both
$\zeta_1(e, t) := \sigma'_\Delta(Q(e, t))$ and
$\zeta_{2,\ell}(\hat{\alpha}) := [\nabla^2 \psi_\ell(\hat{\alpha})]^{-1}$
for $\ell \in \{p, m\}$
are locally Lipschitz and locally bounded, then
the result follows via repeated applications
of the sum and product composition rules~(Proposition~\ref{prop:local_composition}).

We now verify that $\zeta_1$ is locally Lipschitz
and locally bounded.
Since $\nabla Q(e, t)$ is locally bounded, this means that
$Q(e, t)$ is locally Lipschitz.
Furthermore, since $0 \leq Q(e, t) \leq \mu_2(\norm{e}_2)$,
it is clear that $Q(e, t)$ is locally bounded.
Next, $\sigma'_\Delta$ is locally Lipschitz by admissibility.
Since $\sigma'_\Delta$ does not depend on time, then it is
also locally bounded.
This shows that $\zeta_1$ is locally Lipschitz and bounded,
since it is the composition of two locally Lipschitz and bounded functions.

For $\zeta_{2,\ell}$, we first observe that,
since $\psi_\ell$ is strongly convex with respect to a norm $\norm{\cdot}$ on $O_\ell$, there exists a $c > 0$ such that $\nabla^2 \psi_\ell(\hat{\alpha}) \succcurlyeq c I$ for all $\hat{\alpha} \in O_\ell$.
Next, for any invertible square matrices $A, B$, we have the algebraic identity
$A^{-1} - B^{-1} = A^{-1} (B-A) B^{-1}$.
Therefore, for any two $\hat{\alpha}_1, \hat{\alpha}_2$,
\begin{align*}
    \opnorm{ [\nabla^2 \psi_\ell(\hat{\alpha}_1)]^{-1} - [\nabla^2 \psi_\ell(\hat{\alpha}_2)]^{-1} } \leq c^{-2} \opnorm{ \nabla^2 \psi(\hat{\alpha}_1) - \nabla^2 \psi(\hat{\alpha}_2) }.
\end{align*}
Because the potential $\psi_\ell$ has locally Lipschitz Hessians, this shows that $\zeta_{2,\ell}$ is locally Lipschitz. Since $\zeta_{2,\ell}$ does not depend on time, it is also
locally bounded.
\end{proof}

\begin{myprop}
\label{prop:bregman_lower_bound}
Let $O \subseteq \R^\ell$ be an open convex set, and 
let $\psi : O \rightarrow \R$ be a strongly convex potential with respect to some norm $\norm{\cdot}$ on $O$.
Then we have that
\begin{align*}
    \inf_{\substack{\alpha,\hat{\alpha} \in O, \\\alpha \neq \hat{\alpha}}} \frac{\bregd{\alpha}{\hat{\alpha}}}{\norm{\alpha - \hat{\alpha}}_2^2} > 0.
\end{align*}
\end{myprop}
\begin{proof}
Because all norms on $\R^\ell$ are equivalent, strong convexity
on $O$ there exists a $c > 0$ such that
$\nabla^2 \psi(a) \succcurlyeq c I$ for all $a \in O$.
By Taylor's theorem, for any arbitrary $\alpha, \hat{\alpha}$,
\begin{align*}
    \bregd{\alpha}{\hat{\alpha}} &= \frac{1}{2} (\alpha-\hat{\alpha})^\T \left[ \int_0^1 \nabla^2 \psi((1-t)\hat{\alpha}+\alpha t) dt \right] (\alpha - \hat{\alpha}) 
    \geq \frac{c}{2} \norm{\alpha-\hat{\alpha}}^2_2.
\end{align*}
The claim now follows.
\end{proof}

\acfiniteapprox*
\begin{proof}
By Proposition~\ref{prop:ode_rhs_lipschitz},
the right-hand side of the dynamical system on $(x, e, \hat{\alpha}_p, \hat{\alpha}_m)$ is locally Lipschitz, and
therefore there exists a maximal time $T_{\max} > 0$
such that there exists a unique $C^1$ curve
$t \mapsto (x(t), e(t), \hat{\alpha}_p(t), \hat{\alpha}_m(t))$ 
that satisfies the dynamics on $[0, T_{\max})$.

We now define the candidate Lyapunov function $V : [0, T_{\max}) \rightarrow \R_{\geq 0}$ as
\begin{align*}
    V(t) = \sigma_\Delta(Q(e(t), t)) + \bregdp{\alpha_p}{\hat{\alpha}_p} + \bregdm{\alpha_m}{\hat{\alpha}_m}.
\end{align*}
where $\alpha_m$ is the minimizing $\alpha_m$ in the definition
$B_{\text{approx}}$\footnote{Such a minimizing
$\alpha_m$ exists since the function $\alpha_m \mapsto \sup_{\norm{x}_2 \leq B_x} \norm{\Psi(x) \alpha_m - h(x)}_2$
is continuous and the set
$\{ \bregdm{\alpha_m}{\alpha_{m,0}} \leq B_{\alpha_m} \}$ is closed.
}. 
Taking the time derivative of $V$ and suppressing dependence on time,
\begin{align*}
    \frac{d}{dt} V(t) &= \sigma'_\Delta(Q)\left(\ip{\nabla Q }{ f_e + g_e (Y\tilde{\alpha}_p + \Psi \alpha_m - h) } + \frac{\partial Q}{\partial t} \right) + \bigip{\frac{d}{dt} \nabla \psi_p(\hat{\alpha}_p)}{\tilde{\alpha}_p} + \bigip{\frac{d}{dt} \nabla \psi_m(\hat{\alpha}_m)}{\tilde{\alpha}_m} \\
    &\leq \sigma'_\Delta(Q)\left( - \rho(\norm{e}_2) + \ip{\nabla Q}{ g_e(Y\tilde{\alpha}_p + \Psi \alpha_m - h) }\right) + \bigip{\frac{d}{dt} \nabla \psi_p(\hat{\alpha}_p)}{\tilde{\alpha}_p} + \bigip{\frac{d}{dt} \nabla \psi_m(\hat{\alpha}_m)}{\tilde{\alpha}_m} \\
    &= -\sigma'_\Delta(Q)\rho(\norm{e}_2) + \sigma'_\Delta(Q)\bigip{g_e^\T \nabla Q}{  \Psi\hat{\alpha}_m - h} - \sigma'_\Delta(Q) \bigip{ g_e^\T \nabla Q}{\Psi\tilde{\alpha}_m} \\
    &= -\sigma'_\Delta(Q)\rho(\norm{e}_2) + \sigma'_\Delta(Q) \bigip{g_e^\T \nabla Q}{\Psi \alpha_m - h} \\
    &\leq -\sigma'_\Delta(Q)\rho(\norm{e}_2) + \sigma'_\Delta(Q) \norm{g_e^\T \nabla Q}_2 \norm{\Psi \alpha_m - h}_2.
\end{align*}
Because $\sigma_\Delta$ is a $\Delta$-admissible deadzone, 
$\sigma'_\Delta > 0$ only when $Q > \Delta$.
But since $Q(e, t) \leq \mu_2(\norm{e}_2)$, 
we have $\norm{e}_2 > \mu_2^{-1}(\Delta)$.
Therefore,
\begin{align}
    \frac{d}{dt} V(t) &\leq- \sigma'_\Delta(Q)\rho(\mu_2^{-1}(\Delta)) +  \sigma'_\Delta(Q) \norm{g_e^\T \nabla Q}_2 \norm{\Psi \alpha_m - h}_2. \label{eq:d_dt_V_one}
\end{align}
Let
$T_0$ be defined as
\begin{align*}
    T_0 := \sup\{ T \in [0, T_{\max}) \mid \norm{e(t)}_2 \leq R \:\: \forall t \in [0, T]\}.
\end{align*}
Note that since 
\begin{align*}
    \norm{e(0)}_2 \leq \mu_1^{-1}(V(0)) \leq \mu_1^{-1}\left( Q(e(0), 0) + B_{\alpha_p} + B_{\alpha_m} \right) < R,
\end{align*}
$T_0$ is well-defined.
Now, by means of contradiction,
suppose $T_0 < T_{\max}$.
For every $t \in [0, T_0]$, by \eqref{eq:small_error_implies_small_state}, we have that
$\norm{x(t) - x_d(t)}_2 \leq C_e R$,
and hence $\norm{x(t)}_2 \leq C_e R + B_d = B_x$.
Hence, by the definition of $B_{g_e}$ and $B_{\nabla Q}$,
from \eqref{eq:d_dt_V_one}
and the requirement that
$\Delta \geq \mu_2(\rho^{-1}(2 B_{g_e} B_{\nabla Q} B_{\mathrm{approx}}))$, for every $t \in [0, T_0]$,
\begin{align*}
    \frac{d}{dt} V(t) &\leq - \sigma'_\Delta(Q)\rho(\mu_2^{-1}(\Delta)) + \sigma'_\Delta(Q) B_{g_e} B_{\nabla Q} B_{\mathrm{approx}} \\
    &\leq - \sigma'_\Delta(Q)\rho(\mu_2^{-1}(\Delta))/2.
\end{align*}

Hence, $V(t) \leq V(0)$ for all $t \in [0, T_0]$.
On the other hand, since $T_0$ is maximal,
we must have that $\norm{e(T_0)}_2 = R$, 
otherwise, if $\norm{e(T_0)}_2 < R$, 
by continuity of the solution $e(t)$
on $[0, T_{\max})$,
there would exist a $\delta > 0$
such that for all $t \in [0, T_0+\delta]$,
we have $\norm{e(t)}_2 \leq R$.
This means then that,
\begin{align*}
    V(0) \geq V(T_0) \geq \mu_1(\norm{e(T)}_2) = \mu_1(R) > \mu_1(\mu_1^{-1}(V(0))) = V(0),
\end{align*}
a contradiction. Hence $T_0 = T_{\max}$.

Now we argue that $T_{\max}$ cannot be finite.
Suppose towards a contradiction that $T_{\max}$ is finite.
We already have 
$\max_{t \in [0, T_{\max})} \norm{e(t)}_2 \leq R$.
This implies that $\norm{x(t)}_2 \leq C_e R + B_d = B_x$
for $t \in [0, T_{\max})$.
Finally, 
since $V(t) \leq V(0)$ on all $t \in [0, T_{\max})$,
this shows that both $\norm{\hat{\alpha}_p(t)}_2$
and $\norm{\hat{\alpha}_m(t)}_2$ are uniformly bounded
for all $t \in [0, T_{\max})$ via Proposition~\ref{prop:bregman_lower_bound}.
This contradicts the maximality of $T_{\max}$, showing
that $T_{\max} = \infty$.

To continue the proof, 
we integrate the inequality
$\frac{d}{dt} V(t) \leq - \sigma'_\Delta(Q) \rho(\mu_2^{-1}(\Delta))/2$ to conclude that
\begin{align*}
    \int_0^\infty \sigma'_\Delta(Q(e(t), t)) dt \leq \frac{2 V(0)}{\rho(\mu_2^{-1}(\Delta))}.
\end{align*}
We now argue that the integrand
$t \mapsto \sigma'_\Delta(Q(e(t), t))$ is uniformly continuous.

To do this, we 
will argue that (a) $t \mapsto Q(e(t), t)$ is uniformly bounded,
(b) $t \mapsto e(t)$ is uniformly Lipschitz,
and (c) $t \mapsto Q(e(t), t)$ is uniformly Lipschitz.
To see (a), we note that $Q(e(t), t) \leq V(t) \leq V(0)$.
To see (b), we note that:
\begin{align*}
    \norm{\dot{e}(t)}_2 \leq \norm{f_e(e(t), t)}_2 + \opnorm{g_e(x(t), t)} (\opnorm{Y(x(t), t)}\norm{\tilde{\alpha}_p(t)}_2 + \opnorm{\Psi(x(t))}\norm{\hat{\alpha}_m(t)}_2 + \norm{h(x(t))}_2).
\end{align*}
Since
$f_e$, $g_e$, $Y$, $\Psi$, and $h$ are locally bounded 
in the first argument uniformly in $t$,
and since
$\norm{\hat{\alpha}_p(t)}_2$ and
$\norm{\hat{\alpha}_m(t)}_2$ are uniformly bounded,
this shows that $\norm{\dot{e}(t)}_2$ is uniformly bounded, 
and hence $t \mapsto e(t)$ is uniformly Lipschitz.
To see (c), we observe that:
\begin{align*}
    &\abs{Q(e(s), s) - Q(e(t), t)} \\
    &\qquad\qquad\leq \abs{Q(e(s), s) - Q(e(s), t)} + \abs{Q(e(s), t) - Q(e(t), t)} \\
    &\qquad\qquad\leq \left[ \sup_{t \geq 0} \sup_{\norm{e}_2 \leq R} \bigabs{\frac{\partial Q}{\partial t}(e, t)}\right] \abs{s-t} + \left[ \sup_{t \geq 0} \sup_{\norm{e}_2 \leq R} \norm{\nabla Q(e, t)}_2 \right]\norm{e(s) - e(t)}_2.
\end{align*}
Since $\frac{\partial Q}{\partial t}$
and $\nabla Q$ are both locally bounded in $e$
uniformly in $t$, and since $t \mapsto e(t)$ is uniformly Lipschitz, 
we see that $t \mapsto Q(e(t), t)$ is also uniformly Lipschitz.

We now argue that $t \mapsto \sigma'_\Delta(Q(e(t), t))$ is uniformly continuous.
Since $\sigma'_\Delta$ is locally Lipschitz, 
it is uniformly Lipschitz on $[0, V(0)]$.
Therefore, $t \mapsto \sigma'_\Delta(Q(e(t), t))$ is the composition of two Lipschitz functions, and is hence Lipschitz
(and therefore uniformly continuous).

From this, we apply Barbalat's lemma to conclude that:
\begin{align*}
    \lim_{t \rightarrow \infty} \sigma'_\Delta(Q(e(t), t)) = 0.
\end{align*}
Since $\sigma_\Delta$ is a $\Delta$-admissible deadzone, this implies that
\begin{align*}
    \limsup_{t \rightarrow \infty} \norm{e(t)}_2 \leq \mu_1^{-1}(\Delta).
\end{align*}
\end{proof}

\subsection{Proof of Theorem~\ref{thm:dp_finite_approx}}

\newcommand{\fnom}{\bar{f}}

\dpfiniteapprox*
\begin{proof}
We proceed by reduction to Theorem~\ref{thm:ac_finite_approx}. Observe that we may write the predictor \eqref{eqn:dyn_predict} in the matched uncertainty form \eqref{eqn:gen_dyn} with $g(\hat{x}, t) = I$
\begin{equation*}
    \dot{\hat{x}} = f(\hat{x}, t) + k(\hat{x}, x(t)) + \left(\hat{f}(\hat{x}, \hat{\alpha}, t) - f(\hat{x}, t)\right).
\end{equation*}
This is an adaptive control problem with input $\hat{f}(\hat{x}, \hat{\alpha}, t)$ and desired trajectory $x(t)$. The ``nominal dynamics'' $\fnom(\hat{x}, t) := f(\hat{x}, t) + k(\hat{x}, x(t))$ is contracting 
at rate $\lambda$ in the metric $M$ by assumption, meaning that
\begin{align*}
    \frac{\partial \fnom}{\partial \hat{x}}(\hat{x}, t)^\T M(\hat{x}, t) + M(\hat{x}, t) \frac{\partial \fnom}{\partial \hat{x}}(\hat{x}, t) + \dot{M}(\hat{x}, t) \preccurlyeq - 2 \lambda M(\hat{x}, t) \:\:\forall \hat{x} \in \R^n, t \in \R_{\geq 0}.
\end{align*}
Let the error signal $e(t) := \hat{x}(t) - x(t)$.
The error dynamics are
\begin{align*}
    \dot{e} &= f(\hat{x}, t) - f(x(t), t) + k(\hat{x}, x(t)) + \left( \hat{f}(\hat{x}, \hat{\alpha}, t) - f(\hat{x}, t)\right) \\
    &= \fnom(\hat{x}, t) - f(x(t), t) + \left( \hat{f}(\hat{x}, \hat{\alpha}, t) - f(\hat{x}, t)\right).
\end{align*}
Hence we can define
\begin{align*}
    f_e(e, t) := \fnom(e + x(t), t) - f(x(t), t), \:\: g_e(x, t) := I. 
\end{align*}
We first check that $f_e(e, t)$ is locally Lipschitz and locally bounded uniformly in $t$ via Proposition~\ref{prop:local_composition}.
First we consider $(e, t) \mapsto f(e + x(t), t)$.
Write this map as the composition $f(\phi(e, t), t)$ with
$\phi(e, t) := e + x(t)$. Since the signal $x(t)$ is uniformly bounded,
it is clear that $\phi$ is both locally Lipschitz and locally bounded
uniformly in $t$. Since the outer function $f(x, t)$ is also locally
Lipschitz and locally bounded uniformly in $t$, the composition remains locally Lipschitz and locally bounded uniformly in $t$.
Next, since $k(\hat{x}, x)$ is locally Lipschitz in $\hat{x}$
and continuous,
and since $x(t)$ is uniformly bounded, the function
$(\hat{x}, t) \mapsto k(\hat{x}, x(t))$ is locally Lipschitz
and locally bounded uniformly in $t$.
By an identical composition argument, so is $(e, t) \mapsto k(e + x(t), x(t))$. Finally, the function $(e, t) \mapsto f(x(t), t)$ is trivially locally Lipschitz and locally bounded uniformly in $t$
since $x(t)$ is bounded. Therefore, $f_e$ is locally Lipschitz and locally bounded uniformly in $t$.

The Jacobian $\frac{\partial f_e}{\partial e}(e, t) = \frac{\partial \fnom}{\partial \hat{x}}(e + x(t), t)$,
which shows that $f_e(e, t)$ is contracting at rate $\lambda$ in the
metric $M_e(e, t) := M(e + x(t), t)$.
Furthermore, it is easy to check that
$e = 0$ is a particular solution to $\dot{e} = f_e(e, t)$,
as $k(x, x) = 0$ for all $x$.
Therefore, $f_e$ admits an exponentially stable Lyapunov function $Q(e, t) = E_{M_e(\cdot, t)}(e, 0)$ that satisfies:
\begin{align*}
    \ip{\nabla Q(e, t)}{f_e(e, t)} + \frac{\partial Q}{\partial t}(e, t) \leq - 2 \lambda Q(e, t) \:\: \forall e \in \R^n, t \in \R_{\geq 0}.
\end{align*}
Moreover, because $\mu I \preceq M_e(e, t) \preceq L I$,
\begin{equation*}
    \mu\norm{e}_2^2 \leq Q(e, t) \leq L\norm{e}_2^2 \:\: \forall e \in \R^n, t \in \R_{\geq 0}.
\end{equation*}
Now, observe that
\begin{equation*}
    \nabla Q(e, t) = M_e(e, t)\gamma_s(0; e + x(t), x(t), t),
\end{equation*}
so that $B_{\nabla Q} = \sup_{t\geq 0}\sup_{\norm{e}_2 \leq R}\norm{\nabla Q(e, t)}_2 \leq L B_\gamma$. 
Furthermore, by the boundedness of $M$ and
the assumption that
$(\hat{x}, t) \mapsto \norm{\gamma_s(0; \hat{x}, x(t), t)}_2$ is
locally bounded in $\hat{x}$ uniformly in $t$,
we have that $\nabla Q(e, t)$ is locally bounded in $e$ uniformly in $t$.
Similarly, since
\begin{align*}
    \frac{\partial Q}{\partial t}(e, t) = - \gamma_s(1; e + x(t), x(t), t)^\T M_e(e, t) f_e(e, t),
\end{align*}
by the boundedness of $M$, the assumption that
$(\hat{x}, t) \mapsto \norm{\gamma_s(1; \hat{x}, x(t), t)}_2$ 
is locally bounded
in $\hat{x}$ uniformly in $t$ (since geodesics have constant speed,
we have $\norm{\gamma_s(1; \hat{x}, x(t), t)}_2 = \norm{\gamma_s(0; \hat{x}, x(t), t)}_2$),
we have that $\frac{\partial Q}{\partial t}(e, t)$ is locally bounded in $\hat{x}$ uniformly in $t$. Hence, we can invoke Theorem~\ref{thm:ac_finite_approx}
with
\begin{align*}
    C_e = 1, \:\:
    B_d = B_x, \:\:
    B_x = B_{\hat{x}}, \:\:
    B_{g_e} = 1, \:\:
    B_{\nabla Q} = L B_\gamma, \\
    \rho\left(\norm{e}_2\right) = 2\lambda \mu \norm{e}_2^2, \:\: \mu_1\left(\norm{e}_2\right) = \mu\norm{e}_2^2, \:\:
    \mu_2\left(\norm{e}_2\right) = L\norm{e}_2^2.   
\end{align*}
The result now follows.
\end{proof}

\subsection{Proof of Theorem~\ref{thm:ac_approx_interp}}

\acapproxinterp*
\begin{proof}
From the proof of Theorem~\ref{thm:ac_finite_approx},
the solution $t \mapsto (x(t), e(t), \hat{\alpha}_p(t), \hat{\alpha}_m(t))$
exists for $t \geq 0$, is unique, and is continuously differentiable. Furthermore,
by Proposition~\ref{prop:bregman_lower_bound} we have the following
uniform estimates
\begin{align*}
    \sup_{t \geq 0} \norm{e(t)}_2 \leq R, \:\: \sup_{t \geq 0} \norm{x(t)}_2 \leq B_x, \:\: \sup_{t \geq 0} \norm{\hat{\alpha}_\ell(t)}_2 \leq \sqrt{\frac{B_{\alpha_\ell}}{c_\ell}} + \norm{\alpha_{\ell,0}}_2 +  \sqrt{\frac{V(0)}{c_\ell}}, \:\: \ell \in \{p, m\}.
\end{align*}
Here, $c_p$ (resp. $c_m$) is a constant depending only on the ambient dimension $p$ and $\psi_p$ (resp. $m$ and $\psi_m$).

Now, applying that
$f, g, f_e, g_e, Y, \Psi$, and $h$ are all locally bounded in their first arguments uniformly in $t$, 
that $\nabla^2 \psi_\ell$ is uniformly 
bounded from below for $\ell \in \{p, m\}$,
and the assumption that $\sigma'_\Delta$ is $B$-bounded, we conclude that 
$\dot{x}(t)$, $\dot{e}(t)$, $\dot{\alpha}_p(t)$,
and $\dot{\alpha}_m(t)$ are all uniformly bounded. Hence
$x(t)$, $e(t)$, $\alpha_p(t)$, and $\alpha_m(t)$ are uniformly Lipschitz
with Lipschitz constants that do not depend on $\Delta$ and $L$.

Next, the fact that $f_e$, $g_e$, $Y$, $\Psi$, and $h$ are locally Lipschitz and locally bounded in their first arguments
uniformly in $t$
implies that $\dot{e}(t)$ is uniformly Lipschitz, 
with a Lipschitz constant that depends affinely on $L$.
Therefore, by Proposition~\ref{prop:generalized_barbalat},
we have that:
\begin{align*}
    \limsup_{t \rightarrow \infty} \norm{\dot{e}(t)}_2 \leq C_1 \sqrt{\mu_1^{-1}(\Delta)(1+L)},
\end{align*}
for a constant $C_1$ that does not depend on $\Delta$ and $L$.
Now for any $t$,
\begin{align*}
    \norm{g_e(x(t), t)(u(x(t), t) - Y(x(t), t) \alpha_p - h(x(t)))}_2 &\leq \norm{\dot{e}(t)}_2 + \norm{f_e(e(t), t)}_2 \\
    &\leq \norm{\dot{e}(t)}_2 + C_2 \norm{e(t)}_2,
\end{align*}
where $C_2$ does not depend on $\Delta$ and $L$. Taking the $\limsup$ on both sides yields the claim.
\end{proof}


\end{document}